\pdfoutput=1

\batchmode
\AtBeginDocument{\scrollmode}

\newif\iffinalversion

\documentclass[reqno,12pt,english]{amsart}

\finalversiontrue

\usepackage[l2tabu, orthodox]{nag}
\usepackage[T1]{fontenc}
\usepackage[utf8]{inputenc}

\usepackage[english]{babel}
\usepackage{csquotes}

\iffinalversion
\usepackage{microtype}
\fi
\usepackage{geometry}
\usepackage{setspace}
\usepackage[stable]{footmisc}
\usepackage{lmodern}

\usepackage[lowtilde]{url}
\usepackage[unicode]{hyperref}
\usepackage{bookmark}

\usepackage[capitalise]{cleveref}

\crefformat{section}{#2\S#1#3}
\crefrangeformat{section}{#3\S\S#1#4--#5#2#6}
\crefmultiformat{section}{#2\S\S#1#3}{ and #2#1#3}{, #2#1#3}{ and~#2#1#3}
\crefformat{internal_passage}{#2\P#1#3}
\crefrangeformat{internal_passage}{#3\P\P#1#4--#5#2#6}
\crefmultiformat{internal_passage}{#2\P\P#1#3}{ and #2#1#3}{, #2#1#3}{ and~#2#1#3}

\crefmultiformat{situation}{Situations~#2#1#3}{ and #2#1#3}{, #2#1#3}{ and~#2#1#3}
\crefformat{sequence}{Seq.~(#2#1#3)}
\Crefname{lstlisting}{Listing}{Listings}

\PassOptionsToPackage{%
  backend=biber,hyperref,bibencoding=auto,
  style=alphabetic,citestyle=alphabetic,sorting=anyt,
  doi=true,url=true,isbn=false,
  maxbibnames=99,maxalphanames=4,giveninits=true
  }{biblatex}

\iffinalversion
\usepackage{biblatex}
\fi

\providecommand{\printbibliography}{ 
  \bibliographystyle{alpha}
  \bibliography{bibliography}
}
\providecommand{\addbibresource}[1]{}
\providecommand{\DeclareSourcemap}[1]{}
\providecommand{\DeclareFieldFormat}[2]{}

\providecommand{\renewbibmacro}[2]{}

\makeatletter\@ifpackageloaded{biblatex}{}{
  \newcounter{biburllcpenalty}
  \newcounter{biburlucpenalty}
  \newcounter{biburlnumpenalty}
   
}\makeatother

\DeclareSourcemap{\maps[datatype=bibtex]{
  \map[overwrite]{
    \step[fieldsource=doi, final]
    \step[fieldset=url, null]
    \step[fieldset=eprint, null]
  }
  \map[overwrite]{
    \step[fieldset=language, null]
  }
  \map[overwrite]{
    \pertype{book}
    \step[fieldset=pages, null]
  }
}}

\renewbibmacro{in:}{%
  \ifentrytype{article}{}{\printtext{\bibstring{in}\intitlepunct}}}

\DeclareFieldFormat{postnote}{#1}
\DeclareFieldFormat{multipostnote}{#1}

\usepackage{amsmath}
\usepackage{mathtools}
\usepackage{amsthm}
\usepackage{amssymb}

\usepackage{tikz-cd}

\usepackage{fixmath}

\usepackage{eucal}

\usepackage[inline]{enumitem}

\usepackage{color}

\usepackage{etoolbox}

\setlist[enumerate]{nosep,label=(\roman*),font=\normalfont,leftmargin=2.8em}
\AtBeginEnvironment{enumerate}{~}

\usepackage{chngcntr}

\usepackage{xcolor}

\usepackage{amsmath}
\usepackage{amsthm}
\usepackage{amssymb}
\usepackage{mathrsfs}
\usepackage{fixmath}
\usepackage{upgreek}
\usepackage{xparse}
\usepackage{mathtools}

\newcounter{theoremcounter}
\makeatletter
\numberwithin{theoremcounter}{subsection}
\let\c@theoremcounter\c@subsubsection
\makeatother

\numberwithin{equation}{subsection}

\theoremstyle{plain}
\newtheorem{intro-theorem}{Theorem}
\newtheorem{intro-conjecture}[intro-theorem]{Conjecture}
\newtheorem{intro-corollary}[intro-theorem]{Corollary}
\theoremstyle{remark}
\newtheorem{intro-question}[intro-theorem]{Question}
\crefname{intro-theorem}{Theorem}{Theorems}

\theoremstyle{plain}
\newtheorem{conjecture}[theoremcounter]{Conjecture}
\newtheorem{corollary}[theoremcounter]{Corollary}

\newtheorem{lemma}[theoremcounter]{Lemma}
\newtheorem{proposition}[theoremcounter]{Proposition}
\newtheorem{theorem}[theoremcounter]{Theorem}
\newtheorem*{theorem*}{Theorem}
\newtheorem*{conjecture*}{Conjecture}

\theoremstyle{remark}
\newtheorem{remark}[theoremcounter]{Remark}
\newtheorem{question}[theoremcounter]{Question}
\newtheorem*{remark*}{Remark}
\newtheorem*{question*}{Question}

\theoremstyle{definition}

\newtheorem{construction}[theoremcounter]{Construction}
\newtheorem{convention}[theoremcounter]{Convention}
\newtheorem{definition}[theoremcounter]{Definition}
\newtheorem{example}[theoremcounter]{Example}
\newtheorem{notation}[theoremcounter]{Notation}

\newtheorem{situation}[theoremcounter]{Situation}
\newtheorem*{notation*}{Notation}

\newenvironment{claim}{\par\textbf{Claim.}}{\par}

\patchcmd{\swappedhead}{(#3)}{#3}{}{}

\theoremstyle{definition}
\makeatletter
\newtheorem{internal_passage}[theoremcounter]{}
\newenvironment{passage}[1][]
  {\begin{internal_passage}%
  \if\relax\detokenize{#1}\relax\else
    \textbf{#1.}%
  \fi}
  {\end{internal_passage}}
\makeatother

\newcounter{char}
\loop\ifnum\value{char}<27
  \expandafter\edef\csname\Alph{char}bb\endcsname{\noexpand\mathbb{\Alph{char}}}
  \expandafter\edef\csname\Alph{char}ca\endcsname{\noexpand\mathcal{\Alph{char}}}
  \expandafter\edef\csname\Alph{char}sc\endcsname{\noexpand\mathscr{\Alph{char}}}
  \expandafter\edef\csname\Alph{char}fr\endcsname{\noexpand\mathfrak{\Alph{char}}}
  \expandafter\edef\csname\alph{char}fr\endcsname{\noexpand\mathfrak{\alph{char}}}
  \expandafter\edef\csname\Alph{char}bf\endcsname{\noexpand\mathbf{\Alph{char}}}
  \expandafter\edef\csname\alph{char}bf\endcsname{\noexpand\mathbf{\alph{char}}}
  \expandafter\edef\csname\Alph{char}rm\endcsname{\noexpand\mathrm{\Alph{char}}}
  \expandafter\edef\csname\alph{char}rm\endcsname{\noexpand\mathrm{\alph{char}}}
  \ifnum\value{char}=1\else
    \expandafter\edef\csname\Alph{char}\Alph{char}\endcsname{\noexpand\mathbb{\Alph{char}}}
  \fi
\addtocounter{char}{1}
\repeat

\let\oldAA\AA
\renewcommand*{\AA}{\ifmmode\mathbb{A}\else\oldAA\fi}

\usepackage{mathbbol}
\DeclareSymbolFontAlphabet{\mathbb}{AMSb}
\DeclareSymbolFontAlphabet{\mathbbl}{bbold}

\DeclareMathOperator{\Cobar}{Cobar}
\DeclareMathOperator{\conn}{conn}

\DeclareMathOperator{\ev}{ev}
\DeclareMathOperator{\fib}{fib}
\DeclareMathOperator{\free}{free}
\DeclareMathOperator{\frgt}{forg}

\DeclareMathOperator{\Ho}{H}
\DeclareMathOperator{\HoCat}{ho}
\DeclareMathOperator{\Loc}{L}
\DeclareMathOperator{\Map}{Map}
\DeclareMathOperator{\MMap}{\mathbb{M}ap} 
\DeclareMathOperator{\T}{T}
\DeclareMathOperator{\infNv}{\mathcal{N}}

\newcommand*{\Loop}{\operatorname{\Omega}\mathopen{}}
\newcommand*{\susp}{\operatorname{\Sigma}\mathopen{}}

\newcommand*{\quot}[2]{\left.#1 \right/ \mathord #2}

\newcommand*{\cocolon}{\nobreak\mskip6mu plus1mu\mathpunct{}\nonscript\mkern-\thinmuskip{:}\mskip2mu\relax}

\newcommand*{\aug}{\mathrm{aug}}

\newcommand*{\coaug}{\mathrm{coaug}}
\newcommand*{\f}{\mathrm{f}}
\newcommand*{\fin}{\mathrm{fin}}

\newcommand*{\ndp}{\mathrm{ndp}}
\newcommand*{\nut}{\mathrm{nu}}

\newcommand*{\rex}{\mathrm{rex}}
\newcommand*{\op}{\mathrm{op}}
\newcommand*{\st}{\mathrm{st}}

\newcommand*{\id}{\operatorname{id}}
\newcommand*{\pr}{\operatorname{pr}}
\newcommand*{\blank}{{-}}

\newcommand*{\Category}[1]{\operatorname{\mathbf{#1}}}
\DeclareMathOperator{\CatAlg}{\Category{Alg}}
\DeclareMathOperator{\CatCh}{\Category{Ch}}
\DeclareMathOperator{\CatFin}{\Category{Fin}}

\ExplSyntaxOn
\NewDocumentCommand{\mathcalUpper}{m}{
  \tl_set:Nn \l_tmpa_tl { #1 }
  \regex_replace_all:nnN { ([A-Z]+) } { \c{mathcal}\cB\{\1\cE\} } \l_tmpa_tl
  \tl_use:N \l_tmpa_tl
  }
\ExplSyntaxOff

\newcommand*{\InftyCategory}[1]{\operatorname{\mathcalUpper{#1}}}

\DeclareMathOperator{\infAlg}{\InftyCategory{Alg}}

\DeclareMathOperator{\infCAT}{\InftyCategory{CAT}}

\DeclareMathOperator{\infcoAlg}{\InftyCategory{coAlg}}
\DeclareMathOperator{\infcoExc}{\InftyCategory{coExc}}
\DeclareMathOperator{\infcoOpd}{\InftyCategory{coOpd}}
\DeclareMathOperator{\infcoSp}{\InftyCategory{coSp}}
\DeclareMathOperator{\infD}{\InftyCategory{D}}
\DeclareMathOperator{\infExc}{\InftyCategory{Exc}}
\DeclareMathOperator{\infFin}{\InftyCategory{Fin}}
\DeclareMathOperator{\infFun}{\InftyCategory{Fun}}
\DeclareMathOperator{\infL}{\InftyCategory{L}}
\DeclareMathOperator{\infLmod}{\InftyCategory{LMod}}
\DeclareMathOperator{\infMap}{\InftyCategory{Map}}
\DeclareMathOperator{\infMod}{\InftyCategory{Mod}}

\DeclareMathOperator{\infOpd}{\InftyCategory{Opd}}
\DeclareMathOperator{\infPrl}{\InftyCategory{Pr}^{\mathrm{L}}}
\DeclareMathOperator{\infPrr}{\InftyCategory{Pr}^{\mathrm{R}}}
\DeclareMathOperator{\infSp}{\InftyCategory{Sp}}
\DeclareMathOperator{\infSSeq}{\InftyCategory{SymSeq}}
\DeclareMathOperator{\infHType}{\InftyCategory{Ho}}
\DeclareMathOperator{\infSplx}{\mathrm{\Delta}}

\newcommand*{\Perm}{\Sfr}

\newcommand*{\Operad}[1]{\operatorname{\mathbf{#1}}}

\DeclareMathOperator{\opE}{\Operad{E}}
\DeclareMathOperator{\opLie}{\Operad{Lie}}

\newcommand*{\InftyOperad}[1]{\operatorname{\mathcalUpper{#1}}}

\DeclareMathOperator{\infopCom}{\InftyOperad{Com}}
\DeclareMathOperator{\infopLie}{\InftyOperad{Lie}}
\DeclareMathOperator{\infopE}{\InftyOperad{E}}
\DeclareMathOperator{\infopTriv}{\InftyOperad{Triv}}

\newcommand*{\augideal}{\operatorname{I}}
\newcommand*{\trivaug}{\operatorname{A}}

\newcommand*{\PC}{\mathrm{\Pi}}
\newcommand*{\pt}{\mathrm{pt}}  
\newcommand*{\usphere}{\mathrm{S}}

\newcommand*{\Mon}{\mathrm{M}} 
\newcommand*{\EM}{\mathrm{H}} 
 
\newcommand*{\MK}{\mathrm{K}} 

\newcommand*{\Tel}{\mathrm{T}} 

\newcommand*{\munit}{\mathbbl{1}}

\usepackage{xspace}
\usepackage{xpunctuate}

\newcommand*{\ie}{i.e.\@\xspace}
\newcommand*{\cf}{cf.\@\xspace}

\newcommand*{\loccit}{loc.~cit\xperiod}

\let\oldiff\iff
\renewcommand*{\iff}{\ifmmode\oldiff\else{if and only if}\xspace\fi}

\newcommand*{\htype}{homotopy type\@\xspace}
\newcommand*{\htypes}{homotopy types\@\xspace}

\newcommand*{\rom}[1]{\textrm{(\romannumeral #1)}}

\usepackage{etoolbox} 
\newcounter{itemnumcounter}[theoremcounter] 
\AtBeginEnvironment{proof}{\setcounter{itemnumcounter}{0}}
\newcommand*{\itemnum}{\refstepcounter{itemnumcounter}\textnormal{(\roman{itemnumcounter})}\xspace}

\DeclareMathOperator{\infModuli}{\InftyCategory{Moduli}}

\DeclareMathOperator{\rB}{B} 
\DeclareMathOperator{\rrB}{C} 
\DeclareMathOperator{\Bac}{Bar} 
\DeclareMathOperator{\BK}{\Phi} 
\DeclareMathOperator{\rBK}{\Theta} 
\DeclareMathOperator{\CE}{CE}

\DeclareMathOperator{\indc}{indec}
\DeclareMathOperator{\KD}{K} 
\DeclareMathOperator{\lBK}{\Theta}

\DeclareMathOperator{\Null}{P} 
\DeclareMathOperator{\Poly}{P} 
\DeclareMathOperator{\sfree}{\mathbb{F}ree}

\DeclareMathOperator{\triv}{triv}
\DeclareMathOperator{\UE}{U} 
\DeclareMathOperator{\rUE}{T} 
\DeclareMathOperator{\relP}{F}

\newcommand*{\opdsusp}{\sigma\mathopen{}}

\newcommand*{\contraction}{contraction\@\xspace}

\newcommand*{\less}{less\@\xspace}

\hypersetup{colorlinks=true,linktocpage}
\hypersetup{citecolor=magenta,linkcolor=cyan,urlcolor=magenta}
\hypersetup{breaklinks=true}

\title[Universal property of the Bousfield--Kuhn functor]{Universal property of the Bousfield--Kuhn functor}

\author{Yuqing Shi}
\date{\today}
\address{Max-Planck-Institute for Mathematics, Vivatsgasse 7, 53111 Bonn, Germany}
\email{yuqing.shi@mpim-bonn.mpg.de}

\hypersetup{
  pdfauthor={Yuqing Shi},
  pdftitle={Universal property of the Bousfield--Kuhn functor},
  pdfcreator=pdflatex
}

\begin{document}

\begin{abstract}
  We present a universal property of the Bousfield--Kuhn functor~$\operatorname{\Phi}_h$ of height~$h$, for every positive natural number~$h$.
  This result is achieved by proving that the costabilisation of the~$\infty$\nobreakdash-category of~$v_h$\nobreakdash-periodic homotopy types is equivalent to the~$\infty$\nobreakdash-category of~$\operatorname{T}(h)$-local spectra.
  A key component in our proofs is the spectral Lie algebra model for~$v_h$-periodic homotopy types~\cite{Heu21}:
  We associate the costabilisation of the~$\infty$\nobreakdash-category of spectral Lie algebras with the costabilisations of the~$\infty$\nobreakdash-category of non-unital~$\mathcal{E}_{n}$\nobreakdash-algebras, via our construction of higher enveloping algebras of spectral Lie~algebras. 
\end{abstract}

\maketitle

\tableofcontents


\numberwithin{equation}{section}

\section*{Introduction}

Quillen's work on rational homotopy theory demonstrates the important role of Lie algebraic models in understanding rational \htypes: It proves a functorial one-to-one correspondence up to homotopy between simply connected rational \htypes and Lie algebra objects in rational spectra, see~\cite{Qui69}.
Using this correspondence, we have a better understanding of the structural and computational properties of rational \htypes.
In the framework of stable chromatic homotopy theory, the~$\infty$\nobreakdash-category~$\infSp_{\QQ}$ of rational spectra is the begin of a sequence~$(\infSp_{\Tel(h)})_{h \geq 0}$ of localisations of the~$\infty$\nobreakdash-category~$\infSp_{(p)}$ of~$p$\nobreakdash-local spectra, for a fixed prime number~$p$.
Here~$\Tel(h)$ denotes the~$p$\nobreakdash-local telescope spectrum of type~$h$ and~$\infSp_{\Tel(h)}$ denotes the localisation of~$\infSp_{(p)}$ at the set of~$\Tel(h)$\nobreakdash-homology equivalences~(by convention we have~$\Tel(0) \coloneqq \EM\QQ$).
Moreover, the notion of~$v_h$\nobreakdash-periodic homotopy equivalences~(of spectra or of \htypes) generalises the notion of rational homotopy equivalence in chromatic height~$h \geq 1$.
Recently, Heuts extended Quillen's Lie algebra model for rational \htypes to~$v_h$\nobreakdash-periodic \htypes by establishing the following equivalence
\begin{equation}\label{eq:spectral-Lie-vh}
  \infAlg_{\infopLie}(\infSp_{\Tel(h)}) \simeq \infHType_{v_h}  
\end{equation}
between the~$\infty$\nobreakdash-category~$\infAlg_{\infopLie}(\infSp_{\Tel(h)})$ of spectral Lie algebra objects in~$\infSp_{\Tel(h)}$ and the~$\infty$\nobreakdash-category~$\infHType_{v_h}$ of~$v_h$\nobreakdash-periodic ($p$-local) \htypes, \ie the localisation of the~$\infty$\nobreakdash-category~$\infHType_{(p)}$ of~$p$\nobreakdash-local \htypes at the set of~$v_h$\nobreakdash-periodic equivalences, see~\cite{Heu21}.
Fix a prime number~$p$ and a positive natural number~$h$, the theme of this article is to use the aforementioned spectral Lie algebra model to study~$v_h$\nobreakdash-periodic \htypes.

One can associate the stable monochromatic layer~$\infSp_{\Tel(h)}$ with the unstable monochro\-matic layer~$\infHType_{v_h}$ in the following homological and homotopical ways.
Recall that the {adjunction}~$\susp^{\infty} \dashv \Loop^{\infty}$ exhibits the~$\infty$\nobreakdash-category~$\infSp$ of spectra as the stabilisation of the~$\infty$\nobreakdash-category~$\infHType_{\ast}$ of pointed \htypes.
The analogous statement holds in chromatic homotopy theory: The induced functor~$\susp_{v_h}^{\infty} \coloneqq \Loc_{\Tel(h)} \circ \susp^{\infty}$, \ie the~$\Tel(h)$\nobreakdash-localisation of~$\susp^{\infty}$, exhibits~$\infSp_{\Tel(h)}$ as the stabilisation of~$\infHType_{v_h}$, see~\cite[Proposition~3.19]{Heu21}.
The homotopical way relating~$\infHType_{v_h}$ and~$\infSp_{\Tel(h)}$ is through the \emph{Bousfield--Kuhn~functor}
\[
\BK_h \colon \infHType_{v_h} \to \infSp_{\Tel(h)},
\] 
which is constructed in~\cite{Bou01} and plays an essential role in periodic homotopy theory. 
See~\cite{Kuh04, Kuh07,BR20, Heu21} for an incomplete list of applications of~$\BK_h$.
We call it homotopical because the~$v_h$\nobreakdash-periodic homotopy group~$v_h^{-1}\pi_{\bullet}(X; V_h)$ of a \htype~$X \in \infHType_{v_h}$ with coefficient in a finite complex~$V_h$ is isomorphic to the stable homotopy group of the mapping spectrum~$\MMap(V_h, \BK_h(X))$.
By~\cite{Bou01} the functor~$\BK_h$ admits a left adjoint~$\lBK_h \colon \infSp_{\Tel(h)} \to \infHType_{v_h}$, which supplies the~$\infty$\nobreakdash-category~$\infHType_{v_h}$ with infinitely many non-trivial objects admitting infinite desuspension.
In other words, for every~$E \in \infSp_{\Tel(h)}$, there exists an infinite sequence~$(Y_n)_{n \in \NN}$ of objects in~$\infHType_{v_h}$ such that~$\lBK_{h}(E) \simeq \susp_{v_h}^n Y_n$ for every~$n \in \NN$, where~$\susp_{v_h}$ denotes the suspension endofunctor on~$\infHType_{v_h}$.
Therefore, the object~$\lBK_h(E)$ gives an object in the~$\infty$\nobreakdash-category
\[
\infcoSp\left(\infHType_{v_h}\right) \coloneqq \varprojlim_{n} \left(\cdots \xrightarrow{\susp_{v_h}} \infHType_{v_h}\xrightarrow{\susp_{v_h}} \infHType_{v_h} \xrightarrow{\susp_{v_h}} \infHType_{v_h}\right)
\]
of the costabilisation of~$\infHType_{v_h}$, where the inverse limit is taken in the~$\infty$\nobreakdash-category of small~$\infty$\nobreakdash-categories.
The first theorem of this paper implies that every~$v_h$\nobreakdash-periodic \htype admitting infinite desuspensions lies in the essential image of~$\lBK_h$.

\begin{intro-theorem}[\cref{thm:costab-Lie-Tn}]\label{thm:costab-vh-intro-2}
  Let~$h \geq 1$ be a natural number. 
  The Bousfield--Kuhn adjunction~$\lBK_{h} \dashv \BK_{h}$ exhibits the~$\infty$\nobreakdash-category~$\infSp_{\Tel(h)}$ as the costabilisation of~$\infHType_{v_{h}}$.
\end{intro-theorem}

Note that~\cref{thm:costab-vh-intro-2} applies exclusively to positive chromatic height: The {costabilisation} of the~$\infty$\nobreakdash-category~$\infHType_{\ast}$ of pointed \htypes and is trivial, since any object admitting infinite desuspensions in~$\infHType_{\ast}$ is weakly contractible for connectivity reasons.
By the same arguments, the costabilisation of the~$\infty$\nobreakdash-category~$\infHType_{\QQ}$ of rational \htypes is also trivial.
One reason for the non-triviality of the costabilisation of~$\infHType_{v_h}$ is the fact that~$v_h$\nobreakdash-periodic equivalences of \htypes is insensitive to connectivity: 
The canonical morphism~$\tau_{> n}(X) \to X$, given by taking the fibre of the Postnikov truncation~$X \to \tau_{\leq n}(X)$, is a~$v_h$\nobreakdash-periodic equivalence, for every~$n \in \NN$ and for every~$1 \leq h \in \NN$.
This observation is also reflected in the equivalence~$\infHType_{v_h} \simeq \infAlg_{\infopLie}(\infSp_{\Tel(h)})$, which doesn't have assumption on connectivity, as opposed to Quillen's Lie algebra model in rational homotopy theory.

As a consequence of~\cref{thm:costab-vh-intro-2}, we obtain the following universal property of the Bousfield--Kuhn~functor.
\begin{intro-theorem}[\cref{cor:universal-property-BK}]\label{thm:intro-thm-up-BK}
  Let~$h \geq 1$ be a natural number.
  For any presentable stable~$\infty$\nobreakdash-category~$\Dca$, composition with the Bousfield--Kuhn functor~$\BK_{h}$ induces an equivalence
  \[
  \infFun^{\mathrm{R}}(\infSp_{\Tel(h)}, \Dca) \xrightarrow{\sim} \infFun^{\mathrm{R}}\left(\infHType_{v_{h}}, \Dca\right),
  \]
  where~$\infFun^{\mathrm{R}}$ denotes the~$\infty$\nobreakdash-category of functors that are accessible and preserve small limits, \ie functors that admits a left adjoint. 
\end{intro-theorem}

We analyse the costabilisation of~$\infHType_{v_h}$ by studying the costabilisation of~$\infAlg_{\infopLie}(\infSp_{\Tel(h)})$, using the equivalences~\eqref{eq:spectral-Lie-vh}.
By the recent work Heuts and Land~\cite{HLEn} we know that the costabilisation of the~$\infty$\nobreakdash-category~$\infAlg_{\infopE_n}^{\nut}(\infSp_{\Tel(h)})$ of non-unital~$\infopE_n$\nobreakdash-algebras is equivalent to~$\infSp_{\Tel(h)}$.
In order to make use of this result, we relate spectral Lie algebras with~$\infopE_n$\nobreakdash-algebras via our construction of higher enveloping algebras: We construct the following commutative diagram
\begin{equation*} 
\begin{tikzcd}[row sep = huge, column sep = small]
  & \infAlg_{\infopLie}\left(\infSp_{\Tel(h)}\right) \arrow[rrrrd, "\UE_0"] \arrow[rrrd, "\UE_1"']  \arrow[rd, "\UE_{n-1}"'] \arrow[d, "\UE_n"']  \arrow[ld] &             &             &   \\
  \cdots \arrow[r] & \infAlg_{\infopE_{n}}^{\nut}\left(\infSp_{\Tel(h)}\right) \arrow[r, "\rB_n"']                                                & \infAlg_{\infopE_{n-1}}^{\nut}\left(\infSp_{\Tel(h)}\right) \arrow[r, "\rB_{n-1}"'] & \cdots \arrow[r] & \infAlg_{\infopE_{1}}^{\nut}\left(\infSp_{\Tel(h)}\right) \arrow[r, "\rB_{1}"'] & \infSp_{\Tel(h)}
\end{tikzcd}
\end{equation*}
in the~$\infty$\nobreakdash-category~$\infPrl$ of presentable~$\infty$\nobreakdash-categories and small-colimit-preserving functors~(see~\cref{sec:higher-enveloping-Th}).
By the universal property of the inverse limits, the above commutative diagram induces a small-colimit-preserving functor
\[
\UE_{\infty} \colon \infAlg_{\infopLie}\left(\infSp_{\Tel(h)}\right) \to \varprojlim \infAlg_{\infopE_n}^{\nut}\left(\infSp_{\Tel(h)}\right).
\]
\begin{intro-theorem}[\cref{thm:U-infty-ff}]\label{thm:intro-thm-U-infty-ff}
  The functor~$\UE_{\infty}$ is fully faithful.
\end{intro-theorem}

The functors~$\UE_n$, for~$n \in \NN$, are called \emph{higher enveloping algebras} of spectral Lie algebras, since the functor~$\UE_1$ is an~$\infty$\nobreakdash-categorical analogue of the classical universal enveloping algebra functor.
These functors were first studied in~\cite{Knu18}. 
In~\cref{sec:higher-enveloping-Th} we give a different construction of~$\UE_n$, which helps us analyse the unit natural transformation of the induced adjunction~$\UE_{\infty} \dashv \rUE_{\infty}$ for the proof of~\cref{thm:intro-thm-U-infty-ff}.
For comparison, we discuss three candidates for the construction of~$\UE_n$ for a presentable symmetric monoidal~$\infty$\nobreakdash-category~$\Cca$ and we conjecture that they should all~agree, see~\cref{sec:relation-to-KD}.

By~\cref{thm:intro-thm-U-infty-ff} the functor~$\UE_{\infty}$ induces a fully faithful functor on the costabilisations of its source and target~(see~\cref{lem:ff-costab}).
Using this and several diagram chases we show that the functors~$\infcoSp(\UE_0)$ and~$\infcoSp(\free_{\infopLie})$ on costabilisations induced by~$\UE_0$ and the free Lie algebra functor~$\free_{\infopLie}$ are inverse to each other.
From this we conclude that the costabilisation of~$\infAlg_{\infopLie}(\infSp_{\Tel(h)})$ is equivalent to the costabilisation of~$\infSp_{\Tel(h)}$.
The latter is equivalent to~$\infSp_{\Tel(h)}$, since it is a stable~$\infty$\nobreakdash-category.

Another motivation for~\cref{thm:intro-thm-U-infty-ff} is the following analogous theorem for Lie algebras and~$\infopE_n$\nobreakdash-algebras in rational spectra, which is a consequence of the formality of the rational~$\infopE_n$\nobreakdash-operads and self Koszul-duality of~$\infopE_n$\nobreakdash-operads~\cite{CS22}.

\begin{intro-theorem}[\cref{thm:U-infty-rational-equi}]
  The sequence of inclusions~$\infopE_0 \hookrightarrow \cdots \hookrightarrow \infopE_n \hookrightarrow \infopE_{n+1} \hookrightarrow \cdots$ induces an~equivalence
  \[
  \infAlg_{\infopLie}\left(\infSp_{\QQ}\right) \xrightarrow{\sim} \varprojlim \infAlg_{\infopE_n}^{\nut}\left(\infSp_{\QQ}\right)
  \]
  of~$\infty$\nobreakdash-categories.
\end{intro-theorem}

Considering~$\infSp_{\Tel(h)}$ as a higher height analogue of~$\infSp_{\QQ}$ in chromatic homotopy theory, we make the following conjecture.

\begin{intro-conjecture}[{\cref{conj:U-finty-equi-Sp-Th}}]
  The functor
  \[
  \UE_{\infty} \colon \infAlg_{\infopLie}(\infSp_{\Tel(h)}) \to \varprojlim \infAlg_{\infopE_n}^{\nut}(\infSp_{\Tel(h)})
  \] 
  is an equivalence of~$\infty$\nobreakdash-categories.
\end{intro-conjecture}

\subsection*{Conventions}

\begin{enumerate}
  \item Throughout the text we work in the language of~$(\infty, 1)$\nobreakdash-categories, abbreviated as~$\infty$\nobreakdash-categories and modelled by quasi-categories, as introduced in~\cite{Joy02,Lur09}.
    The fundamental~$\infty$\nobreakdash-category we consider is \emph{the~$\infty$\nobreakdash-category~$\infHType$ of homotopy types}~(also known as~$\infty$\nobreakdash-groupoids, animas, or spaces), which is the~$\infty$\nobreakdash-categorical ground for doing homotopy theory of topological spaces.
  \item  We follow the size conventions in~\cite[§1.2.15]{Lur09} to deal with set-theoretic technicalities; see~\cite[§1.1]{Lan21} for a more precise explanation.
    In particular, we work in the set-theoretic framework of the ZFC axioms.
    In addition we assume the large cardinal axiom which guarantees the existence of a tower~$\Uca_{0} \subsetneq \Uca_{1} \subsetneq \cdots$ of Grothendieck universes~\cite[Definition~1.1.2]{Bor08}.
    We call a mathematical object \emph{small} if it is an element of~$\Uca_0$.
    Then, for example, the~$\infty$\nobreakdash-category~$\infHType$ of small \htypes is well-defined and it is itself an element of~$\Uca_1$.
    Since we are not writing a foundational text, this will be the last comment about set theory that we~make. 
\end{enumerate}

\subsection*{Acknowledgements}

The results in this article are part of my Ph.D. thesis written at Utrecht University.
I would like to thank my PhD supervisor Gijs Heuts for suggesting the thesis topics and I am very grateful to him for his guidance, availability and continuous support.
I would like to thank Gabriel Angelini-Knoll, Tobias Barthel, Thomas Blom, Jack Davies, Connor Malin and Gijs Heuts for proofreading drafts of this article and for their valuable feedback.

\numberwithin{equation}{subsection}

\section{Unstable~\texorpdfstring{$v_h$}{vₕ}-periodic homotopy theory}

We begin this section with necessary prerequisites on stable chromatic homotopy theory~(see~\cref{sec:stable-chrom}) and unstable periodic homotopy theory~(see~\cref{sec:usntable-periodic}) respectively.  
With those backgrounds we will present the main theorem of this paper in~\cref{sec:BK}: The universal property of the Bousfield--Kuhn functor~(see~\cref{thm:universal-property-BK}).
Throughout this section we work with a fixed prime number~$p$.

\subsection{Stable chromatic homotopy theory}\label{sec:stable-chrom}

Many ideas and methods from stable chromatic homotopy theory are indispensable for the development of unstable periodic homotopy theory.
In this section we include the relevant background on stable chromatic localisations of the~$\infty$\nobreakdash-category~$\infSp_{(p)}$ of~$p$\nobreakdash-local spectra, which will be convenient for our later discussions. 
For more detailed expositions on this subject, see~\cite{AB20},~\cite{Rav04} or~\cite{Pet19}. 

\begin{definition}\label{def:type-fin-spectra}
  Let~$h$ be a natural number and let~$\MK(h)$ denotes the~$p$-local Morava~$\MK$-theory of height~$h$.
  A~$p$\nobreakdash-local finite spectrum~$F$ is of \emph{type at least}~$h$ if its Morava~$\MK$\nobreakdash-theory homology~$\MK(n)_{\bullet}(F)$ is trivial for every~$0 \leq n \leq h-1$.
  If in addition~$\MK(h)_{\bullet}(F)$ is non-trivial, we say~$F$ is of \emph{type~$h$}.
\end{definition}

\begin{theorem}[Ravenel]\label{thm:finite-spectrum-of-type}
Let~$F$ be a~$p$\nobreakdash-local finite spectrum and let~$h$ be a natural number.
If~$\MK(h)_{\bullet}(F) = 0$, then~$F$ is of type at least~$h+1$.
\end{theorem}

\begin{proof}
  See~\cite[Theorem~2.11]{Rav84}.
\end{proof}

\begin{theorem}[Mitchell~\cite{Mit85}]
  For every~$h \in \NN$, there exists a finite~$p$\nobreakdash-local spectrum of type~$h$.
\end{theorem}

\begin{theorem}[Periodicity Theorem]\label{thm:periodicity}
Let~$F$ be a finite spectrum of type at least~$h$ for a natural number~$h$.
Then there exists a \emph{$v_{h}$ self-map}~$\susp^{d_{F}} F \to F$ for some natural number~$d_{F}$ which induces an isomorphism on~$\MK(h)$\nobreakdash-homology and induces the zero map on~$\MK(j)$\nobreakdash-homology for all~$j \neq h$.
\end{theorem}

\begin{proof}
 See~\cite[Theorem~9]{HS98}.
\end{proof}

\begin{definition}
  Let~$\Cca$ be a stable~$\infty$\nobreakdash-category. 
  A full~$\infty$\nobreakdash-subcategory~$\Cca_{0} \subseteq \Cca$ is \emph{thick} if the homotopy category~$\HoCat(\Cca_0)$ is a thick subcategory of the triangulated category~$\HoCat(\Cca)$.
  In other words,~$\Cca_0$ is thick if it is closed under equivalences, cofibre sequences and retracts.
  By closed under cofibre sequences we mean that if two of the three objects of a cofibre sequence is contained in~$\Cca_0$, then so is the remaining one.
\end{definition}

\begin{example}
  Denote the~$\infty$\nobreakdash-category of~$p$\nobreakdash-local finite spectra by~$\infSp^{\fin}_{(p)}$.
  Let~$\Pca_{ \geq h}$ denote the full~$\infty$\nobreakdash-subcategory whose objects are~$p$\nobreakdash-local finite spectra of type at least~$h$. 
  One can show that the~$\infty$-category~$\Pca_{\geq h}$ is a thick~$\infty$\nobreakdash-subcategory of~$\infSp^{\fin}_{(p)}$.
  By convention we denote the full~$\infty$\nobreakdash-subcategory whose objects are finite contractible~$p$\nobreakdash-local spectra by~$\Pca_{ \geq \infty}$.
\end{example}

\begin{theorem}[Thick Subcategory Theorem]\label{thm:thick-sub-cat}
  If~$\Tca$ is a thick subcategory of~$\infSp^{\fin}_{(p)}$, then there exists a unique~$h \in \NN \cup \{\infty\}$ such that~$\Tca$ is equivalent to the~$\infty$\nobreakdash-category~$\Pca_{ \geq h}$.
\end{theorem}

\begin{proof}
  See~\cite[Theorem~7]{HS98}.
\end{proof}

\begin{remark}
  In the situation of~\cref{thm:thick-sub-cat}, one can regard the full~$\infty$\nobreakdash-category~$\Pca_{ \geq h}$ as the ``prime ideals'' of~$\infSp^{\fin}_{(p)}$, using the theory of tensor-triangulated geometry, see~\cite{Bal20}.
\end{remark}

\begin{corollary}
  Let~$F$ be a non-trivial finite~$p$\nobreakdash-local spectrum. 
  There exists a unique natural number~$h$ such that~$F$ is of type~$h$. \qed 
\end{corollary}

\begin{passage}[The finite chromatic localisation of~$\infSp_{(p)}$]\label{para:finite-localisation-stable}
  Consider the nested sequence
  \[
  \cdots \subseteq \Pca_{ \geq h + 1} \subseteq \Pca_{ \geq h} \subseteq \cdots \subseteq \Pca_{ \geq 1} \subseteq \Pca_{\geq 0}
  \]
  of thick~$\infty$\nobreakdash-subcategories of~$\infSp^{\fin}_{(p)}$.
  For~$h \in \NN$ let~$\infSp_{\geq h}$ denote the full~$\infty$\nobreakdash-subcategory whose objects are equivalent to small colimits of objects of~$\Pca_{\geq h}$.
  Thus, there exists a~filtration 
  \[
  \cdots \subseteq \infSp_{\geq {h + 1}} \subseteq \infSp_{\geq {h}}  \subseteq \cdots \subseteq \infSp_{\geq {1}} \subseteq \infSp_{\geq {0}} = \infSp_{(p)}
  \]
  of~$\infSp_{(p)}$.
  The \emph{Verdier~quotient} 
  \[
  \Loc_{h}^{\f} \colon \infSp_{(p)} \to \infL_{h}^{\f}\left(\infSp_{(p)}\right) \coloneqq \quot{{\infSp_{(p)}}}{{\infSp_{\geq h + 1}}}
  \]
  is a reflective localisation of~$\infSp_{(p)}$, characterised by the property that it is the initial functor from~$\infSp_{(p)}$ to a stable~$\infty$\nobreakdash-category sending every object of~$\infSp_{\geq h + 1}$ to the zero object.
  The \emph{finite chromatic localisation tower} denotes the following sequence
  \begin{equation}\label{eq:finite-localisation-tower}
    \cdots \to\Loc_{h}^{\f} \to\Loc_{h -1}^{\f} \to \cdots \to \Loc_{0}^{\f} \simeq \Loc_{\QQ} \to \pt
  \end{equation}
  of localisations of the~$\infSp_{(p)}$; here~$\Loc_{0}^{\f} \simeq \Loc_{\QQ}$ is the localisation of~$\infSp_{(p)}$ at the set of rational homology~equivalences.
  
  Let~$\Loc_{h}$ denote the localisation functor of~$\infSp_{(p)}$ at the set of~$\MK(i)$\nobreakdash-homology equivalences for every~$0 \leq i \leq h$.
  By the comparison of homological localisation in~\cite[Theorem~2.1]{Rav84} we obtain another~tower  
  \begin{equation}\label{eq:chrom-tower-homology}
    \cdots \to \Loc_{h + 1} \to \Loc_{h} \to \cdots \to \Loc_{1} \to \Loc_{0}
  \end{equation}
  of localisations of~$\infSp_{(p)}$.
  There exists a natural transformation~$\Loc_{h}^{\f} \to \Loc_{h}$ for every~$h \in \NN$ which combines to a map from the tower~\eqref{eq:chrom-tower-homology} the tower~\eqref{eq:finite-localisation-tower}.
  Ravenel's telescope conjecture states that this comparison between the two towers is an equivalence, but the conjecture is recently proven to be false for all~$h \geq 2$, see~\cite{BHLS}.
\end{passage}

\begin{passage}[The spectra~$\Tel(h)$]\label{para:S(h)-T(h)}
  Let~$h \in \NN$ and let~$F_{h}$ be a finite spectrum of type~$h$ together with a~$v_{h}$-self-map~$\susp^{d} F_{h} \to F_{h}$.
  The \emph{$p$\nobreakdash-local telescope spectrum~$\Tel(h)$ of height~$h$} is defined as the colimit
  \[
  \varinjlim\left(F_{h} \xrightarrow{\susp^{-d} v_{h}} \susp^{-d} F_{h} \to \cdots \xrightarrow{\susp^{-nd} v_{h}} \susp^{-nd} F_{h} \to \cdots \right)
  \]
  in the~$\infty$\nobreakdash-category~$\infSp_{(p)}$.
  
  The construction of~$\Tel(h)$ depends on the choices of~$F_{h}$ and the~$v_{h}$ self-map.
  However, these choices are elided from the notation by the following reason: By the Thick Subcategory Theorem and the asymptotic uniqueness of the~$v_{h}$\nobreakdash-self maps~\cite[Corollaries~3.7 and~3.8]{HS98}, the notion of~$\Tel(h)$\nobreakdash-homology equivalence does not depend on these choices, see~\cite[§3]{Bou01}.
  In this article, we are only concerned with~$\Tel(h)$\nobreakdash-homology equivalences, and not with any specific~$\Tel(h)$\nobreakdash-spectrum.
\end{passage}

We will use the following important features of~$\infSp_{\Tel(h)}$.

\begin{remark}\label{rmk:Th-compact-generated}
  The~$\infty$\nobreakdash-category~$\infSp_{\Tel(h)}$ is compactly generated with a generator given by~$\Loc_{h}^{\f}(F_{h})$, by the Thick Subcategory Theorem.
  Note that any~$v_{h}$ self-map~$v_{h} \colon \susp^{d} F_{h} \to F_{h}$ becomes an equivalence~$\Loc_{h}^{\f}(v_{h})$ in~$\infSp_{\Tel(h)}$, since it is a~$v_{h}$\nobreakdash-periodic equivalence.
\end{remark}

\begin{passage}[$v_h$\nobreakdash-periodic homotopy equivalences]\label{para:stable-vh-periodic-homotopy-group}
  Let~$F_{h}$ be a~$p$\nobreakdash-local spectrum of type~$h$ together with a~$v_{h}$\nobreakdash-self-map.
  For~$E \in \infSp_{(p)}$, define the \emph{homotopy group~$\pi_{\bullet}(E; F_h)$ of~$E$ with coefficient~$F_h$} as
  \[
  \pi_{\bullet}(E; F_h) \coloneqq \pi_{\bullet}\left(\MMap(E, F_h)\right),
  \] 
  where~$\MMap(\blank, \blank)$ denotes the mapping spectrum functor.
  The~$v_h$\nobreakdash-self-map of~$F_h$ induces a self-map on the graded abelian group~$\pi_{\bullet}(E; F_h)$ by composition. 
  Let~$\BK_{F_h}(E)$ denote the colimit~in~$\infSp_{(p)}$ of the following diagram
  \begin{equation}\label{eq:stable-BK}
     \MMap_{\ast}\left(F_{h}, E\right) \xrightarrow{(v_{h})_{\ast}} \MMap_{\ast}\left(\susp^{d_{F_{h}}}F_{h}, E\right) \xrightarrow{(v_{h})_{\ast}} \cdots \xrightarrow{(v_{h})_{\ast}} \MMap_{\ast}\left(\susp^{kd_{F_{h}}}F_{h}, E\right) \xrightarrow{(v_{h})_{\ast}} \cdots
  \end{equation}
  The~\emph{$v_h$\nobreakdash-periodic homotopy group~$v_h^{-1}\pi_{\bullet}(E; F_h)$ of~$E$ with coefficient in~$F_h$} is defined as
  \[
  v_h^{-1}\pi_{\bullet}(E; F_h) \coloneqq \pi_{\bullet}\left(\BK_{F_h}(E)\right).
  \]
  The reader shall consider~$v_h^{-1}\pi_{\bullet}(E; F_h)$ as ``inverting the~$v_h$\nobreakdash-action on~$\pi_{\bullet}(E; F_h)$''.
  A morphism~$f \colon E \to E' \in \infSp_{(p)}$ is called a~\emph{$v_h$\nobreakdash-periodic equivalence} if~$f$ induces an isomorphism~$f_{\ast} \colon  v_h^{-1}\pi_{\bullet}(E; F_h) \xrightarrow{\simeq}  v_h^{-1}\pi_{\bullet}(E'; F_h)$.
  By the Thick Subcategory Theorem and the asymptotically uniqueness of the~$v_{h}$\nobreakdash-self maps, the notion of~$v_h$\nobreakdash-periodic equivalence of spectra is independent of the choices of~$F_h$ and the~$v_h$\nobreakdash-self-maps.
\end{passage}

\begin{passage}[Stable monochromatic layer of type~$h$]\label{para:stable-monocrhom}
   We explain another equivalent interpretation of the localisation ~$\infL_{h}^{\f}\left(\infSp_{(p)}\right)$, which can be easily generalised to unstable~setting.
   
  Let~$F_{h + 1}$ be a $p$\nobreakdash-local finite spectrum of type~$h + 1$.
  We say a~$p$\nobreakdash-local spectrum~$E$ is~\emph{$F_{h + 1}$\nobreakdash-null} if the \emph{mapping spectrum}~$\MMap(F_{h + 1}, E)$ is equivalent to the zero spectrum.
  A morphism~$E_1 \to E_2$ of~$p$\nobreakdash-local spectra is an \emph{$F_{h + 1}$\nobreakdash-equivalence} if the induced morphism~$\MMap(E_2, E) \to \MMap(E_1, E)$ is an equivalence for every~$F_{h + 1}$\nobreakdash-null~$p$\nobreakdash-local spectrum~$E$.
By the Class Invariance Theorem~(see~\cite[Theorem~14]{HS98}) a morphism of~$p$\nobreakdash-local spectra is a~$F_{h + 1}$\nobreakdash-equivalence if and only if it is a~$F_{h + 1}'$\nobreakdash-equivalence for any~$p$\nobreakdash-local finite spectrum~$F_{h + 1}'$ of type~$h+1$.

The functor~$\Loc_{h}^{\f}$ exhibits~$\infL_{h}^{\f}\left(\infSp_{(p)}\right)$ as the localisation of the~$\infty$\nobreakdash-category~$\infSp_{(p)}$ at the set of~$F_{h + 1}$\nobreakdash-equivalences of~$p$\nobreakdash-local spectra, see~\cite[§2, §3.2]{Bou01}.
By~\cite[§3.8]{Bou01} a morphism of~$p$\nobreakdash-local spectra is a~$F_{h + 1}$\nobreakdash-equivalence if and only if it is a~$v_{n}$\nobreakdash-periodic equivalence for all~$0 \leq n \leq h$.
From \loccit we also know that~$v_m^{-1}\pi_{\bullet}(\Loc_h^{\f}(E); F_h) = 0$ for any~$m > h$ and for any~$E \in \infSp_{(p)}$.

By the Class Invariance Theorem there exists a fibre sequence
\[
\Mon_{h}^{\f} \to \Loc_{h}^{\f} \to \Loc_{h - 1}^{\f}
\]
of functors for~$h \geq 1$.
We say a~$p$\nobreakdash-local spectrum~$E$ is \emph{$v_{h}$\nobreakdash-periodic} if it is in the essential image of~$\Mon_{h}^{\f}$. 
For a~$v_h$\nobreakdash-periodic spectrum~$E$, the~$v_{i}$\nobreakdash-periodic homotopy groups of~$E$ vanishes for every~$i \neq h$.
Let~$\Mca_{h}^{\f}(\infSp_{(p)})$ denotes the full~$\infty$\nobreakdash-subcategory of~$\infSp_{(p)}$ whose objects are~$v_{h}$\nobreakdash-periodic spectra.
A morphism in~$\Mca_{h}^{\f}(\infSp_{(p)})$ is an equivalence if and only if the underlying morphism of spectra is a~$v_{h}$\nobreakdash-periodic equivalence.
Therefore, the~$\infty$\nobreakdash-category~$\Mca_{h}^{\f}(\infSp_{(p)})$ is a model for the localisation of the~$\infty$\nobreakdash-category~$\infSp_{(p)}$ at the set of~$v_h$\nobreakdash-periodic equivalences.

In the definition of~$v_h$\nobreakdash-periodic homotopy groups, see~\cref{para:stable-vh-periodic-homotopy-group}, we have
\[
\BK_{F_h}(E) \simeq \Tel(h) \otimes E, 
\]
where~$\Tel(h)$ is the telescope spectrum defined using the Spanier--Whitehead dual of~$F_{h + 1}$, see~\cref{para:S(h)-T(h)}.
Thus, a morphism of~$p$\nobreakdash-local spectra is a~$v_{h}$\nobreakdash-periodic equivalence if and only if it is a~$\Tel(h)$\nobreakdash-homology equivalences.
Moreover, we have an equivalence 
\[
  \Mca_{h}^{\f}(\infSp_{(p)}) \xrightarrow{\sim} \infSp_{\Tel(h)}
\]
of~$\infty$\nobreakdash-categories, see~\cite[Theorem~3.3]{Bou01}.
The stable~$\infty$\nobreakdash-categories in the above equivalence are both known as \emph{monochromatic chromatic layer of height~$h$}.
\end{passage}

\subsection{Unstable periodic homotopy theory}\label{sec:usntable-periodic}

Let~$h \geq 1$ be a natural number.
The definition of~$v_h$\nobreakdash-periodic equivalences can be analogously applied to~$p$\nobreakdash-local \htypes.
Let~$V_h$ be a pointed~$p$\nobreakdash-local \emph{finite complex} of types~$h$.
By the Periodicity~\cref{thm:periodicity} and the Freudenthal Suspension Theorem~\cite[Corollary~3.2.3]{Koc96} we know that~$\susp^{c_{h}}(V_h)$ admits a~$v_h$\nobreakdash-self-map, for some~$c_h \in \NN$.
For our application we can assume without loss of generality that~$V_h$ admits a~$v_h$\nobreakdash-self-map.
Replacing~$F_h$ by~$V_h$, the spectrum~$E$ by a pointed~$p$\nobreakdash-local \htype~$X$ and the mapping spectrum by pointed mapping space in~\eqref{eq:stable-BK}, one obtains a pointed space~$\Tel_{V_h}(X)$ as the colimit of the corresponding diagram in the~$\infty$-category~$\infHType_{(p)}$ of~$p$-local \htypes. 
It turns out that~$\Tel_{V_h}(X)$ is an infinite loop space and is thus the underlying infinite loop space of a spectrum~$\BK_{V_h}(X)$.\footnote{See for example~\cite[§3.2]{ShiTh} for a proof and a~$\Loop$\nobreakdash-spectrum model for~$\BK_{V_h}(X)$.}

\begin{definition}\label{def:vh-periodic-group}
  Use the notations as above.
  \begin{enumerate}
    \item The~\emph{$v_h$\nobreakdash-periodic homotopy group}~$v_h^{-1}\pi_{\bullet}(X;V_h)$ of~$X$ with coefficient~$V_h$ is defined as the graded abelian~group
      \[
      v_h^{-1}\pi_{\bullet}(X;V_h) \coloneqq \pi_{\bullet}\left(\Tel_{V_h}(X)\right).
      \]
    \item A morphism of pointed connected~$p$\nobreakdash-local \htypes is a~\emph{$v_h$\nobreakdash-periodic homotopy equivalence} if it induces an isomorphism on the~$v_h$\nobreakdash-periodic homotopy groups with coefficient~$V_h$.
    \item A morphism of~$p$\nobreakdash-local \htypes is a~$v_h$\nobreakdash-periodic homotopy equivalence if its restriction to each connected component with an arbitrarily chosen basepoint is a~$v_h$\nobreakdash-periodic homotopy equivalence.
    \item  For a morphism~$f \colon X \to Y$ between nilpotent \htypes, we say by convention that~$f$ is a \emph{$v_0$\nobreakdash-periodic equivalence} if~$f$ is a rational homotopy equivalence.
  \end{enumerate}
\end{definition}

\begin{remark}
  \begin{enumerate}
    \item A~$v_h$\nobreakdash-periodic homotopy equivalence of pointed homotopy types depends neither on the choice of the finite complex of type~$h$ nor on the choice of the~$v_h$\nobreakdash-self-map.
    \item Let~$X$ be a connected \htype. 
      The canonical morphism~$\tau_{> c}(X) \to X$ defined as the fibre of the Postnikov truncation~$X \to \tau_{\leq c}(X)$ is a~$v_h$\nobreakdash-periodic equivalence, for any natural number~$c$.
  \end{enumerate}
  
  See~\cite[§3.2]{ShiTh} for the proofs. 
\end{remark}

Using the theory of unstable localisation~\cite{Dro95,Bou94}, one can also construct the localisation of the~$\infty$\nobreakdash-category~$\infHType_{(p)}^{\geq 2}$ of~$p$\nobreakdash-local simply-connected \htypes at the set of~$v_h$\nobreakdash-periodic equivalences from an ``unstable finite chromatic localisation tower'' analogue to~\eqref{eq:finite-localisation-tower}.
The classical reference for unstable (periodic) localisations are~\cite{Dror92,Bou94,Bou01}.
For expositions in~$\infty$\nobreakdash-categorical languages, see~\cite{EHMM19, Heu20, Heu21, ShiTh}.
Here we recall briefly the basic constructions.

\begin{situation}
  Let~$V_{h+1}$ be a pointed~$p$\nobreakdash-local finite complex of type~$h+1$ and assume that~$V_{h+1}$ \emph{admits a desuspension}, \ie there exists a pointed \htype~$V_{h+1}'$ such that $V_{h+1} \simeq \susp\left(V_{h+1}'\right)$.
  For a~\htype we denote by~$X_{+}$ the pointed \htype~$X \coprod \{\pt\}$ where~$\pt$ is the basepoint.
\end{situation}

\begin{definition}\label{def:V-less-V-equivalence}
  \begin{enumerate}
    \item A pointed \htype~$X$ is~\emph{$V_{h+1}$\nobreakdash-less} if the induced map
      \[
      \infMap_{\ast}(\pt_{+}, X) \xrightarrow{\sim} \infMap_{\ast}((V_{h+1})_{+}, X)
      \]
      of pointed mapping spaces is an equivalence.
    \item A morphism~$f \colon Y \to Z$ of pointed \htypes is a \emph{$V_{h+1}$\nobreakdash-equivalence} if for every pointed~$V_{h+1}$\nobreakdash-less~\htype~$(X, x_0)$ the induced map 
        \[
        f^{\ast} \colon \infMap_{\ast}(Z, X) \to \infMap_{\ast}(Y, X)
        \]
        is an equivalence of pointed mapping spaces.
  \end{enumerate}
\end{definition}

\begin{remark}
  \Cref{def:V-less-V-equivalence} is an analogue of the notion of ``$F_h$\nobreakdash-null'' and ``$F_h$\nobreakdash-equivalence'' in~\cref{para:stable-monocrhom}.
  Note that a pointed connected \htype~$X$ is~$V_{h+1}$\nobreakdash-less if and only if~$\infMap_{\ast}(V_{h+1}, X)$ is contractible.
\end{remark}

It is shown that the localisation of the~$\infty$\nobreakdash-category~$\infHType_{\ast}$ of pointed \htypes at the set of~$V_{h+1}$\nobreakdash-equivalences exists.

\begin{theorem}[Bousfield]\label{thm:exists-nullification-pointed}
  There exists a functor~${\Null_{V_{h+1}} \colon \infHType_{\ast} \to \infHType_{\ast}}$ together with a natural transformation~${\lambda_{V_{h+1}} \colon \id_{\infHType_{\ast}} \to \Null_{V_{h+1}}}$ such that for every~$X \in \infHType_{\ast}$
  \begin{enumerate}
    \item the pointed \htype~$\Null_{V_{h+1}}(X)$ is~$V_{h+1}$\nobreakdash-\less,
    \item the morphism~$\lambda_{V_{h+1}}(X) \colon X \to \Null_{V_{h+1}}(X)$ is a~$V_{h+1}$\nobreakdash-equivalence in~$\infHType_{\ast}$, and
    \item the induced natural transformation~$\Null_{V_{h+1}} \to \Null_{V_{h+1}} \circ \Null_{V_{h+1}}$ is an equivalence in the~$\infty$\nobreakdash-category~$\infFun(\infHType_{\ast}, \infHType_{\ast})$.
  \end{enumerate}
\end{theorem}

\begin{proof}
  See~\cite[Theorem~2.10]{Bou94}.
\end{proof}

\begin{definition}\label{def:contraction}
  We call the functor~$\Null_{V_{h+1}}$ the \emph{contraction of~$V_{h+1}$} or \emph{$V_{h+1}$\nobreakdash-contraction}.
  For~$X \in \infHType_{\ast}$, we call the pointed \htype~$\Null_{V_{h+1}}(X)$ together with the morphism~$\lambda_{V_{h+1}}(X) \colon X \to \Null_{V_{h+1}}(X)$ the~\emph{$V_{h+1}$\nobreakdash-contraction of~$X$}.
\end{definition}

\begin{passage}[Explanation]
In the situation of~\cref{def:contraction} one can regard the \contraction of~$V_{h+1}$ as the universal functorial way to ``quotient out'' the~$V_{h+1}$\nobreakdash-information in a pointed homotopy type. 
Informally speaking, it is an unstable analogue of taking the Verdier quotient in stable~$\infty$\nobreakdash-categories, \cf~\cref{para:stable-monocrhom}.

\Cref{def:V-less-V-equivalence} has nothing special to finite complexes.
In other word, one can replace~$V_{h+1}$ by any other object~$W \in \infHType_{\ast}$ and obtain the notions of~$W$\nobreakdash-less,~$W$\nobreakdash-equivalence and~$W$\nobreakdash-contraction.
For example, let~$\usphere^{n+1}$ denote the \htype of the sphere of dimension~$n+1$.
Then a pointed \htype~$X$ is~$\usphere^{n+1}$\nobreakdash-less if~$X$ is~$n$\nobreakdash-truncated, \ie~$\pi_{k}(X) = 0$ for all~$k \geq n+1$.
And the~$\usphere^{n+1}$-contraction~$\Null_{\usphere^{n+1}}(\blank)$ is the Postnikov truncation~$\tau_{\leq n}(\blank)$.
Thus, one can also view the notions of~$W$\nobreakdash-less and~$W$\nobreakdash-contraction as generalisations of connectedness and the Postnikov truncations, respectively.

In the literatures the functor~$\Null_{W}$ is often called ``$W$\nobreakdash-nullification'' or ``$W$\nobreakdash-localisation'', \cf~\cite{Dror92, Bou94, Dro95}
We prefer the name~$W$\nobreakdash-\contraction because under the construction~$\Null_{W}$ the \htype~$W$ becomes contractible instead of becoming ``null''.
\end{passage}

The relationship between~$V_{h+1}$\nobreakdash-contraction and~$v_h$\nobreakdash-periodic equivalences of~\htypes is summarised in the following theorem.

\begin{theorem}[Bousfield]\label{thm:Vh-truncatio-vh-periodic-equivalence}
  Let~$X$ be a pointed connected~$p$\nobreakdash-local \htype.
  \begin{enumerate}
    \item The natural morphism~$\lambda_{V_{h+1}}(X) \colon X \to \Null_{V_{h+1}}(X)$ is a~$v_n$\nobreakdash-periodic equivalence for every~$1 \leq n \leq h$.
      If~$X$ is simply-connected, then~$\lambda_{V_{h+1}}(X)$ is also a~$v_0$\nobreakdash-periodic equivalence, \ie rational homotopy equivalence.
    \item Every~$V_{h+1}$\nobreakdash-equivalence is a~$v_n$\nobreakdash-periodic equivalence for every~$1 \leq n \leq h$.
      If we assume in addition that the source and the target is also simply-connected, then it is also a~$v_0$\nobreakdash-periodic equivalence.
    \item The~$v_m$\nobreakdash-periodic homotopy groups of~$\Null_{V_{h+1}} (X)$ vanish, for every~$m \geq h+1$.
  \end{enumerate}
\end{theorem}

\begin{proof}
  See~\cite[§§10-11]{Bou94}.
\end{proof}

Now we will review the construction of the~$v_h$\nobreakdash-periodic localisation~$\infHType_{v_h}$ of the~$\infHType_{\ast}$, which is already indicated in~\cite{Bou94,Bou01} and worked out in details~$\infty$\nobreakdash-categorically in~\cite{EHMM19,Heu21}.
We omit all the proofs and refer the reader to~\loccit or to~\cite[§3.4]{ShiTh} for more details.

\begin{situation}\label{sit:monochrom-construction}
  We fix a natural number~$h \geq 1$.
  Let~$V_{h + 1}$ and~$V_{h}$ be pointed~$p$\nobreakdash-local finite complexes of type~$h+1$ and type~$h$, respectively.
  Assume in addition that
  \begin{enumerate}
    \item both~$V_{h+1}$ and~$V_h$ admit a desuspension, and
    \item $\conn(V_{h + 1}) \geq \conn(V_{h})$, where~$\conn(\blank)$ denotes the connectivity of \htypes. 
  \end{enumerate}
  
  Let us denote
  \[
  c^{+}\left(V_{h+1}\right) \coloneqq \conn\left(V_{h+1}\right) + 1 \text{ and similarly } c^{+}\left(V_{h}\right) \coloneqq \conn\left(V_{h}\right) + 1
  \] 
  Thus we have that~$c^{+}\left(V_{h + 1}\right) \geq c^{+}\left(V_{h}\right)$.
  By the Unstable Class Invariance Theorem~\cite[Theorem~9.15]{Bou94} there exists a natural transformation~$\Null_{V_{h + 1}} \to \Null_{V_{h}}$, given by~$V_{h}$\nobreakdash-contraction.
\end{situation}

\begin{construction}
  Consider the natural transformation~$\Null_{V_{h + 1}} \to \Null_{V_{h}}$ as a morphism in~$\infty$\nobreakdash-category~$\infFun\left(\infHType_{\ast}, \infHType_{\ast}\right)$ of functors.
  Denote its fibre by~$ \relP_{V_{h + 1}, V_{h}}$.
  In particular, for every pointed connected \htype~$X$, there exists a fibre sequence
  \[
  \relP_{V_{h + 1}, V_{h}}(X) \to \Null_{V_{h + 1}}(X) \to \Null_{V_{h}}(X) 
  \] 
  of pointed connected \htypes, since limits of functors are computed pointwise.
\end{construction}

\begin{proposition}\label{prop:vn-group-monochrom}
  Let~$X$ be a pointed connected \htype.
  Then the \htype~$\relP_{V_{h + 1}, V_{h}}(X)$ is also connected.
  Furthermore:
  \begin{enumerate}
    \item For every natural number~$n \geq 1$ and~$n \neq h$, the~$v_{n}$\nobreakdash-periodic homotopy groups of the \htype~$ \relP_{V_{h + 1}, V_{h}}(X)$ vanish.
    \item The~$v_{h}$\nobreakdash-periodic homotopy groups of the \htype~$\relP_{V_{h + 1}, V_{h}}(X)$ are isomorphic to those of~$X$.
  \end{enumerate}
\end{proposition}

\begin{proposition}\label{prop:monochrom-vh-equivalence}
  Let~$f \colon X \to Y$ be a morphism of pointed connected \htypes.
  The following statements are equivalent:
  \begin{enumerate}
    \item The map~$f$ is a~$v_{h}$\nobreakdash-periodic equivalence.
    \item The induced map~$(\tau_{> c^{+}(V_{h + 1})} \circ  \relP_{V_{h + 1}, V_{h}})(f)$ is a~$v_{h}$\nobreakdash-periodic equivalence.
    \item The induced map~$(\tau_{> c^{+}(V_{h + 1})} \circ  \relP_{V_{h + 1}, V_{h}})(f)$ is an equivalence in~$\infHType_{\ast}$.
  \end{enumerate}
\end{proposition}

\begin{definition}\label{def:vh-loc-cat}
  We define the functor
  \begin{align*}
   \Loc_{v_{h}} \colon \infHType_{\ast} &\to \infHType_{\ast}^{> c^{+}(V_{h + 1})} \\
    X &\mapsto \left(\tau_{> c^{+}(V_{h + 1})}\circ  \relP_{V_{h + 1}, V_{h}}\right)(X).
  \end{align*}
  Let~$\infHType_{v_{h}}$ denote the full~$\infty$\nobreakdash-subcategory of~$\infHType_{\ast}$ whose objects are pointed connected homotopy types that are equivalent to~$\left(\tau_{> c^{+}(V_{h + 1})}\circ  \relP_{V_{h + 1}, V_{h}}\right)(X)$ for some pointed \htype~$X$.
\end{definition}

\begin{theorem}[{\cite[Theorem~2.2]{Heu21}}]\label{thm:vh-loc-cat}
  The functor~$\Loc_{v_{h}}$ exhibits~$\infHType_{v_{h}}$ as a localisation of the~$\infty$\nobreakdash-category~$\infHType_{\ast}$ at the set of~$v_{h}$\nobreakdash-periodic equivalences.
  In particular, for every~$\infty$\nobreakdash-category~$\Cca$, composing with~$\Loc_{v_{h}}$ induces an equivalence
  \[
  \infFun\left(\infHType_{v_{h}}, \Cca\right) \xrightarrow{\sim} \infFun^{v_{h}}\left(\infHType_{\ast}, \Cca\right)
  \] 
  of~$\infty$\nobreakdash-categories where~$\infFun^{v_{h}}$ denotes the~$\infty$\nobreakdash-category of functors that send~$v_{h}$\nobreakdash-periodic equivalences in~$\infHType_{\ast}$ to equivalences in~$\Cca$.
\end{theorem}

\begin{remark}
  We also call the~$\infty$\nobreakdash-category~$\infHType_{v_h}$ the \emph{unstable monochromatic layer of type~$h$}.
\end{remark}

\subsection{The Bousfield--Kuhn functor}\label{sec:BK}

Let~$h \geq 1$ be a natural number.
The unstable and stable monochromatic layers of type~$h$ are connected via the so-called Bousfield--Kuhn functor~$\BK_h$~\cite{Bou87,Kuh89,Bou01}.
We recall the characterising properties of~$\BK_h$ and present our theorem on the universal property of the Bousfield--Kuhn functor, see~\cref{thm:universal-property-BK}. 

\begin{theorem}[Bousfield--Kuhn functor]\label{thm:BK-functor}
  Let~$h \geq 1$ be a natural number.
  There exists a functor~$\widetilde{\BK}_{h} \colon \infHType_{\ast} \to \infSp$ satisfying the following properties:
  \begin{enumerate}
    \item For every~$X \in \infHType_{\ast}$, the spectrum~$\widetilde{\BK}_{h}(X)$ is~$\Tel(h)_{\bullet}$\nobreakdash-local.
    \item For every~$p$\nobreakdash-local finite complex~$V_{h}$ of type~$h$, there exist isomorphisms
      \[
      v_{h}^{-1}\pi_{\bullet}(X; V_{h}) \cong \pi_{\bullet}\MMap(\susp^{\infty} V_{h}, \widetilde{\BK}_{h}(X)) \cong v_{h}^{-1}\pi_{\bullet}(\widetilde{\BK}_{h}(X); \susp^{\infty} V_{h}),
      \]
      which are natural in~$X$.
    \item For every spectrum~$W$, there exists a natural equivalence~$\widetilde{\BK}_{h}(\Loop^{\infty} E) \simeq \Loc_{\Tel(h)}(E)$.
    \item The functor~$\widetilde{\BK}_{h}$ sends a~$v_{h}$\nobreakdash-periodic equivalence of pointed \htypes to an equivalence of spectra.
  \end{enumerate}
\end{theorem}

\begin{proof}
  See~\cite[Theorem~5.3]{Bou01}.
\end{proof}

\begin{remark}\label{rmk:BK-computes-vh-groups}
  Let us emphasise one important feature of the Bousfield--Kuhn functor: For a \htype~$X$, one can compute its~$v_h$\nobreakdash-periodic homotopy groups from the stable homotopy groups of~$\widetilde{\BK}_{h}(X)$, by~\rom{2} in the above theorem.
\end{remark}

\begin{theorem}[Bousfield]\label{thm:BK-adjunction}
  Let~$h \geq 1$ be a natural number.
  Then the induced functor~$\widetilde{\BK}_{h} \colon \infHType_{\ast} \to \infSp_{\Tel(h)}$ factors through~$\infHType_{v_{h}}$, \ie there exists a functor
  \[
  \BK_{h} \colon \infHType_{v_{h}} \to \infSp_{\Tel(h)},
  \]
  unique up to contractible choices, such that~$\widetilde{\BK}_{h} \simeq \BK_{h} \circ \Loc_{v_{h}}$.
  Moreover, the functor~$\BK_{h}$ admits a left adjoint~$\lBK_{h}$.
\end{theorem}

\begin{proof}
  The factorisation follows from~\cref{thm:BK-functor}.\rom{4}. 
  See~\cite[Theorem~5.4 and Corollary~5.6]{Bou01} for the existence of the left adjoint.
\end{proof}

\begin{convention}
  The functor~$\BK_{h}$ is also known as the \emph{Bousfield--Kuhn functor}. 
\end{convention}

The definition of the Bousfield--Kuhn functor~$\BK_h$ is through construction. 
In the following theorem we give a universal property of the functor~$\BK_h$ for every natural number~$h \geq 1$.

\begin{theorem}\label{thm:universal-property-BK}
    For every natural number~$h \geq 1$ the adjunction~$\lBK_h \dashv \BK_h$ satisfies the following universal property:
    \begin{enumerate}
      \item Let~$\Cca$ be a stable~$\infty$\nobreakdash-category.
      Then composing with~$\lBK_{h}$ induces an~equivalence
      \[
      \infFun^{\rex}(\Cca, \infSp_{\Tel(h)}) \xrightarrow{\sim} \infFun^{\rex}\left(\Cca, \infHType_{v_{h}}\right),
      \]
      of~$\infty$\nobreakdash-categories, where~$\infFun^{\rex}$ denotes the~$\infty$\nobreakdash-category of right exact functors, \ie functors that preserve finite colimits.
      \item Let~$\Dca$ be a presentable stable~$\infty$\nobreakdash-category.
      Composing with the Bousfield--Kuhn functor~$\BK_{h}$ induces an equivalence
       \[
       \infFun^{R}(\infSp_{\Tel(h)}, \Dca) \xrightarrow{\sim} \infFun^{R}\left(\infHType_{v_{h}}, \Dca\right),
       \]
       of~$\infty$\nobreakdash-categories, where~$\infFun^{R}$ denote the~$\infty$\nobreakdash-category of functors that are accessible and  preserves small limits, \ie functors admitting left adjoints. 
    \end{enumerate}  
\end{theorem}

We will prove this theorem in~\cref{sec:costab-vh}.
The idea of the proof is to use a Lie-algebra model for the~$\infty$\nobreakdash-category~$\infHType_{v_h}$~(see~\cref{sec:Lie-en-operad}) and associate the adjunction~$\lBK_h \dashv \BK_h$ with the \emph{costabilisation} of the~$\infty$\nobreakdash-category~$\infHType_{v_h}$~(see~\cref{sec:costab-vh}).

\section{Higher enveloping algebras in monochromatic layers}\label{sec:higher-enveloping}

Following the work of Heuts~\cite{Heu21}, we review in~\cref{sec:Lie-en-operad} the spectral Lie algebra model for~$v_h$\nobreakdash-periodic \htypes.
We give our construction of the higher enveloping algebra functor~$\UE_n$ on~$\Tel(h)$\nobreakdash-local spectral Lie algebras in~\cref{sec:higher-enveloping-Th}, for every~$n \in \NN$ and every~$h \geq 1$, and prove that they induce a fully faithful embedding
\begin{equation}\label{eq:U-infty-intro-sec}
   \UE_{\infty} \colon \infAlg_{\infopLie}\left(\infSp_{\Tel(h)}\right) \to \varprojlim \infAlg_{\infopE_n}^{\nut}\left(\infSp_{\Tel(h)} \right)
\end{equation}
 of~$\infty$\nobreakdash-categories in~\cref{sec:fully-faithful}~(see~\cref{thm:U-infty-ff}).
Towards the end of the section we discuss three conjecturally equivalent candidates for the construction of higher enveloping algebras in an arbitrary presentable stable symmetric monoidal~$\infty$\nobreakdash-category~(see~\cref{sec:relation-to-KD}).
As consequences, we obtain that the~$\MK(h)$\nobreakdash-local version of the functor~\eqref{eq:U-infty-intro-sec} is also fully faithful~(see~\cref{thm:U-infty-ff-K(h)}), and we show that the~$\infty$\nobreakdash-category~$\infAlg_{\infopLie}(\infSp_{\QQ})$ of rational spectral Lie algebras~(which is the~$\infty$\nobreakdash-category underlying differential graded rational Lie algebras) is \emph{equivalent} to the inverse limit~$\varprojlim \infAlg_{\infopE_n}^{\nut}\left(\infSp_{\QQ} \right)$~(see~\cref{thm:U-infty-rational-equi}).
The latter motivates us to make the conjecture that the functor~$\UE_{\infty}$ in~\eqref{eq:U-infty-intro-sec} is also an equivalence~(see~\cref{conj:U-finty-equi-Sp-Th}).

\subsection{Lie algebras in stable monochromatic layers}\label{sec:Lie-en-operad}

In this section we recall this spectral Lie algebra model for~$\infHType_{v_h}$ from the work~\cite{Heu21} of Heuts, which we will use in our proof of~\cref{thm:universal-property-BK}.
For this purpose we first recall briefly some prerequisites and notations on~$\infty$\nobreakdash-operads.
The theory of~$\infty$\nobreakdash-operads is an important tool for describing algebraic structures on objects of~$\infty$\nobreakdash-categories.

In~\cite{HA} Lurie introduced his theory of~$\infty$\nobreakdash-operads, which is the~$\infty$\nobreakdash-categorical generalisation of the theory of coloured operad with values in simplicial sets.
One can use Lurie's theory of~$\infty$\nobreakdash-operads with values in~$\infHType$ to define~$\infty$\nobreakdash-operads with values in a presentable symmetric monoidal~$\infty$\nobreakdash-category, as discussed and used in~\cite{BraTh,BCN23,ShiTh,HLEn}.
We recall briefly this model for~$\infty$\nobreakdash-operads and refer the reader to~\loccit for the proofs.
Another construction of~$\infty$\nobreakdash-operads via symmetric sequence is given in~\cite{Hau22}.

\begin{passage}[Symmetric sequences and composition products]\label{para:symseq-compo}
  Let~$\Cca$ be a \emph{presentable symmetric monoidal~$\infty$\nobreakdash-category}, \ie~$\Cca$ is a commutative algebra object in the symmetric monoidal~$\infty$\nobreakdash-category~$\infPrl$ of presentable~$\infty$\nobreakdash-categories and small-colimit-preserving functors, see~\cite[Proposition~4.8.1.15]{HA}.
  Define the~$\infty$\nobreakdash-category~$\infFin^{\simeq} \coloneqq \infNv(\CatFin^{\cong})$ of finite sets and bijections, where~$\infNv(\blank)$ denotes the the simplicial nerve functor~(see~\cite[Definition~1.1.5.5]{Lur09}).
  Define the~$\infty$\nobreakdash-category~$\infSSeq(\Cca) \coloneqq \infFun(\infFin^{\simeq}, \Cca)$ of \emph{symmetric sequences} in~$\Cca$.
  We denote an object~$M$ in~$\infSSeq(\Cca)$ by~$M = \left(M(r)\right)_{r \geq 0}$ where~$M(r)$ is the evaluation of the functor~$M$ on the set of~$r$ elements.
  There exists a functor
  \begin{align*}
  \T_{\blank} \colon \infSSeq(\Cca) &\to \infFun(\Cca, \Cca) \\
  M &\mapsto \left(X \mapsto \coprod_{r \geq 0} \left(M(r) \otimes_{\Cca} X^{\otimes r} \right)_{\Perm_r}\right)
  \end{align*}
  
  Using the universal property of (the~$\infty$\nobreakdash-categorical) Day convolution, one can define a monoidal structure on~$\infSSeq(\Cca)$, called the~\emph{composition product~$\circledcirc$}, inspired by the work~\cite{Tri} in 1-categorical setting.
  Then the functor~$\T_{\blank}$ becomes a monoidal functor with respect to~$\circledcirc$ and the composition~$\circ$ of endofunctors.
  Note that an associative algebra object of the monoidal~$\infty$\nobreakdash-category~$\infFun(\Cca, \Cca)$ is called an \emph{monad} on~$\Cca$, see~\cite[Definition~4.7.0.1]{HA}.
\end{passage}

\begin{definition}
  Let~$\Cca$ be a presentable symmetric monoidal~$\infty$\nobreakdash-category.
  \begin{enumerate}
    \item An~\emph{$\infty$\nobreakdash-operad~$\Oca$ with values in~$\Cca$} is an associative algebra object of~$\infSSeq(\Cca)$~(with the composition~product as the monoidal structure).
      Denote the~$\infty$\nobreakdash-category of~$\infty$\nobreakdash-operads with values in~$\Cca$ by~$\infOpd(\Cca)$.
    \item An \emph{$\Oca$\nobreakdash-algebra} in~$\Cca$ is a left module in~$\Cca$ over the associated monad~$\T_{\Oca}$.
      We denote the~$\infty$\nobreakdash-category~$\infLmod_{\T_{\Oca}}(\Cca)$ of~$\Oca$\nobreakdash-algebras also by~$\infAlg_{\Oca}(\Cca)$.
    \item Whenever we denote~$\infLmod_{\T_{\Oca}}(\Cca)$ by~$\infAlg_{\Oca}(\Cca)$, we abbreviate the forgetful functor~$\frgt_{\T_{\Oca}} \colon \infLmod_{\T_{\Oca}}(\Cca) \to \Cca$~(see~\cite[Corollary~4.2.3.5]{HA}) by~$\frgt_{\Oca}$.
  \end{enumerate}
\end{definition}

\begin{passage}\label{sit:underlying-symseq}
  Let~$\Cca$ be a presentable symmetric monoidal~$\infty$\nobreakdash-category. 
  Then there exists a symmetric monoidal functor~$F \colon \infHType \to \Cca$ in~$\infPrl$, unique up to contractible choice, since~$\infHType$ is symmetric monoidal and is~the free presentable~$\infty$\nobreakdash-category generated by a point.
  By~\cite[Proposition~5.2.5.1]{ShiTh} we obtain an induced functor~$F \colon \infOpd(\infHType) \to \infOpd(\Cca)$.
  Let~$\Oca$ be an~$\infty$\nobreakdash-operad with values in~$\infHType$. 
  The~$\infty$\nobreakdash-category~$\infAlg_{\Oca}(\Cca)$ of \emph{$\Oca$\nobreakdash-algebras in~$\Cca$} is defined as the~$\infty$\nobreakdash-category~$\infAlg_{F(\Oca)}(\Cca)$.
\end{passage}

Let~$\Oca^{\otimes} \to \infFin_{\ast}$ be an~$\infty$\nobreakdash-operad in Lurie's model.
Denote the~$\infty$\nobreakdash-category of~$\Oca$\nobreakdash-algebras by~$\infAlg_{\Oca / \infopCom}(\Cca)$, see~\cite[Definition~2.1.3.1]{HA}. 
One can construct a symmetric sequence~$\Oca$ from the structure maps of~$\Oca^{\otimes}$, see~\cite[Situation~5.2.5.3]{ShiTh}. 

\begin{theorem}[{\cite[Theorem~5.2.5.5]{ShiTh}}]\label{thm:sym-fun-algebra}
  \begin{enumerate}
    \item The symmetric sequence~$\Oca$ underlies an~$\infty$\nobreakdash-operad with values in~$\infHType$.
    \item There exists an equivalence~$\infAlg_{\Oca / \infopCom}\left(\Cca\right) \simeq \infAlg_{\Oca}(\Cca)$ of~$\infty$\nobreakdash-categories. 
  \end{enumerate}
\end{theorem}

\begin{example}\label{ex:triv-opd-C}
  Consider the following elementary example.
  Let~$\infopTriv_{\Cca}$ denotes the \emph{trivial~$\infty$\nobreakdash-operad} with values in a presentable symmetric monoidal~$\infty$\nobreakdash-category~$\Cca$: We have~$\infopTriv(1) \simeq \munit_{\Cca}$ and otherwise~$\infopTriv(r) = \emptyset_{\Cca}$, the initial object of~$\Cca$.
  Every object of~$\Cca$ is canonically an algebra over~$\infopTriv_{\Cca}$ via the identity morphism of~$X$.
  More precisely, we have an equivalence~$\infAlg_{\infopTriv}(\Cca) \simeq \Cca$ of~$\infty$\nobreakdash-categories, see~\cite[Example~2.1.3.5]{HA}.
\end{example}

\begin{proposition}\label{prop:operad-map-induced-functor}
  Let~$\Cca$ be a presentable symmetric monoidal~$\infty$\nobreakdash-category.
  A morphism~$f \colon \Oca \to \Pca$ in~$\infOpd(\Cca)$ induces a forgetful functor~$f^{\ast} \colon \infAlg_{\Pca}(\Cca) \to \infAlg_{\Oca}(\Cca)$ such that~$\frgt_{\Oca} \circ f^{\ast} \simeq \frgt_{\Pca}$.
\end{proposition}

\begin{proof}
  This is by~\cite[Corollaries~4.2.3.2 and~4.2.3.3]{HA}.
\end{proof}

\begin{proposition}\label{prop:adj-induced-alg}
  Let~$\Cca$ be a presentable symmetric monoidal~$\infty$\nobreakdash-category.
  Let~$f \colon \Oca \to \Pca$ be a morphism of~$\infty$\nobreakdash-operads with values in~$\Cca$. 
  There exists an adjunction
  \[
  f_{!} \colon \infAlg_{\Oca}(\Cca) \rightleftarrows \infAlg_{\Pca}(\Cca) \cocolon f^{\ast}. 
  \]
\end{proposition}

\begin{proof}
  It is shown in \cite[Corollary~4.2.3.3]{HA} that~$f^{\ast}$ preserves small limits.
  Since~$\Cca$ is presentable, the~$\infty$\nobreakdash-categories~$\infAlg_{\Oca}(\Cca)$ and~$\infAlg_{\Pca}(\Cca)$ are presentable and~$f^{\ast}$ is accessible by~\cite[Corollary 4.2.3.7]{HA}. 
  Thus, the existence of the adjunction follows from the Adjoint Functor Theorem~\cite[Corollary~5.5.2.9]{Lur09}.
\end{proof}

\begin{example}\label{ex:free-forg-adjunction}
  Let~$\Cca$ be a presentable symmetric monoidal~$\infty$\nobreakdash-category.
  We denote the adjunction induced by the morphism~$\infopTriv_{\Cca} \to \Oca$ of~$\infty$-operads with values in~$\Cca$ by 
  \[
  \free_{\Oca} \colon \Cca \rightleftarrows \infAlg_{\Oca}(\Cca) \cocolon \frgt_{\Oca}. 
  \]
\end{example}

\begin{passage}[The spectral Lie~$\infty$\nobreakdash-operad]\label{para:spectral-Lie-operad}
  We summarise here the introduction to the spectral Lie~$\infty$-operad as in~\cite[§5]{Heu20}.
  For the set~$\underline{n} = \{1, 2, \dots, n\}$ one can define the poset~$\widehat{\PC}_n$ of partitions of~$\underline{n}$:
  An element~$\lambda$ of~$\widehat{\PC}_n$ is an equivalence relation on~$\underline{n}$.
  We write~$\lambda \leq \lambda'$ if~$\lambda$ is finer than~$\lambda'$, \ie if~$x \sim_{\lambda} y$ then~$x \sim_{\lambda'} y$. 
  The minimal equivalence relation~$\widehat{0}$ is given by~$x \sim y$ if~${x = y}$, and the maximal equivalence relation~$\widehat{1}$ is given by~$x \sim y$ for every pair~$(x, y)$ of elements in~$\underline{n}$.
  We define the subset
  \[
  \PC_n \coloneqq \widehat{\PC}_n \setminus \left\{\widehat{0}, \widehat{1}\right\} \subseteqq \widehat{\PC}_n.
  \]
  It inherits its poset structure from~$\widehat{\PC}_n$.
  
  The \emph{spectral Lie~$\infty$\nobreakdash-operad~$\infopLie$} is an~$\infty$\nobreakdash-operad with values in the~$\infty$\nobreakdash-category~$\infSp$ of spectra, defined as the Koszul dual~(see~\cref{ex:KD-Lie-Com}) of the commutative~$\infty$-operad in~\cite[Corollary 8.8]{Chi12}.
  Denote the unreduced suspension by~$\diamond$.
  We have
  \[
  {\infopLie(r) \simeq \left(\susp\left(\lvert \PC_r\rvert^{\diamond}\right)\right)^{\vee}},
  \] 
  for every~$r \geq 1$, and~$\infopLie(0) = \pt$~(the zero spectrum).
  Denote the 1-categorical Lie operad with values in abelian groups by~$\opLie$, see for example~\cite[Definition~4.10]{AB21}.
  There exists an~isomorphism
  \[
  \opLie(r) \cong \widetilde{\Ho}_0\left(\susp^{r-1}\infopLie(r); \ZZ\right).
  \]
  of abelian groups. 
  That is, one can consider~$\infopLie$ as a spectral enhancement of~$\opLie$.
\end{passage}

The~$\infty$\nobreakdash-category~$\infSp$ is the presentable stable~$\infty$\nobreakdash-category freely generated by the sphere spectrum~$\SS$, see~\cite[Corollary~1.4.4.6]{HA}.
  Thus, for every presentable stable symmetric monoidal~$\infty$\nobreakdash-category~$\Cca$, there exists a symmetric monoidal functor~${F \colon \infSp \to \Cca}$ in~$\infPrl$, unique up to contractible choice;~$F$ is determined by its evaluation~$F(\SS) \simeq \munit_{\Cca}$.
Again by~\cite[Proposition~5.2.5.1]{ShiTh} we obtain an induced functor~$F \colon \infOpd(\infSp) \to \infOpd(\Cca)$.
Let~$\Oca$ be a \emph{spectral~$\infty$\nobreakdash-operad}, \ie an~$\infty$\nobreakdash-operad with values in~$\infSp$.
Then the~$\infty$\nobreakdash-category~$\infAlg_{\Oca}(\Cca)$ of \emph{$\Oca$\nobreakdash-algebras in~$\Cca$} is defined as the~$\infty$\nobreakdash-category~$\infAlg_{F(\Oca)}(\Cca)$.

\begin{definition}
  Let~$\Cca$ be a presentable stable symmetric monoidal~$\infty$\nobreakdash-category.
  A \emph{spectral Lie algebra} in~$\Cca$ is an algebra over the spectral Lie~$\infty$\nobreakdash-operad~$\infopLie$.
\end{definition}

\begin{remark}
  We refer the interested reader to~\cite[Proposition~5.2]{Cam20} and~\cite[§3.2]{Kja18} for some discussions about the Lie bracket and the Jacobi identity relation for spectral Lie~algebras.
\end{remark}

Now we can reformulate Quillen's Lie algebraic model for simply connected rational \htypes using the language of~$\infty$\nobreakdash-categories as follows.

\begin{passage}[The~$\infty$\nobreakdash-category~$\infD(\QQ)$]\label{para:infinity-rational-chain}
  One can equip the 1-category~$\CatCh_{\QQ}$ of rational chain complexes with a symmetric monoidal combinatorial projective model structure which is simplicial enriched, see~\cite[Remark~1.2.3.21, Proposition~7.1.2.8, Proposition~7.1.2.11, Construction~1.3.1.13]{HA}.
  
  The~$\infty$\nobreakdash-category~$\infD(\QQ)$ of the derived category of the rational numbers~$\QQ$ is defined as the underlying~$\infty$\nobreakdash-category of the model category~$\CatCh_{\QQ}$, \ie the localisation of~$\CatCh_{\QQ}$ at quasi-isomorphisms.
  Equivalently, the~$\infty$\nobreakdash-category~$\infD(\QQ)$ is a localisation of the~$\infty$\nobreakdash-category of the differential graded nerve of~$\CatCh_{\QQ}$, see \cite[Definition~1.3.5.8, Proposition~1.3.5.15]{HA}.
  We will need the following properties of~$\infD(\QQ)$.
  \begin{enumerate}
    \item The~$\infty$\nobreakdash-category~$\infD(\QQ)$ is a presentable symmetric monoidal~$\infty$\nobreakdash-category, see \cite[Proposition~1.3.4.22]{HA}.
    \item There exists an equivalence
      \[
      \infD(\QQ) \xrightarrow{\sim} \infMod_{\EM\QQ},
      \]
      of symmetric monoidal~$\infty$\nobreakdash-categories, where~$\infMod_{\EM\QQ}$ is the~$\infty$\nobreakdash-category of~$\EM\QQ$\nobreakdash-module spectra~(it is symmetric monoidal by~\cite[Theorem~4.5.2.1]{HA}).
      See~\cite[Theorem~7.1.2.13]{HA}.
  \end{enumerate}
\end{passage}

Let~$\CatAlg_{\opLie}(\CatCh_{\QQ})$ denote the model category of rational differential graded Lie algebras whose set of weak equivalences is denoted by~$W_{\opLie}$; the model structure is induced from the projective model structure of~$\CatCh_{\QQ}$, see~\cite[§4]{Hau}.

\begin{theorem}[{\cite[Corollary~4.11]{Hau}}]
  There exists an equivalence
  \[
  \infAlg_{\infopLie}(\infD(\QQ)) \xrightarrow{\sim} \CatAlg_{\opLie}(\CatCh_{\QQ})[W_{\opLie}^{-1}]
  \]
  of~$\infty$\nobreakdash-categories.
\end{theorem}

\begin{theorem}[Quillen~\cite{Qui69}]\label{thm:Quillen}
  The~$\infty$\nobreakdash-category~$\infHType_{\QQ}^{\geq 2}$ of pointed simply-connected rational \htypes is equivalent to the~$\infty$\nobreakdash-category~$\infAlg_{\infopLie}\left(\infD(\QQ)\right)^{\geq 1}$ of connected differential graded Lie algebras over~$\QQ$.
\end{theorem}

\begin{remark}
  Consider the composition
  \[
  \Lca \colon \infHType_{\QQ}^{\geq 2} \underset{\sim}{\xrightarrow{\BK_{0}}} \infAlg_{\infopLie}\left(\infD(\QQ)\right)^{\geq 1} \xrightarrow{\frgt_{\infopLie}} \infD(\QQ) \xrightarrow{\sim} \infMod_{\EM\QQ} \xrightarrow{\frgt} \infSp_{\QQ}.
  \]
  Given~$X \in \infHType_{\QQ}^{\geq 2}$, it is shown in~\cite{Qui69} that~$\pi_{\bullet}^{\st}\left(\Lca(X)\right)$ is isomorphic to~$\pi_{\bullet}(X) \otimes \QQ$.
  In words, the stable homotopy group of the underlying rational spectrum of the Lie algebra~$\BK_{0}(X)$ computes the rational homotopy group of~$X$.
  Recall from~\cref{rmk:BK-computes-vh-groups} that the Bousfield--Kuhn functor~$\BK_{h}$, for~$h \geq 1$, has a similar property as~$\Lca$.
\end{remark}

In his work~\cite{Heu21} Heuts generalises the aforementioned Quillen's theorem to~$v_h$\nobreakdash-periodic \htypes.

\begin{theorem}[Heuts]\label{thm:Heuts-vh-Lie}
  There exists an equivalence
  \[
  \infHType_{v_{h}} \simeq \infAlg_{\infopLie}\left(\infSp_{\Tel(h)}\right)
  \]
  of~$\infty$\nobreakdash-categories, for every natural number~$h \geq 1$. 
\end{theorem}

\begin{passage}
  In the proof of the theorem there are the following two main steps.
  \begin{enumerate}
    \item The Bousfield--Kuhn adjunction~(see~\cref{thm:BK-adjunction})
      \[
      \lBK_{h} \colon \infSp_{\Tel(h)} \rightleftarrows \infHType_{v_{h}}  \cocolon \BK_{h}
      \]
      is monadic, see~\cite{EHMM19}.
    \item The monad~$\BK \circ \lBK$ is equivalent to the arity-wise~$\Tel(h)_{\bullet}$\nobreakdash-localisation of the monad~$\T_{\infopLie}$ associated with spectral Lie~$\infty$\nobreakdash-operad, see~\cite[Theorem~4.21]{Heu21}.
  \end{enumerate}
  In other words, we have the following commutative diagram
  \[
    \begin{tikzcd}
    \infHType_{v_{h}} \arrow[rr, "\BK_{h}"] \arrow[rd, "\sim", "\BK_{h}'"'] && \infSp_{\Tel(h)} \\
    & \infAlg_{\infopLie}\left(\infSp_{\Tel(h)}\right) \arrow[ru, "\frgt_{\infopLie}"']
    \end{tikzcd}
  \]
  of~$\infty$\nobreakdash-categories.
  In particular, under the equivalence~$\infHType_{v_h} \simeq \infAlg_{\infopLie}\left(\infSp_{\Tel(h)}\right)$, we can view the Bousfield--Kuhn adjunction~$\lBK_h \dashv \BK_{h}$ equivalently as the adjunction~$\free_{\infopLie} \dashv \frgt_{\infopLie}$ of~$\Tel(h)$\nobreakdash-local spectral Lie algebras, see~\cref{ex:free-forg-adjunction}.
\end{passage}

\begin{passage}[Stabilisation of the~$\infty$\nobreakdash~category~$\infHType_{v_{h}}$]\label{para:vh-stabilisation}
  Let~$\Cca$ be a presentable stable symmetric monoidal~$\infty$\nobreakdash-category.
  The~$\infty$\nobreakdash-operad~$\infopLie$ is equipped with two morphisms
  \[
  \infopTriv_{\infSp} \xrightarrow{i} \infopLie \xrightarrow{a} \infopTriv_{\infSp}
  \]
  of~$\infty$\nobreakdash-operads where~$\infopTriv_{\infSp}(1) \to \infopLie(1)$ and~$\infopLie(1) \to \infopTriv_{\infSp}(1)$ are both the identity morphisms.
  The morphisms~$i$ and~$a$ induce the following adjunctions:  
  \begin{equation}\label{eq:two-adj-Lie-alg}
    \Cca \underset{\frgt_{\infopLie}}{\overset{\free_{\infopLie}}{\rightleftarrows}} \infAlg_{\infopLie}(\Cca)  \underset{\triv_{\infopLie}}{\overset{\indc_{\infopLie}}{\rightleftarrows}}  \Cca.
  \end{equation}
  In particular, we have~$\indc_{\infopLie} \circ \free_{\infopLie} \simeq \id \simeq \frgt_{\infopLie} \circ \triv_{\infopLie}$, since~$a \circ i \simeq \id$.
  Furthermore, by~\cite[Theorem~4.3]{Heu20}, the adjunction~$\indc_{\infopLie} \dashv \triv_{\infopLie}$ exhibits~$\Cca$ as the stabilisation of the~$\infty$\nobreakdash-category~$\infAlg_{\infopLie}(\Cca)$.
  
  Now let~$\Cca$ be the~$\infty$\nobreakdash-category~$\infSp_{\Tel(h)}$ of~$\Tel(h)$\nobreakdash-local spectra. 
  Under the equivalence of~\cref{thm:Heuts-vh-Lie} the two adjunctions~\eqref{eq:two-adj-Lie-alg} correspond to the following two adjunctions
  \[
  \infSp_{\Tel(h)} \underset{\BK_{h}}{\overset{\lBK_{h}}{\rightleftarrows}} \infHType_{v_h}  \underset{\Loop^{\infty}_{v_{h}}}{\overset{\susp^{\infty}_{v_{h}}}{\rightleftarrows}}  \infSp_{\Tel(h)}
  \]
  associated with~$\infHType_{v_h}$, where the adjunction~$\susp^{\infty}_{v_{h}} \dashv \Loop^{\infty}_{v_{h}}$ exhibits~$\infSp_{\Tel(h)}$ as the stabilisation of~$\infHType_{v_{h}}$, see~\cite[§1.4]{HA}.
  In particular, we have an equivalence~$
  \susp^{\infty}_{v_{h}} \simeq \Loc_{\Tel(h)} \circ \susp^{\infty}_{\infHType_{\ast}}$ of functors, where we abuse the notation~$\susp^{\infty}_{\infHType_{\ast}}$ to denote the restriction of the suspension spectrum functor to the~$\infty$\nobreakdash-subcategory~$\infHType_{v_h}$.
  See~\cite[§3.3]{Heu21} for more details.
\end{passage}

\subsection{Higher enveloping algebras in~\texorpdfstring{$\infSp_{\Tel(h)}$}{Sp_{T(h)}}}\label{sec:higher-enveloping-Th}

To every Lie algebra over a field~$k$ we can associate its universal enveloping algebra, which is an augmented associated algebra over~$k$.
In this and the next section we will discuss a similar construction for spectral Lie algebras.
Let~$n \in \NN$ and consider Lurie's~$\infty$\nobreakdash-operad~$\infopE_n^{\otimes}$ modelled on the topological little~$n$\nobreakdash-disks operad~$\opE_n$, see~\cite[§4.1]{Fre17-1} and~\cite[§5.1]{HA}.
In this section we will construct a commutative diagram
\begin{equation}\label{diag:Lie-En-Th}
\begin{tikzcd}[row sep = huge, column sep = small]
  & \infAlg_{\infopLie}\left(\infSp_{\Tel(h)}\right) \arrow[rrrrd] \arrow[rrrd, "\UE_1"']  \arrow[rd] \arrow[d, "\UE_n"'] \arrow[ld, ] &             &             &   \\
  \cdots \arrow[r] & \infAlg_{\infopE_{n}}^{\nut}\left(\infSp_{\Tel(h)}\right) \arrow[r, "\rB_n"']                                                & \infAlg_{\infopE_{n-1}}^{\nut}\left(\infSp_{\Tel(h)}\right) \arrow[r, "\rB_{n-1}"'] & \cdots \arrow[r] & \infAlg_{\infopE_{1}}^{\nut}\left(\infSp_{\Tel(h)}\right) \arrow[r, "\rB_{1}"'] & \infSp_{\Tel(h)},
\end{tikzcd}
\end{equation}
in~$\infPrl$, relating spectral Lie algebras and non-unital~$\infopE_n$\nobreakdash-algebras in~$\infSp_{\Tel(h)}$.
We consider~$\UE_1$ as the universal enveloping algebra functor and call~$\UE_n$ the~\emph{$n$\nobreakdash-th higher enveloping algebra} functor, inspired by the work~\cite{Knu18} of Knudsen.
From~\eqref{diag:Lie-En-Th} we obtain an adjunction
\begin{equation}\label{eq:U-T-infty-adjunction}
\UE_{\infty} \colon \infAlg_{\infopLie}((\infSp_{\Tel(h)}) \rightleftarrows \varprojlim_{n}\infAlg_{\infopE_n}^{\nut}(\infSp_{\Tel(h)}) \cocolon \rUE_{\infty}.
\end{equation}
by the universal property of limits of~$\infty$\nobreakdash-categories.
In~\cref{sec:fully-faithful} we will show that the functor~$\UE_{\infty}$ is fully faithful~(see~\cref{thm:U-infty-ff}).

\begin{passage}[Non-unital and augmented~$\infopE_n$\nobreakdash-algebras]
  Recall from~\cref{thm:sym-fun-algebra} that we can associate to the~$\infty$\nobreakdash-operad~$\infopE_n^{\otimes}$ an object~$\infopE_n$ in~$\infOpd(\infHType)$.
  For~$n \in \NN$ the~$\infty$\nobreakdash-operad~$\infopE_n$ is a \emph{unital} operad, \ie~$\infopE_n(0)$ is equivalent to the monoidal unit~$\pt$ of~$\infHType$~(with the Cartesian monoidal structure).
  Let~$\Cca$ be a presentable symmetric monoidal~$\infty$\nobreakdash-category. 
  Then an~$\infopE_n$\nobreakdash-algebra~$X \in \Cca$ admits a unit map~$\munit_{\Cca} \to X$, where~$\munit_{\Cca}$ denotes the symmetric monoidal unit of~$\Cca$.
  We consider the following two variants of~$\infopE_n$\nobreakdash-algebras: Augmented~$\infopE_n$\nobreakdash-algebras and non-unital~$\infopE_n$\nobreakdash-algebras.
  
  An \emph{augmented~$\infopE_n$\nobreakdash-algebra} in~$\Cca$ is an~$\infopE_n$\nobreakdash-algebra~$X$ in~$\Cca$ together with a augmentation map~$X \to \munit_{\Cca}$.
  More precisely, the~$\infty$\nobreakdash-category of augmented~$\infopE_n$\nobreakdash-algebra is defined as
  \[
  \infAlg_{\infopE_n}^{\aug}(\Cca) \coloneqq \infAlg_{\infopE_n}\left(\Cca\right)_{/ \munit_{\Cca}}.
  \]
  Moreover, the canonical functor~$\Cca_{/ \munit_{\Cca}} \to \Cca$ is symmetric monoidal and thus induces an~equivalence
  \begin{equation}\label{eq:aug-alg-cat-two-way}
    \infAlg_{\infopE_n}(\Cca_{/ \munit_{\Cca}}) \to \infAlg_{\infopE_n}\left(\Cca\right)_{/ \munit_{\Cca}}
  \end{equation}
  of~$\infty$\nobreakdash-categories, by the universal property of the~$\infty$\nobreakdash-overcategory~$\infAlg_{\infopE_n}\left(\Cca\right)_{/ \munit_{\Cca}}$ and the construction of the symmetric monoidal structure on~$\Cca_{/ \munit_{\Cca}}$, see~\cite[Definition~2.2.2.1 and~Notation~2.2.2.3]{HA}.
   Whenever we denote~$\infAlg_{\infopE_n}(\Cca_{/ \munit_{\Cca}})$ by~$\infAlg_{\infopE_n}^{\aug}(\Cca)$, we denote the forgetful functor~$\infAlg_{\infopE_n}(\Cca_{/ \munit_{\Cca}}) \to \Cca_{/ \munit}$ by 
  \[
  \frgt^{\aug}_{\infopE_n} \colon \infAlg_{\infopE_n}^{\aug}(\Cca) \to \Cca_{/ \munit}.
  \]
  
  The \emph{deunitalisation} of~$\infopE_n$ is a \emph{non-unital}~$\infty$\nobreakdash-operad~$\infopE_n^{\nut}$ where~$\infopE_n^{\nut}(r) \simeq \infopE_n(r)$ for every~$r \neq 0$ and~$\infopE_n^{\nut}(0)$ is equivalent to the initial object~$\emptyset$ of~$\infHType$.
  See~\cite[Definition~5.3.2.11]{ShiTh} for a precise definition (via symmetric sequences).
  Denote the~$\infty$\nobreakdash-category of~$\infopE_n^{\nut}$\nobreakdash-algebras by~$\infAlg_{\infopE_n}^{\nut}(\Cca)$.
  We denote the forgetful functor~$\infAlg_{\infopE_n}^{\nut}(\Cca) \to \Cca$ by~$\frgt_{\infopE_n}^{\nut}$.
  Note that~$X \in \infAlg_{\infopE_n}^{\nut}(\Cca)$ does not necessarily admit a unit map~$\munit_{\Cca} \to X$, in comparison with objects in~$ \infAlg_{\infopE_n}(\Cca)$.
  
  Assume in addition that~$\Cca$ is stable.
  The \emph{augmentation ideal} functor~$\widetilde{\augideal}_{\infopE_n}$ is defined as the following composition
  \[
  \widetilde{\augideal}_{\infopE_n} \colon \infAlg_{\infopE_n}^{\aug}(\Cca) \xrightarrow{\frgt_{\infopE_n}^{\aug}} \Cca_{ /\munit_{\Cca}} \xrightarrow{\fib} \Cca, 
  \]
  where~$\fib$ denotes the composition~$\Cca_{/ \munit_{\Cca}} \to \infFun(\infSplx^{1}, \Cca) \xrightarrow{\text{fibre}} \Cca$, assigning to an object~${X \to \munit_{\Cca}}$ of~$\Cca_{/ \munit_{\Cca}}$ its fibre.
\end{passage}

\begin{example}\label{ex:nu-unital-aug-equiv}
  Let~$\Cca$ be a presentable stable symmetric monoidal~$\infty$\nobreakdash-category.
  We have the following equivalences of~$\infty$\nobreakdash-categories:
  \begin{align*}
    \infAlg_{\infopE_0}(\Cca) &\simeq \Cca_{\munit_{\Cca}/} \\
    \infAlg_{\infopE_0}^{\nut}(\Cca) &\simeq \Cca \\
    \infAlg_{\infopE_0}^{\nut}(\Cca) &\simeq \infAlg_{\infopE_0}^{\aug}.
  \end{align*}
  The first equivalence is proven in~\cite[Proposition~2.1.3.9]{HA}~(and it also holds without the assumption that~$\Cca$ is stable);
  The equivalence~$\infopE_0^{\nut} \simeq \infopTriv$ of~$\infty$\nobreakdash-operads induces the second equivalence.
  Note that the augmentation ideal functor~$\widetilde{\augideal}_{\infopE_0}$ and~the trivial augmentation functor~$\widetilde{\trivaug}_{\infopE_0} \colon \Cca \to \infAlg_{\infopE_0}^{\aug}, \ X \mapsto X \oplus \munit_{\Cca}$ are inverse to each other.
  This fact and the second equivalence imply the third equivalence.
\end{example}
 
 \begin{proposition}\label{cor:equiv-non-unital-aug}
   Let~$\Cca$ be a presentable stable symmetric monoidal~$\infty$\nobreakdash-category and let~$n \geq 1$ be a natural number.
   \begin{enumerate}
     \item The augmentation ideal functor~$\widetilde{\augideal}_{\infopE_n}$ admits a left adjoint~$\widetilde{\trivaug}_{\infopE_n}$, called the \emph{trivial augmentation functor}.
     \item The adjunction~$\widetilde{\trivaug}_{\infopE_n} \dashv \widetilde{\augideal}_{\infopE_n}$ is monadic and the monad~$\widetilde{\augideal}_{\infopE_n} \circ \widetilde{\trivaug}_{\infopE_n}$  is equivalent to the monad~$\T_{\infopE_n^{\nut}}$ associated with the~$\infty$\nobreakdash-operad~$\infopE_n^{\nut}$.
     \item There exists the following commutative diagram
        \begin{equation}\label{eq:equiv-non-unital-aug}
          \begin{tikzcd}
          \infAlg_{\infopE_n}^{\aug}(\Cca) \arrow[d, "\frgt_{\infopE_n}^{\aug}"'] \arrow[r, "\sim"] &  \infAlg_{\infopE_n}^{\nut}(\Cca) \arrow[d, "\frgt_{\infopE_n^{\nut}}"] \\
          \infAlg_{\infopE_0}^{\aug}\left(\Cca\right)  \arrow[r, "\augideal_{0}"', "\sim"] & \infAlg_{\infopE_0}^{\nut}(\Cca)
          \end{tikzcd}
        \end{equation}
        of~$\infty$\nobreakdash-categories, where~$\augideal_{0}$ denotes the composition~$\infAlg_{\infopE_0}^{\aug}\left(\Cca\right) \xrightarrow{\widetilde{\augideal}_{\infopE_0}} \Cca \xrightarrow{\sim} \infAlg_{\infopE_0}^{\nut}(\Cca)$.
   \end{enumerate}
 \end{proposition}
 
 \begin{proof}
   See~\cite[Proposition~6.1.0.9, Corollary~6.1.0.10]{ShiTh} or~\cite[Proposition~7.3.4.5]{HA}
 \end{proof}
 
\begin{notation}
   In the situation of~\cref{cor:equiv-non-unital-aug} we denote the functors left adjoint to~$\frgt_{\infopE_n}^{\aug}$ and~$\frgt_{\infopE_n}^{\nut }$ by~$\free_{\infopE_n}^{\aug}$ and~$\free_{\infopE_n^{\nut}}$ respectively.
   The upper horizontal equivalence in~\eqref{eq:equiv-non-unital-aug} is denoted by
  \[
  \trivaug_{\infopE_n} \colon  \infAlg_{\infopE_n}^{\nut}(\Cca) \overset{\sim}{\rightleftarrows}  \infAlg_{\infopE_n}^{\aug}(\infopE_n) \cocolon \augideal_{\infopE_n} 
  \]
  On the underlying objets~$\trivaug_{\infopE_n}$ is given by the functor 
  \[
  \trivaug_{0} \colon  \Cca \simeq \infAlg_{\infopE_0}^{\nut}\left(\Cca\right) \to  \infAlg_{\infopE_0}^{\aug}\left(\Cca\right), \ X \mapsto X \oplus \munit_{\Cca},
  \]
  inverse to the functor~$\augideal_{0}$.
\end{notation}

\begin{passage}[The iterated Bar construction on~$\infAlg_{\infopE_n}^{\aug}\left(\Cca\right)$]\label{para:Bar-constr-En}
  We recall some facts of this construction from~\cite[§5.2]{HA}.
  Let~$\Cca$ be a symmetric monoidal~$\infty$\nobreakdash-category admitting geometric realisations of simplicial objects and totalisations of cosimplicial objects.
  \begin{enumerate}
    \item There exists a small-colimit-preserving functor
    \[
    \Bac_n \colon \infAlg_{\infopE_n}^{\aug}(\Cca) \to \infAlg_{\infopE_0}^{\aug}(\Cca),
    \]
    called the \emph{iterated Bar construction}
    \item Applying the Additivity Theorem of~$\infopE_n$\nobreakdash-algebras~\cite[Theorem~5.1.2.2]{HA} to~\eqref{eq:aug-alg-cat-two-way} we obtain equivalences
      \[
      \infAlg_{\infopE_n}^{\aug}(\Cca) \simeq \left(\infAlg_{\infopE_n}(\Cca)\right)_{/ \munit_{\Cca}} \simeq \infAlg_{\infopE_n}(\Cca_{/ \munit_{\Cca}}).
      \]
      of~$\infty$\nobreakdash-categories.
      Under these equivalences the functor~$\Bac_{n}$ is equivalent to the composition
      \[
      \infAlg_{\infopE_n}^{\aug}(\Cca) \simeq \infAlg_{\infopE_1}^{\aug}\left(\infAlg_{\infopE_{n-1}}(\Cca)\right) \xrightarrow{\Bac_{1}} \infAlg_{\infopE_{n-1}}^{\aug}(\Cca) \xrightarrow{\Bac_{n-1}} \infAlg_{\infopE_0}^{\aug}(\Cca).
      \]
    \item The functor~$\Bac_n$ induces an adjunction
    \[
    \Bac_{n} \colon \infAlg_{\infopE_n}^{\aug}(\Cca) \rightleftarrows \infAlg_{\infopE_0}^{\aug}(\Cca) \cocolon \Cobar_{n}.
    \]
  \end{enumerate}
    Assume now that for every weakly contractible simplicial set~$K$, the~$\infty$\nobreakdash-category~$\Cca$ admits~$K$\nobreakdash-indexed colimits and the symmetric monoidal product of~$\Cca$ preserves~$K$\nobreakdash-colimits in each variable.
  Then we have
    \begin{enumerate}[resume]
    \item For every object~$X \in \infAlg_{\infopE_0}^{\aug}(\Cca)$ there exists an equivalence
    \[
    \Bac_{n}\left(\free_{\infopE_n}^{\aug}(X)\right) \simeq \susp^{n}_{\infAlg_{\infopE_0}^{\aug}(\Cca)}(X).
    \]
  \end{enumerate}
\end{passage}

\begin{passage}[The functor~$\rB_{n+1}$]\label{para:non-unital-bar}
  The functor~$\rB_{n+1} \colon \infAlg_{\infopE_{n+1}}^{\nut}(\infSp_{\Tel(h)}) \to \infAlg_{\infopE_n}^{\nut}(\infSp_{\Tel(h)})$ from diagram~\eqref{diag:Lie-En-Th} is defined as the following composition
  \[
  \infAlg_{\infopE_{n+1}}^{\nut}\left(\infSp_{\Tel(h)}\right) \xrightarrow{\trivaug_{\infopE_n}} \infAlg_{\infopE_{n+1}}^{\aug}\left(\infSp_{\Tel(h)}\right) \xrightarrow{\Bac_{1}} \infAlg_{\infopE_{n}}^{\aug}\left(\infSp_{\Tel(h)}\right) \xrightarrow{\augideal_{\infopE_n}} \infAlg_{\infopE_{n}}^{\nut}\left(\infSp_{\Tel(h)}\right).
  \]
  In other words, we can regard~$\rB_{n}$ as the iterated Bar construction, under the equivalence between the augmented and the non-unital~$\infopE_n$\nobreakdash-algebras, see~\cref{ex:nu-unital-aug-equiv}.
  Note that we can replace~$\infSp_{\Tel(h)}$ by any presentable symmetric monoidal~$\infty$\nobreakdash-category in this construction.
\end{passage}

\begin{corollary}\label{cor:Bn-free-commutes}
  Let~$\Cca$ be a presentable symmetric monoidal~$\infty$\nobreakdash-category.
  There exists the following commutative diagram of~$\infty$\nobreakdash-categories:
  \[
  \begin{tikzcd}
    \infAlg_{\infopE_{n+1}}^{\nut}(\Cca) \arrow[r, "\rB_n"] & \infAlg_{\infopE_n}^{\nut}(\Cca)\\
    \Cca \arrow[u, "\free_{\infopE_{n+1}^{\nut}}"] \arrow[r, "\susp_{\Cca}"', "\sim"] & \Cca. \arrow[u, "\free_{\infopE_n^{\nut}}"']
  \end{tikzcd}
  \]
\end{corollary}

\begin{proof}
  This follows from~\cref{para:Bar-constr-En}.\rom{2} and~\rom{4} and~\cref{cor:equiv-non-unital-aug}.
\end{proof}

Now we work towards defining the functor~$\UE_{n} \colon \infAlg_{\infopLie}\left(\infSp_{\Tel(h)}\right) \to \infAlg_{\infopE_n}^{\nut}(\infSp_{\Tel(h)})$, for every natural number~$n$.
We work in the following situation.

\begin{situation}
  Let~$\Cca$ be a presentable stable symmetric monoidal~$\infty$\nobreakdash-category.
  From now on we consider~$\infAlg_{\infopLie}(\Cca)$ as a pointed cartesian symmetric monoidal~$\infty$\nobreakdash-category.
  The zero object of~$\infAlg_{\infopLie}(\Cca)$ is the zero object~$\mathbbl{0}_{\Cca}$ of~$\Cca$ equipped with the trivial Lie algebra structure.
  It is also the symmetric monoidal unit of~$\infAlg_{\infopLie}(\Cca)$.
  Thus, we obtain the following equivalences of~$\infty$\nobreakdash-categories:
  \begin{align*}
    \infAlg_{\infopE_0}^{\aug}\left(\infAlg_{\infopLie}(\Cca)\right)  &\simeq \infAlg_{\infopLie}(\Cca) \\
    \infAlg_{\infopE_n}^{\aug}\left(\infAlg_{\infopLie}(\Cca)\right)  &\simeq \infAlg_{\infopE_n}\left(\infAlg_{\infopLie}(\Cca)\right)
  \end{align*}
\end{situation}

\begin{proposition}\label{prop:Loop-fac-En}
  The~$n$\nobreakdash-fold loop functor~$\Loop_{\infopLie}^n$ of the~$\infty$\nobreakdash-category~$\infAlg_{\infopLie}(\Cca)$ admits the following~factorisation:
  \begin{equation}\label{eq:loop-Lie-En}
    \begin{tikzcd}
      \infAlg_{\infopLie}(\Cca) \arrow[rr, "\Loop_{\infopLie}^n"] \arrow[rd, "\widetilde{\Loop}_{\infopLie}^n"'] && \infAlg_{\infopLie}(\Cca) \\
      & \infAlg_{\infopE_n}\left(\infAlg_{\infopLie}(\Cca)\right), \arrow[ru, "\frgt_{\infopE_n}"']
    \end{tikzcd}
  \end{equation}
  for every~$n \in \NN$.
  Furthermore, the functor~$\widetilde{\Loop}^n_{\infopLie}$ is an equivalence of~$\infty$\nobreakdash-categories.
\end{proposition}

\begin{proof}
  Since~$\infAlg_{\infopLie}(\Cca)$ is equipped with the cartesian symmetric monoidal structure, the~$n$-fold suspension endofunctor~$\susp^{n}_{\infopLie}$ on~$\infAlg_{\infopLie}(\Cca)$ admits the following decomposition
  \begin{equation}\label{eq:susp-lie-decom}
    \infAlg_{\infopLie}(\Cca) \simeq \infAlg_{\infopE_0}^{\aug}\left(\infAlg_{\infopLie}(\Cca)\right) \xrightarrow{\free_{\infopE_n}^{\aug}} \infAlg_{\infopE_n}^{\aug}\left(\infAlg_{\infopLie}(\Cca)\right) \xrightarrow{\Bac_{n}}  \infAlg_{\infopLie}(\Cca),
  \end{equation}
  by~\cite[Example~5.2.2.4]{HA}.
  Taking right adjoints to the functors in the above decomposition, we obtain the factorisation of~$\Loop_{\infopLie}^{n}$ as in~\eqref{eq:loop-Lie-En}.
  In other words, the functor~$\widetilde{\Loop}^{n}_{\infopLie}$ denotes the cobar construction, right adjoint to the Bar construction~$\Bac_n$ in~\eqref{eq:susp-lie-decom}.
  
  To finish the proof it remains to show that the unit and counit natural transformations of the adjunction~$\Bac_n \dashv \widetilde{\Loop}^n_{\infopLie}$ are equivalences of functors.
  Since the forgetful functors are conservative, it suffices to show that 
  \[
  \frgt_{\infopLie} \circ \frgt_{\infopE_n} \to \frgt_{\infopLie}  \circ \frgt_{\infopE_n} \circ \widetilde{\Loop}^n_{\infopLie} \circ \Bac_{n}, \text{ and}
  \]
  \[
  \frgt_{\infopLie} \circ  \Bac_{n} \circ \widetilde{\Loop}^n_{\infopLie} \to \frgt_{\infopLie}
  \]
  are equivalences.
  Both of these equivalences follow if we provide an equivalence
  \begin{equation}\label{eq:bar-frgt-commutes}
  \frgt_{\infopLie}  \circ \Bac_{n} \simeq \susp_{\Cca}^{n} \circ \frgt_{\infopLie} \circ \frgt_{\infopE_n},
  \end{equation}
  since the suspension~$\susp_{\Cca}$ and the loop~$\Loop_{\Cca}$ functors on~$\Cca$ are auto-equivalences.
  The line~\eqref{eq:bar-frgt-commutes} holds by the following arguments:
  \begin{enumerate}
  \item It suffices to show this for~$n = 1$ because~$\Bac_n$ is the~$n$\nobreakdash-fold iterated Bar construction~(see~\cref{para:Bar-constr-En}).
  \item The forgetful functors commute with geometric realisations, and~$\Bac_1(X)$ is equivalent to the geometric realisation of~$\Bac(\mathbbl{0}_{\Cca}, \frgt_{\infopE_n}(X), \mathbbl{0}_{\Cca})$ in~$\infAlg_{\infopLie}(\Cca)$.
  \item Consider~$\Cca$ as a cocartesian symmetric monoidal~$\infty$\nobreakdash-category.
  By~\cite[Corollary~2.4.3.10]{HA} we have an equivalence~$\Cca \simeq \infAlg_{\infopE_{\infty}}(\Cca)$ of~$\infty$\nobreakdash-categories.
  Moreover, the geometric realisation of~$\Bac\left(\mathbbl{0}_{\Cca}, (\frgt_{\infopLie} \circ \frgt_{\infopE_n})(X), \mathbbl{0}_{\Cca} \right)$ becomes equivalent to the Bar construction of the commutative algebra~$(\frgt_{\infopLie} \circ \frgt_{\infopE_n})(X)$ in~$\Cca$.
  Therefore, we obtain the equivalence in~$\Cca$
  \[
  \varinjlim\left(\Bac\left(\mathbbl{0}_{\Cca}, (\frgt_{\Oca} \circ \frgt_{\infopE_n})(X), \mathbbl{0}_{\Cca} \right)\right) \simeq \susp_{\Cca}((\frgt_{\Oca} \circ \frgt_{\infopE_n})(X)),
  \]
  by~\cite[Example~5.2.2.4]{HA}. \qedhere
  \end{enumerate}
\end{proof}

\begin{remark}
  From the proof we see that \Cref{prop:Loop-fac-En} still holds if we replace~$\infopLie$ by any other non-unital~$\infty$\nobreakdash-operad~$\Oca$.
\end{remark}

\begin{proposition}\label{prop:indc-Lie-Th-+-sym-mon}
  The functor 
  \[
  \indc_{\infopLie}^{+} \colon \infAlg_{\infopLie}\left(\infSp_{\Tel(h)}\right) \xrightarrow{\indc} \infSp_{\Tel(h)} \xrightarrow{\widetilde{\trivaug}_{\infopE_0}} \infAlg_{\infopE_0}^{\aug}\left(\infSp_{\Tel(h)}\right)
  \]
  is symmetric monoidal with respect to the standard symmetric monoidal structure on~$\infSp_{\Tel(h)}$ and the cartesian monoidal structure on~$\infAlg_{\infopLie}(\infSp_{\Tel(h)})$.
\end{proposition}

\begin{proof}
  Let~$\susp^{\infty}_{+, v_{h}}$ denote the following composition
  \[
  \infHType_{v_{h}} \hookrightarrow \infHType \xrightarrow{\Loc_{\Tel(h)} \circ \susp^{\infty}_{+}} \infSp_{\Tel(h)}.
  \]
  The functor~$\susp^{\infty}_{+, v_{h}}$ is symmetric monoidal with respect to the cartesian monoidal structure on~$\infHType_{v_{h}}$~(given by taking the product of the underlying pointed \htypes) and the standard symmetric monoidal structure on~$\infSp_{\Tel(h)}$ given by the smash product of the underlying spectra, because each functor in the composition is, see~\cite[Proposition~2.2.1.9]{HA}. 
  
  Furthermore, we can lift~$\susp^{\infty}_{+, v_{h}}$ to a symmetric monoidal functor~
  \[
  \susp^{\infty}_{+, v_{h}} \colon \infHType_{v_{h}} \to \infAlg_{\infopE_0}^{\aug}(\infSp_{\Tel(h)});
  \] 
  every object~$X \in \infHType_{v_{h}}$ is pointed and the functor~$\susp^{\infty}_{+, v_{h}}$ assigns to the canonical maps~${\pt \to X \to \pt}$ the following morphisms
  \[
  \Loc_{\Tel(h)}(\SS)  \to \left(\Loc_{\Tel(h)} \circ \susp^{\infty}_{+} \right)(X) \simeq \left(\Loc_{\Tel(h)} \circ \susp^{\infty}\right)(X) \oplus \Loc_{\Tel(h)}(\SS)  \to \Loc_{\Tel(h)}(\SS),
  \]
  where the~$\Tel(h)_{\bullet}$\nobreakdash-local sphere spectrum~$\Loc_{\Tel(h)}(\SS)$ is the symmetric monoidal unit of~$\infSp_{\Tel(h)}$.
  Moreover, this shows that~$\widetilde{\augideal}_{\infopE_0} \circ  \susp^{\infty}_{+, v_{h}} \simeq  \susp^{\infty}_{v_{h}}$.
  Under the equivalence~${\infHType_{v_{h}} \simeq \infAlg_{\infopLie}(\infSp_{\Tel(h)})}$, the functor~$\susp^{\infty}_{+, v_{h}}$ corresponds to~$\indc_{\infopLie}^{+}$.
\end{proof}

\begin{corollary}
  For every~$n \in \NN$, the functor~$\indc_{\infopLie}^{+}$ induces a~functor
  \[
  \indc_{\infopLie}^{+,n} \colon \infAlg_{\infopE_n}\left(\infAlg_{\infopLie}\left(\infSp_{\Tel(h)}\right)\right) \to \infAlg_{\infopE_n}\left(\infAlg_{\infopE_0}^{\aug}\left(\infSp_{\Tel(h)}\right)\right) \simeq \infAlg_{\infopE_n}^{\aug}\left(\infSp_{\Tel(h)}\right).
  \]
\end{corollary}

\begin{passage}[The functor~$\UE_n$]\label{para:Un-Tn}
The functor~$\UE_{n} \colon \infAlg_{\infopLie}(\infSp_{\Tel(h)}) \to \infAlg_{\infopE_n}^{\nut}(\infSp_{\Tel(h)})$ from diagram~\eqref{diag:Lie-En-Th} is defined as the composition
\[
\infAlg_{\infopLie}\left(\infSp_{\Tel(h)}\right) \xrightarrow{\widetilde{\Loop}_{\infopLie}^n} \infAlg_{\infopE_n}\left(\infAlg_{\infopLie}\left(\infSp_{\Tel(h)}\right)\right) \xrightarrow{\indc_{\infopLie}^{+,n}} \infAlg_{\infopE_n}^{\aug}\left(\infSp_{\Tel(h)}\right) \xrightarrow{\augideal_{\infopE_n}} \infAlg_{\infopE_n}^{\nut}\left(\infSp_{\Tel(h)}\right).
\]
By construction~$\UE_n$ admits a right adjoint~$\rUE_n$, constructed as follows. 
The right adjoint~$\triv_{\infopLie}^{+}$ to~$\indc_{\infopLie}^{+}$ is lax monoidal and thus induces a functor
\[
\triv_{\infopLie}^{+, n} \colon \infAlg_{\infopE_n}^{\aug}\left(\infSp_{\Tel(h)}\right) \to \infAlg_{\infopE_n}\left(\infAlg_{\infopLie}\left(\infSp_{\Tel(h)}\right)\right),
\]
for every~$n \in \NN$.
The functor~$\triv_{\infopLie}^{+, n}$ is right adjoint to~$\indc_{\infopLie}^{+, n}$.
The functor~$\rUE_n$ is given by the composition~$\Bac_{n} \circ \triv_{\infopLie}^{+,n} \circ \trivaug_{\infopE_n}$.
\end{passage}

\begin{proposition}\label{prop:Un-Free}
There exists the following commutative diagram of~$\infty$\nobreakdash-categories:
\begin{equation}\label{diag:Lie-En-free-Th}
\begin{tikzcd}
  \infAlg_{\infopLie}(\Cca) \arrow[r, "\UE_n"] & \infAlg_{\infopE_n}^{\nut}(\Cca)\\
  \Cca \arrow[u, "\free_{\infopLie}"] \arrow[r, "\Loop^{n}_{\Cca}"', "\sim"] & \Cca. \arrow[u, "\free_{\infopE_n^{\nut}}"']
\end{tikzcd}
\end{equation}
\end{proposition}

\begin{proof}
  It suffices to show that the following equivalence
  \begin{equation}\label{eq:Lie-En-free-Th-equiv}
    \frgt_{\infopLie} \circ \rUE_n \simeq \susp_{\Cca}^{n} \circ \frgt_{\infopE_n}^{\nut}
  \end{equation}
  of functors holds;
  Then the commutative diagram~\eqref{diag:Lie-En-free-Th} follows by taking the left adjoints to the left-and right-hand-side of~\eqref{eq:Lie-En-free-Th-equiv}.
  
  Now by~\eqref{eq:bar-frgt-commutes} we see that
  \begin{equation}
    \begin{split}
    \frgt_{\infopLie} \circ \left(\Bac_n \circ \triv_{\infopLie}^{+, n} \circ \trivaug_{\infopE_n}\right)
      &\simeq \susp^n_{\Cca} \circ \frgt_{\infopLie} \circ \frgt_{\infopE_n} \circ  \triv_{\infopLie}^{+, n} \circ \trivaug_{\infopE_n}\\
      &\simeq \susp^n_{\Cca} \circ \frgt_{\infopLie} \circ (\triv_{\infopLie} \circ \augideal_0) \circ \frgt_{\infopE_n}^{\aug} \circ \trivaug_{\infopE_n} \\
      &\simeq \susp_{\Cca}^{n} \circ \frgt_{\infopE_n},
    \end{split}
  \end{equation}
  where
  \begin{enumerate}
    \item the second equivalence holds by the monoidality of the functor~$\triv_{\infopLie}^{+}$,~and 
    \item the third equivalence holds by the following~equivalences
      \[
      \frgt_{\infopLie} \circ \triv_{\infopLie} \simeq \id \text{ and } \frgt_{\infopE_n}^{\aug} \circ \trivaug_{\infopE_n} \simeq \trivaug_{0} \circ \frgt_{\infopE_n}^{\nut} \text{ and } \augideal_0 \circ \trivaug_0 \simeq \id. \qedhere
      \]
  \end{enumerate} 
\end{proof}

\begin{proposition}
  For every~$n \in \NN$ the following diagram
  \[
  \begin{tikzcd}
    \infAlg_{\infopLie}\left(\infSp_{\Tel(h)}\right) \arrow[d, "\UE_{n+1}"'] \arrow[rd, "\UE_n"] & \\
    \infAlg_{\infopE_{n+1}}^{\nut}\left(\infSp_{\Tel(h)}\right) \arrow[r, "\rB_{n+1}"'] & \infAlg_{\infopE_n}^{\nut}\left(\infSp_{\Tel(h)}\right)
  \end{tikzcd}
  \]
  of~$\infty$\nobreakdash-categories commutes.
\end{proposition}

\begin{proof}
  This follows by our constructions of the functors~$\UE_n$ and~$\rB_{n}$.
\end{proof}

\subsection{Fully faithfulness of~\texorpdfstring{$\UE_{\infty}$}{U_∞}}\label{sec:fully-faithful}

In this section we give the proof of the following theorem.

\begin{theorem}\label{thm:U-infty-ff}
 Let~$h \geq 1$ be a natural number.
 The functor
 \[
 \UE_{\infty} \colon \infAlg_{\infopLie}\left(\infSp_{\Tel(h)}\right) \to \varprojlim \infAlg_{\infopE_n}^{\nut}\left(\infSp_{\Tel(h)} \right)
 \] 
 induced by the commutative diagram~\eqref{diag:Lie-En-Th} is fully~faithful.  
\end{theorem}

\begin{passage}[Proof strategy]\label{para:U-infty-ff-strategy}
  We explain our proof strategy first before going into detailed arguments. 
  To prove the fully faithfulness of~$\UE_{\infty}$, it is equivalent to show that the unit natural~transformation 
  \[
  \id \to \rUE_{\infty} \circ \UE_{\infty}
  \]
  of the adjunction~$\UE_{\infty} \dashv \rUE_{\infty}$~\eqref{eq:U-T-infty-adjunction} is an equivalence, \ie we need to show that the evaluation~${L \to (\rUE_{\infty} \circ \UE_{\infty})(L)}$ is an equivalence for every spectral Lie algebra~${L \in \infAlg_{\infopLie}(\infSp_{\Tel(h)})}$.
  By~\cite[Proposition~1.2.4.1]{Lur09} it suffices to show that the induced map
  \begin{equation}\label{eq:Yoneda-equivalence-UT}
    \infMap_{\infAlg_{\infopLie}(\infSp_{\Tel(h)})}(M, L) \to \infMap_{\infAlg_{\infopLie}(\infSp_{\Tel(h)})}(M, (\rUE_{\infty} \circ \UE_{\infty})(L))
  \end{equation}
  on mapping spaces is an equivalence in~$\infHType$, for every~${M \in \infAlg_{\infopLie}(\infSp_{\Tel(h)})}$.
  
  Let~$V_{h}$ be any pointed~$p$-local finite complex of type~$h$.
  Recall that~$\infSp_{\Tel(h)}$ is generated under small colimits by the spectrum~$\Loc_{h}^{\f}\susp^{\infty}V_{h}$, see~\cref{rmk:Th-compact-generated}.
  Furthermore, every spectral Lie algebra is equivalent to a sifted colimit of free spectral Lie algebras.
  Thus, to verify that~\eqref{eq:Yoneda-equivalence-UT} is an equivalence for every spectral Lie algebra~$M$, it suffices to prove it for~$M = \free_{\infopLie}(\Loc_{h}^{\f}\susp^{\infty}V_{h})$.
  Using series of adjunctions and the property that~$\rUE_{\infty} \circ \UE_{\infty}$ preserves finite limits~(see~\cref{prop:unit-finite-limit}), we reduce the problem to proving that there exists an~equivalence
  \begin{equation}\label{eq:counit-Fn}
    L^{V_{h}} \xrightarrow{\sim} (\rUE_{\infty} \circ \UE_{\infty})\left(L^{V_{h}}\right)
  \end{equation}
  for every~$\Tel(h)$-local spectral Lie algebras~$L$, where~$L^{V_{h}}$ denotes the limit of the constant diagram~${V_{h} \to \infAlg_{\infopLie}(\infSp_{\Tel(h)})}$ sending every vertex of~$V_{h}$ to~$L$~(see~\cref{prop:proof-reduction-fully-faithfyl}). 
  As the last step, we choose a finite complex~$V$ of type~$h$ which admits a~$v_{h}$ self-map and show that~\eqref{eq:counit-Fn} is an equivalence in the case~$V_{h} = V$~(see~\cref{prop:XV-tu-equi}), using the fact that~$L^{V}$ is a trivial spectral Lie algebra~(see~\cref{cor:XV-triv-lie}).
\end{passage}

In the proof we will use a special property of dual Goodwillie calculus on endofunctors on~$\infSp_{\Tel(h)}$.
First, let us recall the basics about Goodwillie calculus of functors.

\begin{passage}[Goodwillie calculus]\label{para:Goodwillie-calculus}
  Let~$F \colon \Cca \to \Dca$ be a functor where
  \begin{enumerate}
    \item $\Cca$ admits finite colimits and terminal objects,
    \item $\Dca$ admits finite limits and sequential colimits, and
    \item sequential colimits commute with finite limits in~$\Dca$.
  \end{enumerate}
  Using the theory of Goodwillie calculus~\cite{Goo03}, one can construct the following commutative~diagram 
  \[
  \begin{tikzcd}[row sep = large]
  &              &              &              &              & {F} \arrow[dd] \arrow[ldd] \arrow[lldd] \arrow[llldd] \arrow[lllldd] \arrow[llllldd] \\
  \\
  {\cdots} \arrow[r] & {\Poly_{k+1} F} \arrow[r] & {\Poly_{k} F} \arrow[r] & {\cdots} \arrow[r] & {\Poly_2 F} \arrow[r] & {\Poly_{1}F}                                                                                                                           
  \end{tikzcd}
  \]
  of functors from~$\Cca$ to~$\Dca$ where the functor~$\Poly_{k}F$ is the best approximation of~$F$ by a~``$k$\nobreakdash-polynomial'' functors for every~$k \geq 1$.
  Very informally speaking, one can consider the lower horizontal tower as an approximation of the functor~$F$, analogous to the Taylor approximation of a function.
  
  This diagram, or the lower horizontal tower, is known as the \emph{Goodwillie tower} of~$F$.
  In particular, we have 
  \[
  \Poly_{1}(F) \simeq \varinjlim_{n \geq 0} \left(\Loop^{n}_{\Dca} \circ F \circ \susp^{n}_{\Cca}\right),
  \]
  known as the ``linear approximation'' of~$F$, see~\cite[Example~6.1.1.28]{HA}.
  For~$\Cca = \Dca$ and~$F = \id$ we see that~$\Poly_1(F) \simeq \Loop^{\infty}_{\Cca} \susp^{\infty}_{\Cca}$, which relates closely to the stabilisation of~$\Cca$.
  We refer the reader to~\cite{Goo03, HA} for a detailed introduction of this theory and to~\cite{Kuh07} for nice applications.
\end{passage}

\begin{passage}[Dual Goodwillie calculus of Endofunctors of~$\infSp_{\Tel(h)}$]\label{para:dual-Goodwillie}
  Let~$F$ be the functor in~\cref{para:Goodwillie-calculus}.
  Because of the hypotheses \rom{1}-\rom{3} in~\cref{para:Goodwillie-calculus}, one can not simply construct a Goodwillie tower of the functor~${F^{\op} \colon \Cca^{\op} \to \Dca^{\op}}$: For example, the opposite category~$\Dca^{\op}$ does not satisfy the conditions~\rom{2} and~\rom{3} in general. 
  
  However, if~$\Dca$ is stable and~$\Cca$ admits finite limits and initial objects, then~$F^{\op}$ does fulfil the hypotheses \rom{1}-\rom{3} in~\cref{para:Goodwillie-calculus}.
  Assume that we work in this situation.
  Then the \emph{dual Goodwillie tower} of~$F$ is defined as the Goodwillie tower of~$F^{\op} \colon \Cca^{\op} \to \Dca^{\op}$, illustrated by the commutative diagram in the~$\infty$\nobreakdash-category~$\infFun(\Cca, \Dca)$ below:
  \begin{equation*}
  \begin{tikzcd}[row sep = large]
  {\Poly^{1}F} \arrow[r] \arrow[dd] & {\Poly^2 F} \arrow[r] \arrow[ldd] & {\cdots} \arrow[r] \arrow[lldd] & {\Poly^{k}F} \arrow[r] \arrow[llldd] & {\Poly^{k+1}F} \arrow[r] \arrow[lllldd] & {\cdots} \arrow[llllldd] \\
  \\
  {F,}  & &  & &  &                             
  \end{tikzcd}
  \end{equation*}
  This is called the \emph{dual Goodwillie tower} of~$F$.
  Similarly as the formula for~$\Poly_1(F)$, there exists the following \emph{colinear~approximation}
  \[
  \Poly^{1}(F) \simeq \varprojlim_{n \geq 0} \left(\susp^{n}_{\Cca} \circ F \circ \Loop^{n}_{\Cca}\right). 
  \]
  For~$k \geq 1$, we say 
  \begin{enumerate}
    \item $F$ is \emph{$k$\nobreakdash-polynomial} if~$\Poly^{k}F \simeq F$, and 
    \item $F$ is \emph{$k$\nobreakdash-homogeneous} if~$\Poly^{k}F \simeq F$ and~$\Poly^{i}F = \pt$ for~$i \leq k$.
  \end{enumerate}
\end{passage}

The dual Goodwillie tower of a functor~$F \colon \infSp_{\Tel(h)} \to \infSp_{\Tel(h)}$ has particularly nice properties, see~\cite[§4.1,~Appendix~B]{Heu21} for more details.
  We need the following lemma for later applications.

\begin{lemma}\label{lem:Th-dual-calculus-sum}
  Consider a functor~$F \colon \infSp_{\Tel(h)} \to \infSp_{\Tel(h)}$ satisfying the following property: There exists a sequence~$(F_{j})_{j \geq 1}$ of endofunctors of~$\infSp_{\Tel(h)}$ such~that
  \begin{enumerate}
    \item $F \simeq \coprod_{j = 1}^{\infty}F_{j}$ and
    \item $F_j$ is~$j$\nobreakdash-homogeneous for every~$j \geq 1$.
  \end{enumerate}
  Then the natural map 
  \[
  \coprod_{j = 1}^{k} F_{j} \to \Poly^{k}F
  \]
  is an equivalence of functors, for every~$k \geq 1$.
\end{lemma}

\begin{proof}
  See~\cite[Lemma~4.6]{Heu21}.
  The proof uses a ``uniform nilpotence'' result for~$\infSp_{\Tel(h)}$, which uses the fact that the Tate construction in~$\infSp_{\Tel(h)}$ vanishes; this does not hold for a general stable~$\infty$\nobreakdash-category, see~\cite[Lemma~B.4]{Heu21}.
\end{proof}

Now we will work towards the proof of~\cref{thm:U-infty-ff}.

\begin{proposition}\label{prop:ff-left-adjoint}
  The functor~$\UE_{\infty}$ is fully faithful if and only if there exists an equivalence~$\rUE_{\infty} \circ \UE_{\infty} \simeq \id$ of functors.
\end{proposition}

\begin{proof}
  See~\cite[Lemma 1.1.1]{Joh02}.
\end{proof}

\begin{situation}\label{sit:generator-v}
  Fix a prime number~$p$.
  Till the end of this section we work with a~$p$-local finite complex~$V$ of type~$h$ together with a~$v_{h}$~self\nobreakdash-map~$v \colon \susp^{d}V \to V$.
\end{situation}

\begin{definition}\label{def:power-Lie-alg}
  For~$L \in \infAlg_{\infopLie}(\infSp_{\Tel(h)})$, define~$L^{V} \in \infAlg_{\infopLie}(\infSp_{\Tel(h)})$ as the limit of the constant diagram
  \[
  V \to \infAlg_{\infopLie}(\infSp_{\Tel(h)})
  \]
  sending every vertex of~$V$ to~$L$.
  
  By construction the endofunctor~$(\blank)^{V}$ of~$\infAlg_{\infopLie}(\infSp_{\Tel(h)})$ is right adjoint to the \emph{copower functor}~$V \otimes (\blank)$, which assigns to an object~$L \in \infAlg_{\infopLie}(\infSp_{\Tel(h)})$ the colimit~${V \otimes L}$ of the afmorementioned constant diagram.
\end{definition}

\begin{proposition}\label{prop:proof-reduction-fully-faithfyl}
  Consider the unit natural transformation~$\eta \colon \id \to \rUE_{\infty} \circ \UE_{\infty}$ of the adjunction~$\UE_{\infty} \dashv \rUE_{\infty}$.
  The following statements are~equivalent:
  \begin{enumerate}
    \item The natural transformation~$\eta$ is an equivalence of functors.
    \item The natural transformation~$\eta$ induces the following~equivalence
      \[
      (\rUE_{\infty} \circ \UE_{\infty})((\blank)^{V}) \simeq (\blank)^{V}.
      \]
      of functors.
  \end{enumerate}
\end{proposition}

In order to prove~\cref{prop:proof-reduction-fully-faithfyl} we first need to introduce some notations and show some properties of the functor~$\rUE_{\infty} \circ \UE_{\infty}$.

\begin{passage}
Recall the notations from~\cref{para:Un-Tn}.
The~$(\UE_{\infty} \dashv \rUE_{\infty})$\nobreakdash-adjunction is induced by the adjunctions~$\UE_{n} \dashv \rUE_{n}$ where, for every~$n \geq 1$,
\begin{align*}
  \UE_{n} \colon \infAlg_{\infopLie}\left(\infSp_{\Tel(h)}\right) &\to \infAlg_{\infopE_n}^{\nut}\left(\infSp_{\Tel(h)}\right) \\
  L &\mapsto (\augideal_{\infopE_n} \circ \indc_{\infopLie}^{+, n} \circ \widetilde{\Loop}^{n}_{\infopLie})(L) \\
  \rUE_{n} \colon \infAlg_{\infopE_n}^{\nut}\left(\infSp_{\Tel(h)}\right) &\to \infAlg_{\infopLie}\left(\infSp_{\Tel(h)}\right) \\
  K_n &\mapsto (\Bac_n \circ \triv_{\infopLie}^{+, n}\circ \trivaug_{\infopE_n})(K_{n}).
\end{align*}
Denote the right adjoint to the ``non-unital'' Bar construction~$\rB_{n}$ by~$\rrB_{n}$, see~\cref{para:non-unital-bar}.
The family of natural transformations 
\[
\rUE_{n+1}\circ  \UE_{n+1} \to \rUE_{n+1} \circ (\rrB_{n+1} \circ \rB_{n+1}) \circ \UE_{n+1} \to \rUE_{n}\circ  \UE_{n}
\]
for~$n \geq 0$ induces an equivalence 
\begin{align*}
\rUE_{\infty} \circ \UE_{\infty} &\simeq \varprojlim\left(\rUE_n \circ \UE_n\right) \\
                     &\simeq \varprojlim_{n}\left(\Bac_n \circ \triv_{\infopLie}^{n, +} \circ \indc_{\infopLie}^{n, +} \circ \Loop^{n}_{\infopLie}\right)
\end{align*}
of functors, where the limits are taken in~$\infty$\nobreakdash-category of endofunctors of~$\infAlg_{\infopLie}(\infSp_{\Tel(h)})$.
The first equivalence holds by the universal property of the limit.
\end{passage} 

\begin{proposition}\label{prop:unit-finite-limit}
  The functor~$\rUE_{\infty} \circ \UE_{\infty}$ preserves finite limits.
\end{proposition}

\begin{notation}\label{not:abbrev-for-proof-ff}
  In the proof of~\cref{prop:unit-finite-limit}, we make the following simplification of notations.
  \begin{enumerate}
    \item Abbreviate the forgetful functor~$\frgt_{\infopE_n}$ by~$\frgt^{n}$.
    \item Abbreviate the trivial augmentation~$\trivaug_{\infopE_n}$ and augmentation ideal~$\augideal_{\infopE_n}$ functor by~$\trivaug_{n}$ and~$\augideal_{n}$, respectively. 
    \item Abbreviate the suspension and loop functor of~$\infSp_{\Tel(h)}$ by~$\susp$ and~$\Loop$, respectively.
    \item Let~$\Dca$ be a symmetric monoidal~$\infty$\nobreakdash-category.
      For~$m, n \in \NN \cup \{\infty\}$ with~$m > n$, let~$\frgt^{m}_{n}$ denote the forgetful functor 
      \[
      \infAlg_{\infopE_m}(\Dca) \to \infAlg_{\infopE_n}(\Dca),
      \]
      induced by the morphism~${i_{n}^{m} \colon \infopE_n \to \infopE_m}$ of~$\infty$\nobreakdash-operads.
    \item Recall the~$n$\nobreakdash-fold loop endofunctor
      \[
      \Loop_{\infopLie}^n \colon \infAlg_{\infopLie}\left(\infSp_{\Tel(h)}\right) \to \infAlg_{\infopLie}\left(\infSp_{\Tel(h)}\right)
      \] 
      and its factorisation
      \[
      \widetilde{\Loop}^{n}_{\infopLie} \colon \infAlg_{\infopLie}\left(\infSp_{\Tel(h)}\right) \to \infAlg_{\infopE_n}\left(\infAlg_{\infopLie}\left(\infSp_{\Tel(h)}\right)\right)
      \]
      from~\cref{prop:Loop-fac-En}.
      In particular, we have~$\Loop_{\infopLie}^n = \frgt^{n} \circ \widetilde{\Loop}^{n}_{\infopLie}$.
  \end{enumerate}
\end{notation}

\begin{proof}[Proof of~\cref{prop:unit-finite-limit}]
  Since the forgetful functor~$\frgt_{\infopLie}$ is conservative, it suffices to prove that the composition
  \[
  \frgt_{\infopLie} \circ \rUE_{\infty} \circ \UE_{\infty} \colon  \infAlg_{\infopLie}\left(\infSp_{\Tel(h)}\right) \to \infSp_{\Tel(h)}
  \]
  preserves finite limits. 
  Moreover, it is then equivalent to show that~$ \frgt_{\infopLie} \circ \rUE_{\infty} \circ \UE_{\infty}$ preserves the zero object and commutes with the loop functors, because the target~$\infty$\nobreakdash-category is stable, see~\cite[Corollary~4.4.2.5]{Lur09} and~\cite[Lemma~3.9]{Heu21}.
  
  Since every single functor in the composition preserves the zero object, the composition also~does.
  We write the functor~$\frgt_{\infopLie} \circ \rUE_{\infty} \circ \UE_{\infty}$ explicitly as
  \begin{equation}\label{eq:frgt-unit-tu}
    \begin{split}
     \frgt_{\infopLie}  \circ \rUE_{\infty} \circ \UE_{\infty} 
    &\simeq \frgt_{\infopLie} \circ \varprojlim_{n}\left(\rUE_n \circ \UE_n\right) \\
    &\simeq \varprojlim_{n}\left((\frgt_{\infopLie} \circ \rUE_n) \circ \UE_n\right) \\
    &\simeq \varprojlim_{n}\left(\susp^n \circ \frgt_{\infopE_n}^{\nut} \circ \UE_n\right) \\
    &\simeq \varprojlim_{n}\left(\susp^n \circ \frgt_{\infopE_n}^{\nut} \circ \augideal_{n} \circ \indc_{\infopLie}^{n, +} \circ \widetilde{\Loop}^{n}_{\infopLie}\right)
    \end{split}
  \end{equation}
  where the second equivalence holds because~$\frgt_{\infopLie}$ commutes with small limits, and the third equivalence holds by~\eqref{diag:Lie-En-Frgt}.
  For every~$L \in \infAlg_{\infopLie}(\infSp_{\Tel(h)})$, we obtain
  \begin{align*}
    &(\frgt_{\infopLie} \circ  \rUE_{\infty} \circ \UE_{\infty})(\Loop_{\infopLie}(L)) \\
    \simeq &\varprojlim_{n}\left(\susp^n \circ \frgt_{n}^{\nut} \circ \augideal_n \circ \indc_{\infopLie}^{+, n} \circ \widetilde{\Loop}^{n}_{\infopLie}\right)(\Loop_{\infopLie}(L)) \\
    \overset{(a)}{\simeq} &\varprojlim_{n}\left(\susp^n \circ \frgt_{n}^{\nut} \circ \augideal_{n} \circ \indc_{\infopLie}^{+, n} \circ \frgt_{n}^{n+1} \circ \widetilde{\Loop}^{n+1}_{\infopLie}\right)(L) \\
    \overset{(b)}{\simeq} &\varprojlim_{n}\left(\susp^n \circ \frgt_{n}^{\nut} \circ \augideal_{n} \circ \frgt_{n}^{n+1} \circ \indc_{\infopLie}^{+, n+1}\circ \widetilde{\Loop}^{n+1}_{\infopLie}\right)(L) \\
    \overset{(c)}{\simeq} &\Loop\left( \varprojlim_{n}\left(\susp^{n+1} \circ \frgt^{\nut}_{n+1} \circ \augideal_{n+1} \circ  \indc_{\infopLie}^{+, n+1} \circ \widetilde{\Loop}^{n+1}_{\infopLie}\right)(L)\right) \\
    \simeq &\Loop \left( (\frgt_{\infopLie} \rUE_{\infty}\UE_{\infty})(X)\right).
  \end{align*}
  The equivalence (a) holds because of the equivalence
  \[
  \frgt_n^{n+1} \circ \widetilde{\Loop}^{n+1}_{\infopLie} \simeq \widetilde{\Loop}^{n}_{\infopLie} \circ \frgt^1 \circ \widetilde{\Loop}_{\infopLie},
  \] 
  since every functor here is a right adjoint.
  The equivalence (b) holds because the functor~$\indc_{\infopLie}^{+}$ is symmetric monoidal, and thus induces the following commutative diagram
  \[
  \begin{tikzcd}[row sep = huge]
  \infAlg_{\infopE_{m}}\left(\infAlg_{\infopLie}(\infSp_{\Tel(h)})\right) \arrow[r, "\indc_{\infopLie}^{+, m}"] \arrow[d, "\frgt_{n}^{m}"'] & \infAlg_{\infopE_{m}}\left(\infAlg_{\infopE_0}^{\aug}(\infSp_{\Tel(h)})\right) \simeq \infAlg_{\infopE_m}^{\aug}\left(\infSp_{\Tel(h)}\right) \arrow[d, "\frgt_{n}^{m}"] \\
  \infAlg_{\infopE_{n}}\left(\infAlg_{\infopLie}(\infSp_{\Tel(h)})\right) \arrow[r, "\indc_{\infopLie}^{+, n}"']  & \infAlg_{\infopE_{n}}\left(\infAlg_{\infopE_0}^{\aug}(\infSp_{\Tel(h)})\right) \simeq \infAlg_{\infopE_n}^{\aug}\left(\infSp_{\Tel(h)}\right),          
  \end{tikzcd}
  \]
  for every~$m, n \in \NN \cup \{\infty\}$ with~$m > n$.
  The equivalence~$(c)$ holds by
  \[
  \frgt_{n}^{\nut} \circ \augideal_{n} \circ \frgt_{n}^{n+1} \simeq \frgt^{\nut}_{n+1} \circ \augideal_{n+1},
  \]
  following from~\cref{cor:equiv-non-unital-aug}. 
\end{proof}

\begin{corollary}\label{cor:power-unit-commutes}
  For every spectral Lie algebra~$L \in \infAlg_{\infopLie}(\infSp_{\Tel(h)})$, there exists an natural equivalence
  \[
  (\rUE_{\infty}\UE_{\infty})(L^{V}) \simeq \left((\rUE_{\infty}\UE_{\infty})(L)\right)^{V}
  \]
  of spectral Lie algebras.
\end{corollary}

\begin{proof}
  The object~$L^{V}$ is obtained by a finite limit construction.
\end{proof}

\begin{proof}[Proof of~\cref{prop:proof-reduction-fully-faithfyl}]
  It is obvious that~$\rom{1}$ implies~$\rom{2}$.
  We show that \rom{2} implies \rom{1}.
  Abbreviate the mapping space~$\infMap_{\infAlg_{\infopLie}(\infSp_{\Tel(h)})}$ by~$\infMap_{\infopLie}$ in this proof.
  
  As we explained in~\cref{para:U-infty-ff-strategy}, it suffices to prove that the unit natural transformation of the adjunction~${\UE_{\infty} \dashv \rUE_{\infty}}$ induces an equivalence
  \begin{equation}\label{eq:test-counit-equi-on-free}
      \infMap_{\infopLie}\left(\free_{\infopLie}\left(\Loc_{h}^{\f}\susp^{\infty}V\right), \ L\right) \xrightarrow{\sim} \infMap_{\infopLie}\left(\free_{\infopLie}\left(\Loc_{h}^{\f}\susp^{\infty}V\right), \ (\rUE_{\infty} \circ \UE_{\infty})(L)\right)
  \end{equation}
  for every~$L \in \infAlg_{\infopLie}(\infSp_{\Tel(h)})$.
  Then we have
  \begin{equation}\label{eq:adjunctions-equivs}
    \begin{split}
    &\infMap_{\infopLie}\left(\free_{\infopLie}\left(\Loc_{h}^{\f}\susp^{\infty}V\right), \ L\right) \\
    \simeq &\infMap_{\infSp_{\Tel(h)}}\left(\Loc_{h}^{\f}\susp^{\infty}V, \ \frgt_{\infopLie}L\right) \\
    \simeq &\infMap_{\infSp}\left(\susp^{\infty}V, \ \frgt_{\infopLie} L\right) \\
    \simeq &\infMap_{\infSp}\left(V \otimes \SS, \ \frgt_{\infopLie} L\right) \\
    \simeq &\infMap_{\infSp}\left(\SS, \ \left(\frgt_{\infopLie} L\right)^{V}\right) \\
    \simeq &\infMap_{\infSp}\left(\SS, \ \frgt_{\infopLie} \left(L^{V}\right)\right),
    \end{split}
  \end{equation}
  where the first equivalence holds by adjunction~$\free \dashv \frgt$, the second holds by the universal property of the localisation~$\Loc_n^{\f}$, the third and fourth hold by the copower--power adjunction~$(\blank \otimes V) \dashv (\blank)^{V}$ and the last equivalence holds because the forgetful functor preserves small~limits.
  
  Similarly for the target mapping space, we have an equivalence
  \[
    \infMap_{\infopLie}\left(\free_{\infopLie}\left(\Loc_{h}^{\f}\susp^{\infty}V\right), \ (\rUE_{\infty} \circ \UE_{\infty})(L)\right) \xrightarrow{\sim} \infMap_{\infSp}\left(\SS, \ \frgt_{\infopLie}\left(\left((\rUE_{\infty} \circ \UE_{\infty})(L)\right)^{V} \right) \right).
  \]
  Thus, the morphism~\eqref{eq:test-counit-equi-on-free} is equivalent to following morphism
  \[
  \Map_{\infSp}\left(\SS, \ \frgt_{\infopLie} \left(L^{V}\right)\right) \to \Map_{\infSp}\left(\SS, \ \frgt_{\infopLie}\left(\left((\rUE_{\infty} \circ \UE_{\infty})(L)\right)^{V} \right) \right),
  \]
  which is indeed an equivalence for every~$L \in \infAlg_{\infopLie}(\infSp_{\Tel(h)})$, by assumption \rom{2} and~\cref{cor:power-unit-commutes}.
\end{proof}

\begin{proposition}\label{prop:LV-infinite-loop}
  For every~$L \in \infAlg_{\infopLie}(\infSp_{\Tel(h)})$, the spectral Lie algebra~$L^{V}$ is an infinite loop object of the~$\infty$\nobreakdash-category~$\infAlg_{\infopLie}(\infSp_{\Tel(h)})$.
\end{proposition}

\begin{proof}
  Recall the~$v_{h}$-self-map~$v \colon \susp^{d}V \to V$ of~$V$ from~\cref{sit:generator-v}. 
  The induced morphism~$(\Loc_{h}^{\f}\circ \susp^{\infty})(v)$ becomes an equivalence in~$\infSp_{\Tel(h)}$, see~\cref{rmk:Th-compact-generated}.
  Thus, we expect the following equivalence in~$\infAlg_{\infopLie}(\infSp_{\Tel(h)})$.
  
  \begin{claim}
    The map~$v$ induces an equivalence~$L^{v} \colon L^{V} \to L^{\susp^{d}V}$ of spectral Lie algebras.
  \end{claim}
    The equivalence~$v_{\ast} \colon \Loc_{h}^{\f} \susp^{\infty + d}V \to \Loc_{h}^{\f} \susp^{\infty}V$ in~$\infSp_{\Tel(h)}$ induces an equivalence
    \[
    \Map_{\infSp_{\Tel(h)}}\left(\Loc_{h}^{\f} \susp^{\infty}V, \ \frgt_{\infopLie}L\right) \xrightarrow{\sim} \Map_{\infSp_{\Tel(h)}}\left(\Loc_{h}^{\f} \susp^{\infty+d}V, \ \frgt_{\infopLie}L\right).
    \]
    With the same arguments as in~\eqref{eq:adjunctions-equivs} we obtain by adjunctions an equivalence
    \[
    \frgt_{\infopLie}\left(L^{V}\right) \to \frgt_{\infopLie}\left(L^{\susp^{d}V}\right)
    \]
    of spectra.
    The claim follows by the fact that the functor~$\frgt_{\infopLie}$ is conservative.
  
  Considering the~$(\susp \dashv \Loop)$\nobreakdash-adjunction of pointed \htypes, we prove the following~claims.

  \begin{claim}
    There exists a natural equivalence 
   ~$
    \Loop^{d}_{\infopLie} (L^V) \simeq L^{\susp^{d}V}
   ~$ 
    of spectral Lie algebras.
  \end{claim}
  
  Let~$\usphere^d$ denote the~$d$\nobreakdash-dimensional sphere.
  The following sequence of natural equivalences proves the claim:
  \[
  L^{\susp^{d}V} = \varprojlim_{\susp^{d}V} L \simeq \varprojlim_{\usphere^d} \varprojlim_{V} L \simeq \varprojlim_{\usphere^d}\left(L^{V}\right)\simeq \Loop^{d}_{\infopLie}\left(L^V\right). \qedhere
  \]
\end{proof}

\begin{corollary}\label{cor:XV-triv-lie}
  For every~$L \in \infAlg_{\infopLie}(\infSp_{\Tel(h)})$, the spectral Lie algebra~$L^{V}$ is trivial, \ie it lies in the image of the trivial spectral Lie algebra functor~$\triv_{\infopLie}$~(see~\eqref{eq:two-adj-Lie-alg}).
  In particular, there exists a natural equivalence
  \[
  L^{V} \simeq \left(\triv_{\infopLie} \circ \frgt_{\infopLie}\right)\left(L^{V}\right)
  \] 
  in the~$\infty$\nobreakdash-category~$\infAlg_{\infopLie}(\infSp_{\Tel(h)})$.
\end{corollary}

\begin{proof}
   Recall that the adjunction~$\indc_{\infopLie} \dashv \triv_{\infopLie}$ exhibits~$\infSp_{\Tel(h)}$ as the stabilisation of the~$\infty$\nobreakdash-category~$\infAlg_{\infopLie}(\infSp_{\Tel(h)})$. 
  In other words, the image of~$\triv_{\infopLie}$ is exactly the set of infinite loop objects of~$\infAlg_{\infopLie}(\infSp_{\Tel(h)})$, see~\cref{para:vh-stabilisation}.
  So, by~\cref{prop:LV-infinite-loop}, the spectral Lie algebra~$L^{V}$ is contained in the image of~$\triv_{\infopLie}$.
  Therefore, we obtain the equivalence
  \[
  L^{V} \simeq \left(\triv_{\infopLie} \circ \frgt_{\infopLie}\right)\left(L^{V}\right),
  \] 
  because~$\frgt_{\infopLie} \circ \triv_{\infopLie} \simeq \id_{\infAlg_{\infopLie}(\infSp_{\Tel(h)})}$ by~\eqref{eq:two-adj-Lie-alg}.
\end{proof}

\begin{proposition}\label{prop:XV-tu-equi}
  There exists a natural equivalence~$\left(\rUE_{\infty} \circ \UE_{\infty}\right)(L^{V}) \simeq L^V$ of spectral Lie algebras, for every~$L \in \infAlg_{\infopLie}(\infSp_{\Tel(h)})$.
\end{proposition}

\begin{proof}
  Since the forgetful functor is conservative, it is enough to show that
  \[
  \left(\frgt_{\infopLie} \circ \rUE_{\infty} \circ\UE_{\infty}\right) \left( L^{V}\right) \simeq \frgt_{\infopLie}\left( L^{V}\right).
  \]
  We will use the abbreviations of notations as in~\cref{not:abbrev-for-proof-ff}.
  Recall the expression of~$\frgt_{\infopLie} \circ \rUE_{\infty} \circ\UE_{\infty}$ from~\eqref{eq:frgt-unit-tu}.
  We have
  \begin{align*}
    &\left(\frgt_{\infopLie} \circ \rUE_{\infty} \circ\UE_{\infty}\right) \left( L^{V}\right) \\
    \simeq  &\varprojlim_{n}\left(\susp^n \circ \frgt_{n}^{\nut} \circ \augideal_n \circ  \indc_{\infopLie}^{+, n} \circ \widetilde{\Loop}^{n}_{\infopLie}\right) \left( L^{V}\right) \\
    \overset{\rom{1}}{\simeq} &\varprojlim_{n}\left(\susp^n \circ \frgt_{n}^{\nut} \circ \augideal_n \circ \indc_{\infopLie}^{+, n} \circ \widetilde{\Loop}^{n}_{\infopLie} \circ \triv_{\infopLie}\right) \left(\frgt_{\infopLie}\left(L^{V}\right)\right) \\
    \overset{\rom{2}}{\simeq} &\varprojlim_{n}\left(\susp^n \circ \frgt_{n}^{\nut} \circ \augideal_n \circ \indc_{\infopLie}^{+, n} \circ \triv_{\infopLie}^{+, n} \circ \Cobar_n \circ \trivaug_0 \right)\left(\frgt_{\infopLie}\left(L^{V}\right)\right) \\ 
    \overset{\rom{3}}{\simeq} &\varprojlim_{n}\left(\susp^n \circ \frgt_{n}^{\nut} \circ \augideal_n \circ \left(\trivaug_0 \circ \indc_{\infopLie} \circ \triv_{\infopLie} \circ \augideal_0\right)^{n} \circ \Cobar_n \circ \trivaug_0 \right) \left(\frgt_{\infopLie}\left(L^{V}\right)\right) \\ 
    \overset{\rom{4}}{\simeq} &\varprojlim_{n}\left(\susp^n \circ (\indc_{\infopLie} \circ \triv_{\infopLie}) \circ \Loop^{n} \right) \left(\frgt_{\infopLie}\left(L^{V}\right)\right) \\
    \overset{\rom{5}}{\simeq} &\left(\Poly^1\left(\indc_{\infopLie} \circ \triv_{\infopLie}\right)\right) \left(\frgt_{\infopLie} L^{V}\right) \\
    \overset{\rom{6}}{\simeq} &\left(\Poly^1\left(\frgt_{\infopE_{\infty}^{\vee}} \circ \free_{\infopE_{\infty}^{\vee}} \right)\right) \left(\frgt_{\infopLie} L^{V}\right) \\
    \overset{\rom{7}}{\simeq} &\frgt_{\infopLie}\left(L^{V}\right).
  \end{align*}
  The reasons for the equivalences are given below:
    \begin{enumerate}
      \item is by~\cref{cor:XV-triv-lie}.
      \item holds by the equivalence~$\widetilde{\Loop}^{n}_{\infopLie} \circ \triv_{\infopLie} \simeq \triv_{\infopLie}^{+, n} \circ \Cobar_n \circ \trivaug_0$, because we have the equivalence~$\triv_{\infopLie}^{+, n} \circ \Cobar_n \simeq \widetilde{\Loop}^{n}_{\infopLie} \circ \triv_{\infopLie}^{+}$ of right adjoint functors.
      \item holds by the equivalence 
        \[
        \indc_{\infopLie}^{+} \circ \triv_{\infopLie}^{+} \simeq \trivaug_0 \circ \indc_{\infopLie} \circ \triv_{\infopLie} \circ \augideal_0,
        \] 
        and~$(\trivaug_0 \circ \indc_{\infopLie} \circ \triv_{\infopLie} \circ \augideal_0)^n$ denotes the functor on the~$\infty$\nobreakdash-category of augmented~$\infopE_n$\nobreakdash-algebras, induced by the composition~${\trivaug_0 \circ \indc_{\infopLie} \circ \triv_{\infopLie} \circ \augideal_0}$, which is lax monoidal. 
      \item holds by using the lax-monoidality of~$\indc_{\infopLie}^{+, n} \circ \triv_{\infopLie}^{+, n}$ and the equivalences
        \begin{align*}
          \frgt_{n}^{\nut} \circ \augideal_n &\simeq \augideal_0 \circ \frgt^{\aug} \\
          \frgt_n^{\aug} \circ \Cobar_n &\simeq \Loop^n_{\infAlg_{\infopE_0}(\infSp_{\Tel(h)})} \\
          \augideal_0 \circ \Loop^n_{\infAlg_{\infopE_0}(\infSp_{\Tel(h)})} \circ \trivaug_0 &\simeq \Loop^n.
        \end{align*}
      \item holds by the construction of the first colinear approximation~$\Poly^1(F)$ of~$F$ in the dual Goodwillie calculus tower, see~\cref{para:dual-Goodwillie}.
      \item  holds by operadic Koszul duality~\cref{prop:bar-O-algebras}: The Koszul dual~$\infty$\nobreakdash-cooperad of~$\infopLie$ is the cocommutative~$\infty$\nobreakdash-cooperad~$\infopE_{\infty}^{\vee}$.
      \item holds by~\cref{lem:Th-dual-calculus-sum} and the formula for~$\T_{\infopE_{\infty}^{\vee}} = \frgt_{\infopE_{\infty}^{\vee}} \circ \free_{\infopE_{\infty}^{\vee}}$.  \qedhere 
    \end{enumerate} 
\end{proof}

\begin{proof}[Proof of~\cref{thm:U-infty-ff}]
  This is a consequence of~\cref{prop:XV-tu-equi}, \cref{prop:proof-reduction-fully-faithfyl} and~\cref{prop:ff-left-adjoint}, as we explained in the proof strategy~\cref{para:U-infty-ff-strategy}.
\end{proof}

\subsection{Different construction of higher enveloping algebras}\label{sec:relation-to-KD}

One can construct two other commutative diagrams of the form~\eqref{diag:Lie-En-Th}, by Knudsen's construction of higher enveloping algebras in~\cite{Knu18}, and by the self Koszul duality of the~$\infty$\nobreakdash-operad~$\infopE_n$~\cite{CS22}, respectively.
In this section we discuss the relationships among these three constructions.
Throughout this section let~$\Cca$ denote a presentable stable symmetric monoidal~$\infty$\nobreakdash-category.

\begin{passage}[Operadic Koszul duality]\label{para:operadic-KD}
  An~$\infty$\nobreakdash-operad~$\Oca$ with values in~$\Cca$ is \emph{augmented} if there exists a map~$\Oca \to \infopTriv$ of~$\infty$\nobreakdash-operads.
  In other words, it is an augmented associative algebra object in the monoidal~$\infty$\nobreakdash-category~$\infSSeq(\Cca)$ of symmetric sequences in~$\Cca$.
  We recall briefly here what we mean by operadic Koszul duality. 
  For a detailed explanation including proofs, see~\cite[§5.3]{ShiTh}.
  
  A coassociative coalgebra in the monoidal~$\infty$\nobreakdash-category~$\infSSeq(\Cca)$ is called an~\emph{$\infty$\nobreakdash-cooperad} with values in~$\Cca$.
  Operadic Koszul duality, due to by Ginzburg--Kapranov~\cite{GK94}, exhibits a relationship between augmented operads and coaugmented cooperads.
  Let~$\infOpd^{\aug}(\Cca)$~(respectively~$\infcoOpd^{\coaug}(\infSp)$) denote the~$\infty$\nobreakdash-categories of augmented~$\infty$\nobreakdash-operads~(respectively coaugmented~$\infty$\nobreakdash-cooperads) with values in~$\Cca$.
  In \cite[Theorem~5.2.2.7]{HA} Lurie constructs the~$(\Bac \dashv \Cobar)$\nobreakdash-adjunction between the~$\infty$\nobreakdash-category of augmented associative algebras and the~$\infty$\nobreakdash-category of coaugmented coassociative coalgebras in a monoidal~$\infty$\nobreakdash-category which admits geometric realisations and totalisations.
  Thus, we can apply the adjunction to~$\infOpd^{\aug}(\Cca)$, which leads to an adjunction
  \begin{equation}\label{eq:cobar-bar-opd}
    \Bac \colon \infOpd^{\aug}(\Cca) \rightleftarrows \infcoOpd^{\coaug}(\Cca) \cocolon \Cobar.
  \end{equation}
  For an augmented~$\infty$\nobreakdash-operad~$\Oca$ we say~$\Bac(\Oca)$ is the \emph{Koszul dual~$\infty$\nobreakdash-cooperad} of~$\Oca$.
  Using the proof of~\cite[Proposition~3.47]{BCN23} one can see that taking arity-wise linear dual induces a~functor
  \begin{align*}
    \infcoOpd^{\coaug}(\Cca) &\to \infOpd^{\aug}(\Cca) \\
    \Lca & \mapsto \Lca^{\vee} \text{ with } \Lca^{\vee}(r) \simeq \MMap(\Lca(r), \munit_{\Cca}).
  \end{align*}
  We call~$\Bac(\Oca)^{\vee}$ the \emph{Koszul dual~$\infty$\nobreakdash-operad} of~$\Oca$, denoted by
  \[
  \KD(\Oca) \coloneqq \Bac(\Oca)^{\vee}
  \]
  
  An~$\infty$\nobreakdash-operad~($\infty$-cooperad)~$\Oca \in \infOpd(\Cca)$ is \emph{reduced} if it is non-unital and~$\Oca(1) \simeq \munit_{\Cca}$. 
  It is shown in~\cite{Chi05}~(model categorically) and in~\cite{HeuKD}~($\infty$\nobreakdash-categorically) that the adjunction~$\Bac \dashv \Cobar$ restricts to an equivalence on the~$\infty$\nobreakdash-subcategory of reduced~$\infty$\nobreakdash-operads and reduced~$\infty$\nobreakdash-cooperads.
  In the following we will work only with reduced~$\infty$\nobreakdash-operads.
\end{passage}

\begin{example}\label{ex:KD-Lie-Com}
  Let~$\infopE_{\infty} \in \infOpd(\infSp)$ denote the \emph{commutative~$\infty$\nobreakdash-operad}, where~$\infopE_{\infty}(r) \simeq \SS$ for~$r \in \NN$.
  The spectral Lie~$\infty$\nobreakdash-operad is defined as the Koszul dual of the deunitalisation~$\infopE_{\infty}^{\nut}$ of~$\infopE_{\infty}$, see~\cref{para:spectral-Lie-operad}.
\end{example}

\begin{example}
  Recall the non-unital~$\infty$\nobreakdash-operad~$\infopE_n^{\nut}$.
  Applying the suspension spectrum functor arity-wise induces a functor
  \[
  \susp^{\infty}_{+} \colon \infSSeq(\infHType) \to \infSSeq(\infSp).
  \]
  We denote the image of~$\infopE_n$ under this functor again by~$\infopE_n$, called the spectral~$\infopE_n$\nobreakdash-$\infty$\nobreakdash-operad.
  
  The main theorem of~\cite{CS22} shows that the Koszul dual of~$\infopE_n^{\nut} \in \infOpd(\infSp)$ is equivalent to the~$n$\nobreakdash-fold \emph{operadic desuspension}~$\opdsusp^{-n}\left(\infopE_n^{\nut}\right) \in \infOpd(\infSp)$ of~$\infopE_n^{\nut}$.
  Instead of going to the detailed construction on the operadic desuspension, let us just remark that~$\sigma^{-n}\left(\infopE_n^{\nut}\right)$ enjoys the property that there exists a commutative diagram
  \begin{equation}\label{diag:En-desuspension}
  \begin{tikzcd}
    \infAlg_{\infopE_n^{\nut}}(\Cca) \arrow[r, "\sim", "\phi_{\ast}"'] \arrow[d, "\frgt_{\infopE_n^{\nut}}"'] & \infAlg_{\opdsusp^{-n}\infopE_n^{\nut}}(\Cca) \arrow[d, "\frgt_{\opdsusp^{-n}\infopE_n^{\nut}}"] \\
    \Cca \arrow[r, "\susp_{\Cca}^{n}"', "\sim"] & \Cca.
  \end{tikzcd}
  \end{equation}
  See~\cite{HeiTh} or~\cite[§5.3.4]{ShiTh} for more details about operadic suspensions.
\end{example}

\begin{passage}[Nilpotent divided power coalgebras over an~$\infty$\nobreakdash-cooperad]
  Recall the monoidal functor functor~$\Tel_{\blank}$ from~\cref{para:symseq-compo}.
  The functor~$\T_{\blank}$ sends an~$\infty$\nobreakdash-cooperad~$\Lca$ with values in~$\Cca$ to an~$\infty$\nobreakdash-comonad~$\T_{\Lca}$.
  We call left comodules over~$\T_{\Lca}$ \emph{nilpotent divided power coalgebras} over~$\Lca$, see~\cite[Remark~5.3.1.12]{ShiTh} for some explanation about the terminology.
  Denote the~$\infty$\nobreakdash-category of nilpotent divided power coalgebras over~$\Lca$ as~$\infcoAlg_{\Lca}^{\ndp}(\Cca)$.
\end{passage}

Let~$\Oca$ be a reduced~$\infty$\nobreakdash-operad with values in~$\Cca$.
The augmentation~$\Oca \to \infopTriv_{\Cca}$ induces an adjunction
\[
\indc_{\Oca} \colon \infAlg_{\Oca}(\Cca) \rightleftarrows \Cca \cocolon \triv_{\Oca}.
\]
The~$\infty$-categories~$\infAlg_{\Oca}(\Cca)$ and~$\infcoAlg_{\Bac(\Oca)}^{\ndp}(\Cca)$ are related by the functor~$\indc_{\Oca}$.

\begin{proposition}\label{prop:bar-O-algebras}
  The functor~$\indc_{\Oca}$ factors through the~$\infty$\nobreakdash-category~$\infcoAlg_{\Bac(\Oca)}^{\ndp}(\Cca)$, given by the following commutative diagram 
  \[
  \begin{tikzcd}
  \infAlg_{\Oca}(\Cca) \arrow[rr, "\indc_{\Oca}"] \arrow[rd, "\widetilde{\indc}_{\Oca}"'] &              & \Cca \\
                          & \infcoAlg_{\Bac(\Oca)}^{\ndp}(\Cca) \arrow[ru, "\frgt_{\Bac(\Oca)}"'] &
  \end{tikzcd}
  \]
  of~$\infty$\nobreakdash-categories.
\end{proposition}

\begin{proof}[Sketch]
  One shows that the~$\infty$\nobreakdash-comonad~$\triv_{\Oca} \circ \indc_{\Oca}$ is equivalent to~$\T_{\Bac(\Oca)}$.
  For details of the proof, see for example~\cite[Proposition~5.3.2.4]{ShiTh}.
\end{proof}

Now we discuss the relationship between Knudsen's construction of higher enveloping algebras and our construction of the functor~$\UE_n \colon \infAlg_{\infopLie}(\infSp_{\Tel(h)}) \to \infAlg_{\infopE_n}(\infSp_{\Tel(h)})$.
For this purpose let us first consider a \emph{conjectural} generalisation of our construction~$\UE_n$ to an arbitrary presentable stable symmetric monoidal~$\infty$\nobreakdash-category~$\Cca$.

\begin{passage}[Chevalley--Eilenberg functor]\label{para:CE}
Using Lurie's theory of~$\infty$\nobreakdash-operads, define the~$\infty$\nobreakdash-category~$\infcoAlg_{\infopE_{\infty}}^{\aug}(\Cca)$ of augmented cocommutative coalgebras as
\[
\infcoAlg_{\infopE_{\infty}}^{\aug}(\Cca) \coloneqq \infAlg_{\infopE_{\infty} / \infopCom}^{\aug}(\Cca^{\op})^{\op}.
\]
One should think of an object in~$\infcoAlg_{/ \infopE_{\infty}}^{\aug}(\Cca)$ is an object of~$\Cca$ together with a ``comultiplication''.
\[
X \to \prod_{r = 0}^{\infty} \left(X^{\otimes r}\right)^{\Perm_{r}}.
\]
The forgetful functor~$\frgt_{\infopE_{\infty}}^{\aug, \op} \colon \infcoAlg_{/ \infopE_{\infty}}^{\aug}(\Cca) \to \infAlg_{\infopE_0}^{\aug}(\Cca)$ is symmetric monoidal with respect to the cartesian symmetric monoidal structure with respect to the source and the canonical symmetric monoidal structure on the target~$\infty$\nobreakdash-category, by~\cite[Example~3.2.4.4 and~Proposition~3.2.4.7]{HA}.

In~\cite{HeuKD} Heuts defines a natural transformation
\[
\frgt^{\mathrm{nil, dp}} \colon \infcoAlg_{\infopE_{\infty}}^{\aug, \ndp}(\Cca) \to \infcoAlg_{/ \infopE_{\infty}}^{\aug}(\Cca).
\]
Intuitively speaking the functor~$\frgt^{\mathrm{nil, dp}}$ is induced by the natural transformation 
\[
\coprod_{r = 0}^{\infty}\left((\blank)^{\otimes r}\right)_{\Perm_{r}} \to \prod_{r = 0}^{\infty}\left((\blank)^{\otimes r}\right)^{\Perm_{r}}.
\]
The \emph{Chevalley--Eilenberg functor} is defined as the composition
\[
\CE \colon \infAlg_{\infopLie}(\Cca) \xrightarrow{\widetilde{\indc}_{\infopLie}} \infcoAlg_{\infopCom}^{\nut, \ndp} \xrightarrow{\sim} \infcoAlg_{\infopCom}^{\aug, \ndp} \xrightarrow{\frgt^{\mathrm{nil, dp}}} \infcoAlg_{/ \infopE_{\infty}}^{\aug}(\Cca).
\]
Then we obtain the following equivalence of~functors.
\[
\frgt_{\infopE_{\infty}}^{\aug, \op} \circ  \CE \simeq \indc_{\infopLie}^{+} \coloneqq \trivaug_{\Cca} \circ \indc_{\infopLie}.
\]

We believe that~$\CE$ preserves product, as is already shown 1-categorically in~\cite[Chapter~6, §4.2.6]{GR17}.
Assuming this, the functor~$\indc_{\infopLie}^{+} \simeq \frgt_{\infopE_{\infty}}^{\aug, \op} \circ  \CE$ is symmetric monoidal.
Then we can proceed with the construction of~$\UE_n$ as in~\cref{para:Un-Tn}, for any presentable stable symmetric monoidal~$\infty$\nobreakdash-category~$\Cca$.
\end{passage}

Assuming that~$\CE$ preserves products we can give the~$\MK(h)$\nobreakdash-version of~\cref{thm:U-infty-ff}.

\begin{proposition}\label{thm:U-infty-ff-K(h)}
  Let~$\Cca$ be the~$\infty$\nobreakdash-category~$\infSp_{\MK(h)}$.
  In the situation of~\cref{para:CE}, \ie assuming that~$\CE$ preserves products, the induced~functor
  \[
  \UE_{\infty} \colon \infAlg_{\infopLie}\left(\infSp_{\MK(h)}\right) \to \varprojlim \infAlg_{\infopE_n}^{\nut}\left(\infSp_{\MK(h)} \right)
  \] 
  is fully faithful.
\end{proposition}

\begin{proof}
  One can prove this proposition using the same proof strategy~\cref{para:U-infty-ff-strategy} as for~\cref{thm:U-infty-ff}: Every statement holds after replacing~$\Tel(h)$ by~$\MK(h)$.
  In particular, 
  \begin{enumerate}
    \item By~\cite[Corollary~5.5.7.3]{Lur09} the~$\infty$\nobreakdash-category~$\infSp_{\MK(h)}$ is compactly generated by one compact object~$\Loc_{\MK(h)}\susp^{\infty}V_{h}$.
    \item \cref{prop:unit-finite-limit} holds with~$\MK(h)$ in place of~$\Tel(h)$: In the original proof we use the fact that~$\indc_{\infopLie}^{+}$ is symmetric monoidal.
    This is still true in the~$\MK(h)$\nobreakdash-local case, since the localisation functor~$\infSp_{\Tel(h)} \to \infSp_{\MK(h)}$ is product preserving, see~\cite[Lemma~1.4.4.7]{HA}. \qedhere
  \end{enumerate} 
\end{proof}

\begin{proposition}
  In the situation of~\cref{para:CE} our construction of the functor~$\UE_n$ agrees with Knuden's construction of the higher enveloping algebra functor~$\infAlg_{\infopLie}(\Cca) \to \infAlg_{\infopE_n}^{\nut}(\Cca)$ for every~$n \in \NN$.
  In particular, this is true for~$\Cca = \infSp_{\Tel(h)}$.
\end{proposition}

\begin{proof}
  This is due to~\cite[Theorem~B]{Knu18}, the symmetric monoidality of~$\indc_{\infopLie}^{+}$ and~\cref{prop:Loop-fac-En}.
\end{proof}

Now we discuss the construction of a similar-looking diagram as~\eqref{diag:Lie-En-Th} via self-Koszul duality of the spectral~$\infty$\nobreakdash-operads~$\infopE_{n}$, for~$n \in \NN$. 

\begin{construction}
  The topological little~$n$\nobreakdash-disks operads~$\opE_n$ fits into a tower of~inclusions
  \[
  \opE_n \hookrightarrow \opE_1 \hookrightarrow \cdots \hookrightarrow \opE_n \hookrightarrow \opE_{n+1} \hookrightarrow \cdots, 
  \]
  obtained by increasing the dimension of the disks.
  The homotopy colimit of this tower is the topological operad~$\opE_{\infty}$ where~$\opE_{\infty}(r)$ is contractible and admits a free~$\Perm_r$\nobreakdash-action.
  Thus, we obtain a tower of inclusions
  \begin{equation}\label{eq:En-inclusion}
    \infopE_0 \hookrightarrow \infopE_1 \hookrightarrow \cdots \hookrightarrow \infopE_n \hookrightarrow \infopE_{n+1} \cdots
  \end{equation}
  of~$\infty$\nobreakdash-operads whose colimit in~$\infOpd(\infHType)$ is equivalent to the commutative~$\infty$\nobreakdash-operad~$\infopE_{\infty}$.
  Applying Koszul duality to this tower, we obtain a tower
  \[
  \cdots \opdsusp^{-(n+1)}\infopE_{n+1} \xrightarrow{c_{n}^{n+1}} \opdsusp^{-n}\infopE_n \to \cdots \to \opdsusp^{-1}\infopE_1 \to \infopE_0
  \]
  of~$\infty$\nobreakdash-operads with values in~$\infSp$, whose inverse limit is equivalent to the~$\infty$\nobreakdash-operad~$\infopLie$, see~\cref{ex:KD-Lie-Com}.
  Denote the induced morphism~$\infopLie \to \susp^{-n}\infopE_{n}^{\nut}$ by~$c_n$.
  
  The morphism~$c_{n}$ induces the following commutative diagram
  \begin{equation}\label{diag:Lie-En-Frgt}
    \begin{tikzcd}
    \infAlg_{\infopE_n}^{\nut}(\Cca) \arrow[r, "\sim"] \arrow[d, "\frgt_{\infopE_n^{\nut}}"'] & \infAlg_{\susp^{-n}\infopE_{n}}^{\nut}(\Cca) \arrow[r, "(c_n)^{\ast}"] \arrow[d, "\frgt_{\susp^{-n}\infopE_{n}^{\nut}}"] &\infAlg_{\infopLie}(\Cca) \arrow[d, "\frgt_{\infopLie}"] \\
    \Cca \arrow[r, "\susp_{\Cca}^{n}"', "\sim"] & \Cca \arrow[r, "\id"', "\sim"] & \Cca
  \end{tikzcd}
  \end{equation}
  in~$\infPrr$, by~\cref{prop:operad-map-induced-functor} and \eqref{diag:En-desuspension}, for every~$n \in \NN$.
  By adjunction we obtain the following commutative~diagram
  \begin{equation}\label{diag:Lie-En-Free}
  \begin{tikzcd}
    \infAlg_{\infopLie}(\Cca) \arrow[r, "\widetilde{\UE}_n"] & \infAlg_{\infopE_n}^{\nut}(\Cca)\\
    \Cca \arrow[u, "\free_{\infopLie}"] \arrow[r, "\Loop^{n}_{\Cca}"', "\sim"] & \Cca \arrow[u, "\free_{\infopE_n^{\nut}}"']
  \end{tikzcd}
  \end{equation}
  in the~$\infty$\nobreakdash-category~$\infPrl$.
  Similarly, for each~$n \in \NN$, the morphism~$c_{n}^{n+1}$ induces the following commutative diagram in~$\infPrr$
  \[
    \begin{tikzcd}
    \infAlg_{\infopE_n}^{\nut}(\Cca) \arrow[r, "\sim"] \arrow[d, "\frgt_{\infopE_n^{\nut}}"'] & \infAlg_{\susp^{-n}\infopE_{n}}^{\nut}(\Cca) \arrow[r, "(c_n^{n+1})^{\ast}"] \arrow[d, "\frgt"] & \infAlg_{\susp^{-(n+1)}\infopE_{n+1}}^{\nut}(\Cca) \arrow[r, "\sim"] \arrow[d, "\frgt"] & \infAlg_{\infopE_{n+1}}^{\nut}(\Cca) \arrow[d, "\frgt_{\infopE_{n+1}^{\nut}}"] \\
    \Cca \arrow[r, "\susp_{\Cca}^{n}"', "\sim"] & \Cca \arrow[r, "\id"', "\sim"] & \Cca \arrow[r, "\Loop_{\Cca}^{n+1}"', "\sim"] & \Cca.
  \end{tikzcd}
  \]
  Taking left adjoints to the functors in the above diagrams gives the commutative diagram in~$\infPrl$ below
  \begin{equation}\label{diag:En-En-1-Free}
  \begin{tikzcd}
    \infAlg_{\infopE_{n+1}}^{\nut}(\Cca) \arrow[r, "\widetilde{\rB}_n"] & \infAlg_{\infopE_n}^{\nut}(\Cca)\\
    \Cca \arrow[u, "\free_{\infopE_n}^{\nut}"] \arrow[r, "\susp_{\Cca}"', "\sim"] & \Cca. \arrow[u, "\free_{\infopE_n}^{\nut}"']
  \end{tikzcd}
  \end{equation}
  Assembling the functors~$\widetilde{\UE}_n$ and~$\widetilde{\rB}_n$ together gives the commutative diagram~in~$\infPrl$ below:
  \begin{equation}\label{diag:Lie-En}
  \begin{tikzcd}[row sep = huge]
          & \infAlg_{\infopLie}\left(\Cca\right) \arrow[rrrrd] \arrow[rrrd, "\widetilde{\UE}_1"']  \arrow[rd] \arrow[d, "\widetilde{\UE}_n"'] \arrow[ld, ] &             &             &   \\
  \cdots \arrow[r] & \infAlg_{\infopE_{n}}^{\nut}\left(\Cca\right) \arrow[r, "\widetilde{\rB}_n"']                                                & \infAlg_{\infopE_{n-1}}^{\nut}\left(\Cca\right) \arrow[r, "\widetilde{\rB}_{n-1}"'] & \cdots \arrow[r] & \infAlg_{\infopE_{1}}^{\nut}\left(\Cca\right) \arrow[r, "\widetilde{\rB}_1"'] & \Cca.
  \end{tikzcd}
  \end{equation}
  Furthermore, we obtain an induced adjunction
  \[
  \widetilde{\UE}_{\infty} \colon \infAlg_{\infopLie}(\Cca) \rightleftarrows \varprojlim_n \infAlg_{\infopE_n}^{\nut}(\Cca) \cocolon \widetilde{\rUE}_{\infty}.
  \] 
\end{construction}

\begin{conjecture}
  In the situation of~\cref{para:CE}, \ie assuming that we can construct~$\UE_n$ for any presentable stable symmetric monoidal~$\infty$\nobreakdash-category~$\Cca$, the commutative diagram~\eqref{diag:Lie-En} is equivalent to our construction~\eqref{diag:Lie-En-Th}.
\end{conjecture}

\begin{remark}
  The above conjecture also relates to the question about comparing different notions of Koszul duality of~$\infopE_n$\nobreakdash-algebras, \cf~\cite{CS22},~\cite[§5.2]{HA} and~\cite[§3.2]{AF21}.
\end{remark}

\begin{theorem}\label{thm:U-infty-rational-equi}
For~$\Cca = \infSp_{\QQ}$, the functor~$\widetilde{\UE}_{\infty}$ is an equivalence of~$\infty$\nobreakdash-categories.
\end{theorem}

The proof of this theorem uses formality of~$\infopE_n$\nobreakdash-operads with values in~$\infSp_{\QQ}$.
Let~$\Oca$ be an~$\infty$\nobreakdash-operad with values in~$\infSp$.
\begin{enumerate}
  \item The \emph{rationalisation}~$\Oca_{\QQ}$ of~$\Oca$ is the~$\infty$\nobreakdash-operad with values in~$\infD(\QQ)$ induced by the rationalisation functor~$\blank \otimes \EM\QQ \colon \infSp \to \infMod_{\EM\QQ} \simeq \infD(\QQ)$.
    In particular, we have~$\Oca_{\QQ}(r) \simeq \Oca(r) \otimes \EM\QQ$, for every~$r \in \NN$.
  \item The \emph{rational homology}~$\infty$\nobreakdash-operad~$\Ho_{\bullet}\left(\Oca; \QQ \right)$ is induced by the rational homology functor~$\pi_{\bullet}^{\st} \circ (\blank \otimes \EM\QQ)$.
    In particular, we have~$\Ho_{\bullet}\left(\Oca; \QQ \right)(r) \simeq \Ho_{\bullet}\left(\Oca(r); \QQ \right)$, for every~$r \in \NN$.
    We consider~$\Ho_{\bullet}\left(\Oca; \QQ \right)$ as~$\infty$\nobreakdash-operads with values in~$\infD(\QQ)$~(with trivial differentials). 
\end{enumerate}
The content of the formality of~the rational~$\infopE_n$\nobreakdash-operads can be summarised as follows: 

\begin{theorem}[Fresse--Willwacher]\label{thm:En-formality}
  Let~$n \in \NN$.
  \begin{enumerate}
    \item There exists an equivalence~$(\infopE_n)_{\QQ} \simeq \Ho_{\bullet}\left(\infopE_{n}; \QQ \right)$ of~$\infty$\nobreakdash-operads with values in~$\infD(\QQ)$.
    \item Let~$i_n^{n+k} \colon \infopE_n \to \infopE_{n+k}$ denote the composition of the morphisms
    \[
    \infopE_n \to \infopE_{n+1} \to \cdots \infopE_{n+k}
    \] 
    of~$\infty$\nobreakdash-operads in~\eqref{eq:En-inclusion}.
    For every~$n \geq 0$ and every~$k \geq 2$, there exists the following commutative diagram
    \[
    \begin{tikzcd}
    (\infopE_n)_{\QQ} \arrow[d, "i_n^{n+k}"] \arrow[r, "\simeq"] & \Ho_{\bullet}(\infopE_n; \QQ) \arrow[d, "(i_n^{n+k})_{\ast}"] \arrow[r, two heads, "\pi"] & (\infopE_{\infty})_{\EM\QQ} \arrow[d, "\id"] \\
    (\infopE_{n+k})_{\QQ}           \arrow[r, "\simeq"]            & \Ho_{\bullet}(\infopE_{n+k}; \QQ) & (\infopE_{\infty})_{\EM\QQ} \arrow[l, hook', "\iota"]        
    \end{tikzcd}
    \]
    of~$\infty$\nobreakdash-operads with values in~$\infD(\QQ)$,~where
    \begin{enumerate}
      \item[b)] The map~$\pi$ is induced by the unique map~$\opE_n(r) \to \pt$ of topological spaces for all~$r \in \NN$ and for every~$n \in \NN$.
      \item[c)] The map~$\iota_{c}$ is induced by an inclusion~$\iota \colon \pt \hookrightarrow \opE_m(r)$ for every~$r \in \NN$.
        This is a well-defined map of operads because~$\opE_m(r)$ is connected for every~$m \geq 2$.
    \end{enumerate}
  \end{enumerate}
\end{theorem}

\begin{proof}
  See \cite[Theorem~B' and Theorem~D']{FW20} and~\cite[Proposition~0.3.5.a]{Fre11}.
\end{proof}

\begin{corollary}\label{cor:formality-En-map-Lie}
  For every~$n \geq 0$ and every~$k \geq 2$, the Koszul dual morphism 
  \[
  c_{n}^{n+k} \colon \susp^{-(n+k)} (\infopE_{n+k}^{\nut})_{\QQ} \to \susp^{-n}  (\infopE_{n}^{\nut})_{\QQ}
  \] 
  of~$i_n^{n+k}$ factors through~$\infopLie_{\QQ}$.
  More precisely, there exists the following commutative diagram
  \[
  \begin{tikzcd}
    \susp^{-n} (\infopE_{n}^{\nut})_{\QQ} & \infopLie_{\QQ} \arrow[l]\\
    \susp^{-n-k}(\infopE_{n+k}^{\nut})_{\QQ} \arrow[r] \arrow[u, "(c^{n+k}_{n})_{\ast}"]           & \infopLie_{\QQ} \arrow[u, "\id"'] 
  \end{tikzcd}
  \]
  of~$\infty$\nobreakdash-operads with values in~$\infD(\QQ)$.
\end{corollary}

\begin{proof}
  We apply the Koszul duality functor to the outer commutative diagram in~\cref{thm:En-formality}.\rom{2}.
\end{proof}

\begin{proof}[Proof of~\cref{thm:U-infty-rational-equi}]
  It is equivalent to show that 
  \[
  \UE_{\infty} \colon \infAlg_{\infopLie}\left(\infD(\QQ)\right) \to \varprojlim_n \infAlg_{\infopE_n}\left(\infD(\QQ)\right)
  \] 
  is an equivalence, by the equivalence~$\infSp_{\QQ} \simeq \infD(\QQ)$ of symmetric monoidal~$\infty$\nobreakdash-categories.
  Recall that we have the equivalence
  \[
  \infAlg_{\infopE_n}(\infD(\QQ)) \simeq \infAlg_{(\infopE_n)_{\QQ}}(\infD(\QQ))
  \]
  by the unique symmetric monoidal functor~$\infHType \xrightarrow{\susp^{\infty}_{+}} \infSp \xrightarrow{\Loc_{\QQ}} \infSp_{\QQ} \xrightarrow{\sim} \infD(\QQ)$ in~$\infPrl$.
  
  Abbreviate the left adjoint to the induced functor~$(c_{n}^{n+1})_{\ast}$ on algebras by~$c^{!}$.
  Denote the composition~${\widetilde{\rB}_{n+k} \circ \widetilde{\rB}_{n+k-1} \circ \cdots \circ \widetilde{\rB}_{n+1}}$ by~$\widetilde{\rB}_{n+k, n}$.
  By~\cref{cor:formality-En-map-Lie} we obtain the following commutative diagram
  \[
  \begin{tikzcd}[row sep = large]
  \infAlg_{\infopE_{n+k}}^{\nut}(\infD(\QQ)) \arrow[d, "\simeq"] \arrow[r, "\widetilde{\rB}_{n+k, n}"] & \infAlg_{\infopE_{n}}^{\nut}(\infD(\QQ))           \\
  \infAlg_{\susp^{-n-k}\infopE_{n+k}}^{\nut}(\infD(\QQ)) \arrow[d] \arrow[r, "c^{!}"] & \infAlg_{\susp^{-n}\infopE_{n}}^{\nut}(\infD(\QQ))  \arrow[u, "\simeq"] \\
  \infAlg_{\infopLie}(\infD(\QQ)) \arrow[r, "\id"]           &  \infAlg_{\infopLie}(\infD(\QQ)) \arrow[u]
  \end{tikzcd}
  \]
  of~$\infty$\nobreakdash-categories, for every~$n \geq 0$ and every~$k \geq 2$.
  This gives the commutative diagram of~$\infty$\nobreakdash-categories below:
  \[
  \begin{tikzcd}[row sep = large, column sep = small]
  \cdots \arrow[r] \arrow[d]  & \infAlg_{\infopE_{n+2}}^{\nut}\left(\infD(\QQ)\right) \arrow[r, "\rB_{n+2, n}"] \arrow[d]  & \infAlg_{\infopE_{n}}^{\nut}\left(\infD(\QQ)\right) \arrow[r, "\rB_{n, n-2}"] \arrow[d]  & \cdots \arrow[r, "\rB_{2, 0}"] & \infAlg_{\infopE_{0}}^{\nut}\left(\infD(\QQ)\right)\arrow[d, "\id"]\\
  \cdots \arrow[r, "\id"] \arrow[ru] & \infAlg_{\infopLie}\left(\infD(\QQ)\right) \arrow[ru] \arrow[r, "\id"] & \infAlg_{\infopLie}\left(\infD(\QQ)\right) \arrow[ru] \arrow[r, "\id"] & \cdots \arrow[r, "\id"] \arrow[ru] & \infAlg_{\infopE_{0}}^{\nut}\left(\infD(\QQ)\right).  
  \end{tikzcd}
  \]
  Therefore, the inverse limit of the upper row is equivalent to the inverse limit of lower horizontal rows, and the latter is equivalent to~$\infAlg_{\infopLie}\left(\infD(\QQ)\right)$ because the diagram becomes constant after the first arrow.
\end{proof}

Since we consider~$\infSp_{\Tel(h)}$, for every~$1 \leq h \in \NN$, as a higher height analogue of~$\infSp_{\QQ}$, we make the following conjecture.

\begin{conjecture}\label{conj:U-finty-equi-Sp-Th}
  The functor~$\UE_{\infty}$, or the functor~$\widetilde{\UE}_{\infty}$ in the case where~$\Cca \simeq \infSp_{\Tel(h)}$, is an equivalence of~$\infty$\nobreakdash-categories.
  In other words, we have
  \[
  \infAlg_{\infopLie}\left(\infSp_{\Tel(h)}\right) \simeq \varprojlim_n \infAlg_{\infopE_n}\left(\infSp_{\Tel(h)}\right).
  \]
\end{conjecture}

\begin{passage}
  Let~$\Cca$ be the derived~$\infty$\nobreakdash-category~$\infD(k)$ of a field~$k$ of characteristic~0.
  Then we can also interpret the diagram~\eqref{diag:Lie-En} in terms of deformation theory and formal moduli problems. 
  See~\cite{CG21, DAGX, BM} for some introductions to the relationship between Koszul duality and deformation theory.
  It is shown in \cite{DAGX, Pri10} that the~$\infty$\nobreakdash-category~$\infModuli_{k}$ of (commutative) moduli problems over~$k$ is equivalent to the~$\infty$\nobreakdash-category~$\infAlg_{\infopLie}(\infD(k))$ of spectral Lie algebras in~$\infD(k)$.
  Furthermore, the~$\infty$\nobreakdash-category~$\infModuli_{\infopE_n, k}$ of~$\infopE_n$\nobreakdash-moduli problems is equivalent to the~$\infty$\nobreakdash-category~$\infAlg_{\infopE_n}(\infD(k))$.
  The inclusion~$\infopE_n \hookrightarrow \infopE_{\infty}$ of~$\infty$\nobreakdash-operads induces a small-colimit-preserving functor~${\infModuli_{k} \to \infModuli_{\infopE_n, k}}$ by left Kan extension, which should be equivalent to the functor~$\widetilde{\UE}_n$.
  Therefore, another interpretation of~\cref{thm:U-infty-rational-equi} is that the~$\infty$\nobreakdash-category~$\infModuli_{k}$ of commutative formal Moduli problems over~$k$ is equivalent to the inverse limit of the~$\infty$\nobreakdash-category~$\infModuli_{\infopE_n, k}$ of~$\infopE_{n}$\nobreakdash-formal Moduli problems over~$k$.
 
  Furthermore, the association between formal moduli problems and Lie algebras is generalised to positive and mixed characteristic situations, see~\cite{BM}.
  This might help with understanding the functor~$\widetilde{\UE}_{\infty}$ in the case where~$\Cca$ is the~$\infty$\nobreakdash-category of module spectra over a field of positive characteristic.
\end{passage}

\section{The costabilisation of the unstable monochromatic layers}\label{sec:costab}

The goal of this section is to prove the universal property of the Bousfield--Kuhn functor, see~\cref{thm:universal-property-BK}.
We achieve this in~\cref{sec:costab-vh} by proving that the Bousfield--Kuhn adjunction~$\lBK_h \dashv \BK_h$ exhibits the~$\infty$\nobreakdash-category~$\infSp_{\Tel(h)}$ as the costabilisation of~$\infHType_{v_h}$~(see~\cref{thm:costab-Lie-Tn}).
In the first two sections~\crefrange{sec:costab-basic}{sec:costab-example} we introduce the theory of costabilisation and give some examples.
Another work discussing the theory of costabilisation of~$\infty$\nobreakdash-categories is~\cite{FraTh}.

\subsection{The costabilisation of an \texorpdfstring{$\infty$}{∞}-category}\label{sec:costab-basic}\label{sec:expl-costab}

In this section we introduce the theory of costabilisation, the dual notion of stabilisation.

\begin{passage}[Stabilisation of an~$\infty$\nobreakdash-category]\label{para:stabilisation}
  Let us first recall some basic notions of stabilisation, following~\cite[§1.4]{HA}.
  A functor between~$\infty$\nobreakdash-categories is \emph{excisive} if it sends pushout diagrams to pullback diagrams~(given that both exist in the source and target~$\infty$\nobreakdash-categories, respectively).
  
  Let~$\Dca$ be an~$\infty$\nobreakdash-category that admits finite limits.
  The \emph{stabilisation}~$\infSp(\Dca)$ of~$\Dca$ is defined as the stable~$\infty$\nobreakdash-category 
  \[
  \infSp(\Dca) \coloneqq \infExc_{\ast}(\infHType_{\ast}^{\fin}, \Dca)
  \] 
  of reduced (preserves terminal objects) excisive functors from the~$\infty$\nobreakdash-category~$\infHType_{\ast}^{\fin}$ of pointed finite \htypes to~$\Dca$, see~\cite[§1.4.2]{HA}.
  
  Let~$\Dca_{\ast}$ denote the~$\infty$\nobreakdash-category of pointed objects of~$\Dca$. 
  The forgetful functor~$\Dca_{\ast} \to \Dca$ induces an equivalence~$\infSp(\Dca) \simeq \infSp(\Dca_{\ast})$ of their stabilisations, using the fact that a stable~$\infty$\nobreakdash-category is pointed, see~\cite[Remark~1.4.2.18]{HA}.
  Furthermore, there exists an~equivalence
  \[
  \infSp(\Dca_{\ast}) \simeq \varprojlim\left( \cdots \xrightarrow{\Loop_{\Dca}} \Dca_{\ast} \xrightarrow{\Loop_{\Dca}} \Dca_{\ast}\right)
  \]
  of stable~$\infty$\nobreakdash-categories, see~\cite[Proposition~1.4.2.24]{HA}.
  Thus, we think of an object in~$\infSp(\Dca_{\ast})$ as a sequence~$(X_{n})_{n \geq 0}$ of objects of~$\Cca$ such that~$\Loop_{\Dca}(X_{n+1}) \simeq X_{n}$, for every~$n \geq 0$.
  From this we obtain~$\infSp(\infHType_{\ast}) \simeq \infSp$.
  
  The stabilisation~$\infSp(\Dca)$ is equipped with a canonical functor 
  \[
  \Loop^{\infty}_{\Dca} \colon \infSp(\Dca) \to \Dca.
  \]
  Depending on which definitions one uses for the stabilisation, one can consider this functor either as 
  \begin{enumerate}
    \item the evaluation of an excisive functor at~$\usphere^0 \in \infHType^{\fin}_{\ast}$, or
    \item the functor assigning to~$(X_{n})_{n \geq 0}$ the object~$X_0$, if~$\Dca$ is pointed.
  \end{enumerate}
  The functor~$\Loop^{\infty}_{\Dca}$ satisfies the universal property that it is the terminal finite-limit-preserving functor from a stable~$\infty$\nobreakdash-category to~$\Dca$, see~\cite[Corollary~1.4.2.23]{HA}.
  
  If~$\Dca$ is in addition presentable, the functor~$\Loop^{\infty}_{\Dca}$ admits a left adjoint, denoted by~$\susp^{\infty}_{\Dca}$. 
  Sometimes we abbreviate~$\susp^{\infty}_{\Dca}$ and~$\Loop^{\infty}_{\Dca}$ by~$\susp^{\infty}$ and~$\Loop^{\infty}$, respectively. 
\end{passage}

\begin{definition}\label{def:costab}
  Let~$\Cca$ be an~$\infty$\nobreakdash-category which admits finite colimits.
  The \emph{costabilisation}~$\infcoSp(\Cca)$ of~$\Cca$ is defined as
  \[
  \infcoSp(\Cca) \coloneqq \infSp(\Cca^{\op})^{\op}.
  \] 
  An object of~$\infcoSp(\Cca)$ is called a \emph{cospectrum object} of~$\Cca$.
\end{definition}

\begin{example}\label{ex:costab-triv-ex}
  Here are two rather elementary examples.
  \begin{enumerate}
    \item The costabilisation of the~$\infty$\nobreakdash-category~$\infHType$ of \htypes is trivial, \ie there exists an equivalence~$\infcoSp(\infHType) \simeq \pt$ of~$\infty$\nobreakdash-categories. 
      Let~${F \colon \infHType_{\ast}^{\fin} \to \infHType^{\op}}$ be a reduced excisive functor.
      Then we know~${F(\pt) = \emptyset}$~(the terminal object in~$\infHType^{\op}$). 
      Let~$F(\usphere^0) = X$.
      Then the pointed map~$\pt \to \usphere^0$ is sent to~$X \to \emptyset$ in~$\infHType$ under~$F$, which implies that~$X$ is also the empty set. 
      More generally, let~$\Cca$ be an~$\infty$\nobreakdash-category with strict initial objects\footnote{An initial object~$I$ is strict if any morphism mapping into~$I$ is an equivalence}, then~$\infcoSp(\Cca)$ is trivial by the same arguments.
    \item Let~$\Cca$ be a stable~$\infty$\nobreakdash-category.
      Then we have
      \[
      \infcoSp(\Cca) = \left(\infExc_{\ast}(\infHType_{\ast}^{\fin}, \Cca^{\op})\right)^{\op} \overset{(a)}{\simeq} \left(\infFun^{\rex}(\infHType_{\ast}, \Cca^{\op})\right)^{\op} \overset{(b)}{\simeq} \left(\Cca^{\op}\right)^{\op} \simeq \Cca,
      \]
      where~$\infFun^{\rex}$ denotes the~$\infty$\nobreakdash-category of right exact functors. 
      The equivalence~$(a)$ holds because pullbacks and pushouts in a stable~$\infty$\nobreakdash-category coincide, see~\cite[Proposition~1.1.3.4]{HA}, and a functor preserves finite colimits if it preserves the initial object and pushouts, see~\cite[Corollary~4.4.2.5]{Lur09}.
      The equivalence~$(b)$ holds because~$\infHType^{\fin}_{\ast}$ the~$\infty$\nobreakdash-category freely generated by the~$\pt$ under finite colimits, see~\cite[Remark~1.4.2.6]{HA}.
  \end{enumerate}
\end{example}

\begin{proposition}
  Let~$\Cca$ be an~$\infty$\nobreakdash-category which admits finite colimits. 
  The costabilisation~$\infcoSp(\Cca)$ of~$\Cca$ is a stable~$\infty$\nobreakdash-category.
\end{proposition}

\begin{proof}
  The stabilisation of an~$\infty$\nobreakdash-category admitting finite limits is stable, see~\cite[Proposition~1.4.2.17]{HA}.
  The opposite~$\infty$\nobreakdash-category of a stable~$\infty$\nobreakdash-category is stable, see~\cite[Remark~1.1.1.3]{HA}.
\end{proof}

\begin{notation}
  Let~$\Cca$ be an~$\infty$\nobreakdash-category which admits finite colimits. 
  Denote the opposite of the functor~$\Loop^{\infty}_{\Cca^{\op}} \colon \infSp(\Cca^{\op}) \to \Cca^{\op}$~(see~\cref{para:stabilisation}) by~$\susp_{\infty}^{\Cca} \colon \infcoSp(\Cca) \to \Cca$.
  We abbreviate~$\susp_{\infty}^{\Cca}$ by~$\susp_{\infty}$ if it is clear from the context which~$\infty$\nobreakdash-category~$\Cca$ we are working~with. 
\end{notation}

\begin{proposition}
  The functor~$\susp_{\infty} \colon \infcoSp(\Cca) \to \Cca$ preserves finite colimits.
\end{proposition}

\begin{proof}
  By \cite[Remark~1.4.2.3]{HA} the functor~$\Loop^{\infty}_{\Cca^{\op}}$ preserves finite limits, since finite limits of functors are computed point-wise. 
  Thus the functor~$\susp_{\infty}^{\Cca} \simeq \left(\Loop^{\infty}_{\Cca}\right)^{\op}$ preserves finite colimits. 
\end{proof}

\begin{proposition}\label{prop:costab-uni-prop}
  Let~$\Cca$ be an~$\infty$\nobreakdash-category admitting finite colimits.
  The costabilisation~$\infcoSp(\Cca)$ together with the functor~$\susp_{\infty} \colon \infcoSp(\Cca) \to \Cca$ satisfies the following universal property:
  Any finite-colimit-preserving functor~$F$ from a stable~$\infty$\nobreakdash-category~$\Dca$ to~$\Cca$ factors through~$\susp_{\infty} \colon \infcoSp(\Cca) \to \Cca$, uniquely up to contractible choice.
  We illustrate the universal property by the following commutative diagrams.
  \[
  \begin{tikzcd}
    \Dca \arrow[rr, "F"] \arrow[rd, dashed, "\exists !"'] && \Cca \\ 
    &\infcoSp\left(\Cca\right). \arrow[ru, "\susp_{\infty}"']
  \end{tikzcd}
  \]
\end{proposition}

\begin{proof}
  The functor~$\Loop^{\infty}_{\Cca^{\op}}$ satisfies the universal property that it is the terminal finite-limit-preserving functor from a stable~$\infty$\nobreakdash-category to~$\Cca^{\op}$, see~\cite[Corollary~1.4.2.23]{HA}.
  Taking the opposite~$\infty$\nobreakdash-categories, we obtain the universal property of~$\susp_{\infty}(\Cca)$.
\end{proof}

\begin{example}
  Let~$\Cca$ be an~$\infty$\nobreakdash-category which admits finite colimits.
  The~$\infty$\nobreakdash-cate\-gory~$\infcoExc_{\ast}\left(\infHType_{\ast}, \Cca\right)$ of reduced coexcisive functors from~$\infHType_{\ast}$ to~$\Cca$ is a stable~$\infty$\nobreakdash-category, since it is equivalent to the~$\infty$\nobreakdash-category~$\infExc_{\ast}\left((\infHType_{\ast})^{\op}, \Cca^{\op}\right)$, which is stable by~\cite[Proposition~1.4.2.16]{HA}.
  Evaluation at a fixed object~$X \in \Cca$ gives a colimit-preserving functor~${\ev_{X} \colon \infcoExc_{\ast}\left(\infHType_{\ast}, \Cca\right) \to \Cca}$, since the colimits of functors are computed point-wise.
  By the universal property of costabilisation, we obtain the following commutative~diagram
  \[
  \begin{tikzcd}[row sep = large]
  \infcoExc_{\ast}\left(\infHType_{\ast}, \Cca\right)^{\op} \arrow[rr, "\ev_{X}"] \arrow[rd, "{\left(\infMap_{\ast}(\blank, X)\right)^{\ast}}"'] && \Cca \\
  & \infcoSp(\Cca) \simeq \infExc_{\ast}\left(\infHType_{\ast}^{\fin}, \Cca^{\op}\right)^{\op} \arrow[ru, "\ev_{\usphere^0}"']  
  \end{tikzcd}
  \]
  where the functor~$\left(\infMap(\blank, X)\right)^{\ast}$ is given by composing with the functor 
  \[
  \infMap(\blank, X) \colon \infHType^{\fin}_{\ast} \to \infHType_{\ast}^{\op}, \ V \mapsto \infMap(V, X).
  \]
\end{example}

\begin{proposition}\label{prop:costab-infinite-suspension}
  Let~$\Cca$ be a pointed~$\infty$\nobreakdash-category admitting finite colimits. 
  Then there exists an equivalence
  \begin{equation}\label{eq:costab-infinite-suspension}
    \infcoSp(\Cca) \simeq \varprojlim\left(\cdots \xrightarrow{\susp_{\Cca}} \Cca \xrightarrow{\susp_{\Cca}} \Cca\right)
  \end{equation}
  of stable~$\infty$\nobreakdash-categories, where the inverse limit is take in the~$\infty$\nobreakdash-category~$\infCAT_{\infty}$ of (not necessarily small)~$\infty$\nobreakdash-categories.
\end{proposition}

\begin{proof}
  The opposite~$\infty$\nobreakdash-category~$\Cca^{\op}$ is pointed and admits finite limits. 
  By \cite[Proposition 1.4.2.24]{HA}, we have
  \[
    \begin{split}
      \infSp\left({\Cca^{\op}}\right) 
      &\simeq \varprojlim\left(\cdots \xrightarrow{\Loop_{\Cca^{\op}}} \Cca^{\op} \xrightarrow{\Loop_{\Cca^{\op}}} \Cca^{\op} \right) \\
      &\simeq \left(\varprojlim\left(\cdots \xrightarrow{\susp_{\Cca}} \Cca \xrightarrow{\susp_{\Cca}} \Cca \right)\right)^{\op}.
    \end{split}
  \]
  Then the proposition follows from the definition of costabilisation. 
\end{proof}

\begin{remark}
  Let~$\Cca$ be a pointed~$\infty$\nobreakdash-category admitting finite colimits.
  An object of the~$\infty$\nobreakdash-category
  \[
  \varprojlim\left(\cdots \xrightarrow{\susp_{\Cca}} \Cca \xrightarrow{\susp_{\Cca}} \Cca\right)
  \] 
  is a sequence~$(X_i)_{i \geq 0}$ of objects in~$\Cca$ such that~$X_{i} \simeq \susp_{\Cca}(X_{i+1})$ for every~$i \geq 0$.
  In particular, we see that  
  \begin{enumerate}
    \item for every~$i \geq 0$ the object~$X_i \in \Cca$ admits infinite desuspensions in~$\Cca$, and 
    \item the functor~$\susp_{\infty}^{\Cca}$ is equivalent to the canonical functor 
      \[
      \varprojlim\left(\cdots \xrightarrow{\susp_{\Cca}} \Cca \xrightarrow{\susp_{\Cca}} \Cca\right) \to \Cca, \ (X_i)_{i \geq 0} \mapsto X_0,
      \]
      under the equivalence~\eqref{eq:costab-infinite-suspension}.
  \end{enumerate}
\end{remark}

\begin{corollary}
  The costabilisation of the~$\infty$\nobreakdash-category~$\infHType_{\ast}$ of pointed \htypes is trivial.
\end{corollary}

\begin{proof}
  The suspension functor increases the connectivity of pointed \htypes. 
  In particular, any pointed \htype that admits infinite desuspensions is contractible. 
\end{proof}

\begin{remark}
  Let~$\Cca$ be an~$\infty$\nobreakdash-category admitting finite limits and finite colimits.
  Denote the final object of~$\Cca$ by~$t_{\Cca}$.
  Let~$\Cca_{\ast} \coloneqq \Cca_{t_{\Cca} /}$ denote the~$\infty$\nobreakdash-category of pointed objects of~$\Cca$, see~\cite[Definition~7.2.2.1]{Lur09}.
  Note that the canonical forgetful functor~$\Cca_{\ast} \to \Cca$ induces an equivalence~$\infSp\left(\Cca_{\ast}\right) \simeq \infSp\left(\Cca\right)$ of stable~$\infty$\nobreakdash-categories.
  However, the costabilisations~$\infcoSp(\Cca_{\ast})$ and~$\infcoSp(\Cca)$ are in general not equivalent~$\infty$\nobreakdash-categories.
  For example, let~$\Cca$ be~$\infHType^{\op}$.
  We have~$\left(\infHType^{\op}\right)_{\ast} \simeq \{\emptyset\}$.
  Thus, the costabilisation~$\infcoSp(\left(\infHType^{\op}\right)_{\ast})$ is trivial, whereas~$\infcoSp(\infHType^{\op}) \simeq \infSp(\infHType)^{\op} \simeq \infSp^{\op}$ is not trivial. 
\end{remark}

\begin{proposition}
  Let~$\Cca$ be a pointed presentable~$\infty$\nobreakdash-category.
  \begin{enumerate}
    \item The costabilisation~$\infcoSp(\Cca)$ is presentable.
    \item The functor~$\susp_{\infty}^{\Cca} \colon \infcoSp(\Cca) \to \Cca$ admits a right adjoint~$\Loop_{\infty}^{\Cca}$. 
    \item Let~$\Dca$ be a stable presentable~$\infty$\nobreakdash-category. 
      An exact functor~$F \colon \Dca \to \infcoSp(\Cca)$ admits a right adjoint if and only if~$\susp_{\infty} \circ F \colon \Dca \to \Cca$ admits a right adjoint.  
  \end{enumerate}
\end{proposition}

\begin{proof}
  This is the analogue of \cite[Proposition~1.4.4.4]{HA}.
  Since~$\Cca$ is presentable, the functor~$\susp_{\Cca}$ admits a right adjoint~$\Loop_{\Cca}$, see~\cite[Remark~1.1.2.8]{HA}. 
  The inverse limit 
  \[
  \varprojlim\left(\cdots \xrightarrow{\susp_{\Cca}} \Cca \xrightarrow{\susp_{\Cca}} \Cca\right)
  \] 
  in~$\infPrl$ is preserved under the inclusion~$\infPrl \hookrightarrow \infCAT_{\infty}$, see~\cite[Proposition~5.5.3.13]{Lur09}.
  Thus, $\infcoSp(\Cca)$ is an object of~$\infPrl$ and~$\susp_{\infty}^{\Cca}$ is a morphism in~$\infPrl$, by~\cref{prop:costab-infinite-suspension}.
  This proves \rom{1} and \rom{2}.
  Property \rom{3} follows from the universal property of the costabilisation~$\susp_{\infty} \colon \infcoSp(\Cca) \to \Cca$ in the~$\infty$\nobreakdash-category~$\infPrl$.
\end{proof}
 
\begin{corollary}\label{cor:univ-right-adj-costab}
  Let~$\Cca$ and~$\Dca$ be pointed presentable~$\infty$\nobreakdash-categories and assume that~$\Dca$ is stable. 
  Let~$G \colon \Cca \to \Dca$ be a functor that admits a left adjoint.  
  Then~$G$ factors through~$\Loop_{\infty} \colon \Cca \to \infcoSp(\Cca)$, uniquely up to contractible choice. 
  We illustrate this universal property of~$\Loop_{\infty}^{\Cca}$ by the following commutative diagram: 
  \[
  \begin{tikzcd}[row sep = large]
  \Cca\arrow[rr, "G"] \arrow[rd, "\Loop_{\infty}"'] & & \Dca \\
   & \infcoSp(\Cca). \arrow[ru, dashed,  "\exists !"']
  \end{tikzcd}
  \]
\end{corollary}

\begin{proof}
  This follows from~\cref{prop:costab-uni-prop} by adjunctions. 
\end{proof}

\begin{proposition}\label{lem:ff-costab}
  Let~$\Cca$ and~$\Dca$ be~$\infty$\nobreakdash-categories admitting finite colimits and~$F \colon \Cca \to \Dca$ be a functor preserving initial objects. 
  \begin{enumerate}
    \item If~$F$ is right exact, then~$F$ induces an exact functor~$\FF \colon \infcoSp(\Cca) \to \infcoSp(\Dca)$ of stable~$\infty$\nobreakdash-categories.
    \item Assuming that~$\Cca$ and~$\Dca$ are pointed~$\infty$\nobreakdash-categories and that there exists an equivalence~$F \circ \susp_{\Cca} \simeq \susp_{\Dca} \circ F$, then~$F$ induces an exact functor~$\FF$ of their costabilisations.
    \item In the situation of $\rom{1}$ and \rom{2}, if~$F$ is fully faithful, the induced functor~$\FF$ is also fully~faithful.
  \end{enumerate}
\end{proposition}

\begin{proof}
    \itemnum Since~$F$ is right exact, it induces a right exact functor
      \[
      \FF \colon \infExc_{\ast}\left(\infHType^{\fin}_{\ast}, \Cca^{\op}\right)^{\op} \to  \infExc_{\ast}\left(\infHType^{\fin}_{\ast}, \Dca^{\op}\right)^{\op}
      \]
      between stable~$\infty$\nobreakdash-categories, since limits of (excisive) functors are computed pointwise, see~\cite[Remark~1.4.2.3]{HA}.
      By~\cite[Proposition~1.1.4.1]{HA}, the functor~$\FF$ is~exact.
      
    \itemnum Let~$G \colon \infHType^{\fin}_{\ast} \to \Cca^{\op}$ be a reduced and excisive functor.
      We show that the composition~$F^{\op} \circ G$ is also reduced and excisive, where~$F^{\op}$ denotes the canonical functor~$\Cca^{\op} \to \Dca^{\op}$ associated to~$F$.
      Let~$P$ be a pushout diagram in~$\infHType^{\fin}_{\ast}$, so we have~$G(P)$ is a pullback diagram in~$\Cca^{\op}$.
      By~\cite[Lemma~3.9]{Heu21}~$(\Loop_{\Dca^{\op}} \circ F^{\op})(G(P))$ is a pullback diagram in~$\Dca^{\op}$.
      Since~$\infExc_{\ast}\left(\infHType^{\fin}_{\ast}, \Dca^{\op}\right)$ is stable and its limits are computed pointwise~$\Dca^{\op}$, we have that~$F(G(P))$ is a pullback diagram in~$\Dca^{\op}$.
      Similarly, we can show that~$\FF$ is right exact and thus~exact.
      
    \itemnum Fully faithfulness of~$\FF$ is checked pointwise on the induced map on mapping spaces in~$\Cca^{\op}$ and~$\Dca^{\op}$, which can be verified by the fully faithfulness of~$F$. \qedhere
\end{proof}
\subsection{Examples of costabilisation}\label{sec:costab-example}

Let~$\Oca$ be a reduced~$\infty$\nobreakdash-operad with values in a presentable stable symmetric monoidal~$\infty$\nobreakdash-category~$\Cca$.
The~$\infty$\nobreakdash-category~$\infAlg_{\Oca}(\Cca)$ of~$\Oca$\nobreakdash-algebras is pointed: The zero object of~$\infAlg_{\Oca}(\Cca)$ is the zero object of~$\Cca$ endowed with a trivial~$\Oca$\nobreakdash-algebra structure.
Recall that the functor~$\free_{\Oca} \colon \Cca \to \infAlg_{\Oca}(\Cca)$ preserves small colimits.
Since~$\Cca$ is stable, the suspension functor~$\susp_{\Cca}$ of~$\Cca$ is an auto-equivalence of~$\Cca$, with the inverse given by the loop functor~$\Cca$.
Denote the suspension endofunctor of~$\infAlg_{\Oca}(\Cca)$ by~$\susp_{\Oca}$, which admits a right adjoint~$\Loop_{\Oca}$.
For every object~$X \in \Cca$, we have
\[
  \begin{split}
    \free_{\Oca}(X) 
    &\simeq \free_{\Oca}\left(\susp_{\Cca} \Loop_{\Cca} X\right) \\
    &\simeq \susp_{\Oca}\left(\free_{\Oca}(\Loop_{\Cca}X)\right) \\
    &\simeq \susp_{\Oca}^{2}\left(\free_{\Oca}(\Loop_{\Cca}^2X)\right) \\
    &\simeq \cdots \\
    &\simeq \susp_{\Oca}^{\infty}\left(\free_{\Oca}(\Loop_{\Cca}^{\infty} X)\right).
  \end{split}
\]
This implies that every free~$\Oca$\nobreakdash-algebra admits infinite desuspensions in~$\infAlg_{\Oca}(\Cca)$, which corresponds to a cospectrum object of~$\infAlg_{\Oca}(\Cca)$ by~\cref{prop:costab-infinite-suspension}. 
Therefore, we expect that the costabilisation of~$\infAlg_{\Oca}(\Cca)$ is non-trivial. 
In the following we are going to present some examples of this~kind.

\begin{proposition}[Gaitsgory--Rozenblyum]\label{prop:chain-loop-triv}
  Let~$\Oca$ be a reduced~$\infty$\nobreakdash-operad with values in the~$\infty$\nobreakdash-category~$\infSp_{\QQ}$ of rational spectra.
  The loop functor~${\Loop_{\Oca} \colon \infAlg_{\Oca}(\infSp_{\QQ}) \to \infAlg_{\Oca}(\infSp_{\QQ})}$ admits the following factorisation
  \begin{equation}\label{eq:chain-loop-triv}
    \begin{tikzcd}
      \infAlg_{\Oca}\left(\infSp_{\QQ}\right) \arrow[r, "\Loop_{\Oca}"] \arrow[d, "\frgt_{\Oca}"'] &\infAlg_{\Oca}\left(\infSp_{\QQ}\right) \\
      \infSp_{\QQ} \arrow[r, "\Loop_{\infSp_{\QQ}}"'] &\infSp_{\QQ} \arrow[u, "\triv_{\Oca}"']
    \end{tikzcd}
  \end{equation}
\end{proposition}

\begin{proof}
  See~\cite[Chapter~6, Proposition~1.7.2]{GR17}.
\end{proof}

\begin{corollary}\label{prop:chain-susp-free}
  Let~$\Oca$ be a reduced~$\infty$\nobreakdash-operad with values in the~$\infty$\nobreakdash-category~$\infSp_{\QQ}$. 
  The suspension functor~$\susp_{\Oca} \colon \infAlg_{\Oca}(\infSp_{\QQ}) \to \infAlg_{\Oca}(\infSp_{\QQ})$ admits a factorisation 
  \[
    \begin{tikzcd}
      \infAlg_{\Oca}\left(\infSp_{\QQ}\right) \arrow[r, "\susp_{\Oca}"] \arrow[d, "\indc_{\Oca}"'] &\infAlg_{\Oca}\left(\infSp_{\QQ}\right) \\
      \infSp_{\QQ} \arrow[r, "\susp_{\infSp_{\QQ}}"'] &\infSp_{\QQ}. \arrow[u, "\free_{\Oca}"']
    \end{tikzcd}
  \]
\end{corollary}

\begin{proof}
  Taking the left adjoints to the functors in~\eqref{eq:chain-loop-triv} gives the desired result.
\end{proof}

\begin{corollary}\label{cor:costab-rational-chain-O-algebra}
  Let~$\Oca$ be a reduced~$\infty$\nobreakdash-operad with values in~$\infSp_{\QQ}$.
  The~$\infty$\nobreakdash-cate\-gory~$\infSp_{\QQ}$ together with the functor~$\free_{\Oca} \colon \infSp_{\QQ} \to \infAlg_{\Oca}(\Cca)$ is equivalent to the costabilisation of the~$\infty$\nobreakdash-category~$\infAlg_{\Oca}(\Cca)$.
\end{corollary}

\begin{proof} 
  Abbreviate the functor~$\susp_{\infSp_{\QQ}}$ by~$\susp$.
  By~\cref{prop:chain-susp-free} we obtain the following commutative~diagram
  \[
  \begin{tikzcd}[column sep = huge, row sep = large]
  \cdots \arrow[r, "\susp_{\Oca}"] \arrow[rd, "\susp \circ \indc_{\Oca}"]  & \infAlg_{\Oca}\left(\infSp_{\QQ}\right) \arrow[r, "\susp_{\Oca}"] \arrow[rd, "\susp \circ \indc_{\Oca}"] & \infAlg_{\Oca}\left(\infSp_{\QQ}\right) \arrow[r, "\susp_{\Oca}"] \arrow[rd, "\susp \circ \indc_{\Oca}"] & \infAlg_{\Oca}\left(\infSp_{\QQ}\right)          \\
  \cdots \arrow[u, "\free_{\Oca}"'] \arrow[r, "\susp"'] & \infSp_{\QQ} \arrow[u, "\free_{\Oca}"'] \arrow[r, "\susp"']  & \infSp_{\QQ} \arrow[u, "\free_{\Oca}"'] \arrow[r, "\susp"']  & \infSp_{\QQ} \arrow[u, "\free_{\Oca}"']
  \end{tikzcd}
  \]
  in~$\infCAT_{\infty}$.
  Thus the functor 
  \[
  \infcoSp(\free_{\infopLie}) \colon \infSp_{\QQ} \simeq \infcoSp\left(\infSp_{\QQ}\right) \xrightarrow{\sim} \infcoSp\left(\infAlg_{\Oca}\left(\infSp_{\QQ}\right)\right),
  \]
  induced by~$\free_{\Oca}$, from the inverse limit of the lower row to that of the upper row is an equivalence of~$\infty$\nobreakdash-categories.
  In other words, we have a commutative~diagram
  \[
  \begin{tikzcd}[row sep = large]
    \infcoSp\left(\infAlg_{\Oca}\left(\infSp_{\QQ}\right)\right) \arrow[r, "\susp_{\infty}"] & \infAlg_{\Oca}\left(\infSp_{\QQ}\right)  \\
    \infcoSp\left(\infSp_{\QQ}\right) \arrow[r, "\susp_{\infty}"', "\sim"] \arrow[u, "\infcoSp(\free_{\infopLie})", "\simeq"'] & \infSp_{\QQ}  \arrow[u, "\free_{\Oca}"']
  \end{tikzcd}
  \]
  of~$\infty$\nobreakdash-categories.
  From the above diagram we see that~$\free_{\Oca}$ is equivalent to the lower horizontal arrow~$\susp_{\infty}$. 
\end{proof}

In the next section we show that the costabilisation of the~$\infty$\nobreakdash-category~$\infAlg_{\infopLie}(\infSp_{\Tel(h)})$ is equivalent to~$\infSp_{\Tel(h)}$, for every~$h \geq 1$.
In particular, this implies that the costabilisation of the~$\infty$\nobreakdash-category~$\infHType_{v_{h}}$ of~$v_{h}$\nobreakdash-periodic \htypes is non-trivial, in contrast to the costabilisation of~$\infHType_{\ast}$, \cf~\cref{ex:costab-triv-ex}.

\begin{theorem}[Heuts--Land]\label{thm:En-susp-n-free}
  Let~$n \in \NN$ and let~$\Cca$ be a presentable stable symmetric monoidal~$\infty$\nobreakdash-category.
  Then the~$n$\nobreakdash-fold suspension functor~$\susp^n_{\infopE_n} \colon \infAlg_{\infopE_n}^{\nut}(\Cca) \to \infAlg_{\infopE_n}^{\nut}(\Cca)$ admits the following factorisation
  \[
  \begin{tikzcd}[row sep = large]
  \infAlg_{\infopE_n}^{\nut}\left(\Cca\right) \arrow[r, "\susp_{\infopE_n}^n"] \arrow[d, "\indc_{\infopE_n^{\nut}}"'] &\infAlg_{\infopE_n}^{\nut}\left(\Cca\right) \\
  \Cca \arrow[r, "\susp_{\Cca}^{n}"'] &\Cca. \arrow[u, "\free_{\infopE_n^{\nut}}"'] 
  \end{tikzcd}
  \]
\end{theorem}

\begin{proof}
See~\cite[Corollary~2.13]{HLEn}. 
\end{proof}

\begin{corollary}[Heuts--Land]\label{cor:costab-En-general}
  In the situation of~\cref{thm:En-susp-n-free} the~$\infty$\nobreakdash-category~$\Cca$ together with the functor~$\free_{\infopE_n^{\nut}} \colon \Cca \to \infAlg_{\infopE_n}(\Cca)$ is equivalent to the costabilisation of the~$\infty$\nobreakdash-category~$\infAlg_{\infopE_n}^{\nut}(\Cca)$.
\end{corollary}

\begin{proof}
  See~\cite[Corollary~2.16]{HLEn}.
  We can use the same proof strategy as the proof of~\cref{cor:costab-rational-chain-O-algebra}, where we replace the 1-fold suspension functor by the~$n$\nobreakdash-fold suspension functor. 
  In particular, we obtain the following commutative diagram
  \[
    \begin{tikzcd}[row sep = large]
        \infcoSp\left(\infAlg_{\infopE_n}^{\nut}(\Cca)\right) \arrow[r, "\susp_{\infty}"] &\infAlg_{\infopE_n}^{\nut}(\Cca) \\
        \infcoSp(\Cca) \arrow[r, "\susp_{\infty}"', "\sim"] \arrow[u, "\simeq"', "\infcoSp(\free_{\infopE_n^{\nut}})"] &\Cca \arrow[u, "\free_{\infopE_n^{\nut}}"']     
    \end{tikzcd}
  \]
  of~$\infty$\nobreakdash-categories, which gives an equivalence~$\free_{\infopE_n^{\nut}} \simeq \susp_{\infty}^{\infAlg_{\infopE_n}^{\nut}(\Cca)}$ of functors. \qedhere
\end{proof}

\begin{passage}
Let~$\Oca$ be a reduced~$\infty$\nobreakdash-operad with values in a presentable stable symmetric monoidal~$\infty$\nobreakdash-category~$\Cca$.
In~\cite[Theorem~4.3]{Heu20} it is shown that the stabilisation of the~$\infty$\nobreakdash-category~$\infAlg_{\Oca}(\Cca)$ is equivalent to~$\Cca$.
We cannot simply dualise the argument in \loccit to give a general statement about the costabilisation of~$\infAlg_{\Oca}(\Cca)$.
For simplicity, let~$\Oca$ be an~$\infty$\nobreakdash-operad with values in~$\infHType$.
Almost equivalently, we can consider the problem of determining the costabilisation of~$\infAlg_{\Oca}(\Cca)$ as the problem of determining the stabilisation of the~$\infty$\nobreakdash-category~${\infcoAlg_{\Oca}(\Cca) \coloneqq \infAlg_{\Oca / \infopCom}(\Cca^{\op})^{\op}}$ of~$\Oca$-coalgebras.
In the situation of~\cite[Theorem~4.3]{Heu20} the proof uses two~ingredients:
\begin{enumerate}
  \item The adjunction induced by the stabilisation of a monadic adjunction is monadic, see~\cite[Example~4.7.3.10]{HA}. 
  \item The linear approximation (in the context of Goodwillie calculus) of the functor~$\frgt_{\Oca} \circ \free_{\Oca}$ satisfies 
    \[
    \Poly_{1}\left(\coprod_{r \geq 0} \Oca(r) \otimes_{\Perm_r} X^{\otimes r}\right) \simeq \Oca(1) \otimes X.
    \]
\end{enumerate}
Neither of the two facts generalises to the dual/opposite setting in general. 
In particular, we would have to consider the Goodwillie calculus tower for the divided power functor~$\left(\Oca(r) \otimes (\blank)^{\otimes r}\right)^{\Perm_r}$, which is non-trivial, in contrary to the Goodwillie calculus tower of~$\left(\Oca(r) \otimes (\blank)^{\otimes r}\right)_{\Perm_r}$ (this functor is~$r$\nobreakdash-homogeneous).

However, if we work with the~$\infty$\nobreakdash-category~$\infSp_{\Tel(h)}$ of~$\Tel(h)$\nobreakdash-local spectra, recall that we do have similar formula of the co-linear approximation~$\Poly^{1}(F)$ of certain endofunctors of~$\infSp_{\Tel(h)}$, see~\cref{lem:Th-dual-calculus-sum}.
This is exactly one of the important ingredients we use to investigate the costabilisation of~$\infAlg_{\infopLie}(\infSp_{\Tel(h)})$ in the next~section.
\end{passage}

\subsection{Costabilisation of~\texorpdfstring{$v_h$}{vₕ}-periodic homotopy types}\label{sec:costab-vh}

Let~$h \geq 1$ be a natural number.
In this section we prove that the costabilisation of~$\infAlg_{\infopLie}(\infSp_{\Tel(h)})$ is equivalent to~$\infSp_{\Tel(h)}$, see~\cref{thm:costab-Lie-Tn}.
Under the equivalence~$\infHType_{v_{h}} \simeq \infAlg_{\infopLie}(\infSp_{\Tel(h)})$~(see~\cref{thm:Heuts-vh-Lie}), the~$\infty$\nobreakdash-category~$\infSp_{\Tel(h)}$ is also equivalent to the costabilisation of~$\infHType_{v_{h}}$.
As a corollary we provide a universal property of the Bousfield--Kuhn functor, see~\cref{cor:universal-property-BK}.

\begin{situation}\label{sit:free-frgt-Th}
  Let~$\Oca$ be a reduced~$\infty$\nobreakdash-operad with values in the~$\infty$\nobreakdash-category~$\infSp_{\Tel(h)}$. 
  Recall the~adjunction
  \[
     \free_{\Oca} \colon \infSp_{\Tel(h)} \rightleftarrows \infAlg_{\Oca}\left(\infSp_{\Tel(h)}\right) \cocolon \frgt_{\Oca}.
  \]
  Since the functor~$\free_{\Oca}$ preserves small colimits, the adjunction induces an adjunction
  \[
    \sfree_{\Oca} \colon \infSp_{\Tel(h)} \rightleftarrows \infcoSp\left(\infAlg_{\Oca}\left(\infSp_{\Tel(h)}\right)\right) \cocolon \GG_{\Oca},
  \]
  on costabilisations, by~\cref{lem:ff-costab}.
\end{situation}

\begin{proposition}
  In~\cref{sit:free-frgt-Th} the functor~$\sfree_{\Oca}$ is fully faithful. 
\end{proposition}

\begin{proof}
  By~\cref{prop:ff-left-adjoint} it suffices to show that there is an equivalence~${\GG_{\Oca} \circ \sfree_{\Oca} \simeq \id}$.
  Denote the suspension and the loop functor on~$\infSp_{\Tel(h)}$ by~$\susp$ and~$\Loop$, respectively.
  
  For every object~$E \in \infSp_{\Tel(h)}$, we have 
  \[
    \begin{split}
       (\GG_{\Oca} \circ \sfree_{\Oca})(E) &\simeq \varprojlim_{n}\left(\susp^{n} \circ \frgt_{\Oca} \circ \free_{\Oca} \circ \Loop^n \right) (E) \\
       &\simeq \Poly^{1}\left(\coprod_{i \geq 1}\left(\Oca(n)\otimes E^{\otimes i}\right)_{\Perm_i}\right) \\
       &\simeq E.
    \end{split}
  \]
  The first equivalence holds by the construction of~$\sfree_{\Oca}$ and~$\GG_{\Oca}$. 
  The second equivalence holds by the construction of the colinear approximation, in the sense of dual Goodwillie calculus, of the endofunctors~$\frgt_{\Oca} \circ \free_{\Oca}$ of~$\infSp_{\Tel(h)}$.
  The last equivalence holds by~\cref{lem:Th-dual-calculus-sum}.
\end{proof}

\begin{proposition}\label{prop:sfree-Lie-equi}
  The functor~$\sfree_{\infopLie}$ induced by~$\free_{\infopLie} \colon \infSp_{\Tel(h)} \to \infAlg_{\infopLie}(\infSp_{\Tel(h)})$ is an equivalence of stable~$\infty$\nobreakdash-categories.
\end{proposition}

\begin{proof}
  Recall the fully faithful functor~(see~\cref{thm:U-infty-ff})
  \[
  \UE_{\infty} \colon \infAlg_{\infopLie}\left(\infSp_{\Tel(h)}\right) \to \varprojlim \infAlg_{\infopE_n}^{\nut}\left(\infSp_{\Tel(h)}\right)
  \]
  induced by the commutative diagram~\eqref{diag:Lie-En-Th}.
  Let 
  \[
  \pr_0 \colon \varprojlim_{n} \infAlg_{\infopE_n}^{\nut}\left(\infSp_{\Tel(h)}\right) \to  \infAlg_{\infopE_0}^{\nut}\left(\infSp_{\Tel(h)}\right) \simeq \infSp_{\Tel(h)}
  \]
  denote the canonical projection map in the inverse limit~diagram.
  Note that we~have
  \[
  \indc_{\infopLie} \simeq \UE_{0} = \pr_0 \circ \UE_{\infty}.
  \]
  Since~$\indc_{\infopLie} \circ \free_{\infopLie} \simeq \id$~(see~\cref{para:vh-stabilisation}), the following~composition
  \begin{equation}\label{eq:Lie-id-UE}
    \infSp_{\Tel(h)} \xrightarrow{\free_{\infopLie}} \infAlg_{\infopLie}\left(\infSp_{\Tel(h)}\right) \xrightarrow{\UE_{\infty}} \varprojlim_{n} \infAlg_{\infopE_n}\left(\infSp_{\Tel(h)}\right) \xrightarrow{\pr_0} \infSp_{\Tel(h)}
  \end{equation}
   of small-colimit-preserving functors is equivalent to identity functor on~$\infSp_{\Tel(h)}$.
   Thus the following composition of functors
   \[
   \infSp_{\Tel(h)} \xrightarrow{\sfree_{\infopLie}} \infcoSp\left(\infAlg_{\infopLie}\left(\infSp_{\Tel(h)}\right)\right) \xrightarrow{\infcoSp(\UE_{\infty})}  \infcoSp\left(\varprojlim_{n} \infAlg_{\infopE_n}^{\nut}\left(\infSp_{\Tel(h)}\right)\right) \xrightarrow{\infcoSp(\pr_{0})} \infSp_{\Tel(h)}
   \]
   on costabilisations, induced by~\eqref{eq:Lie-id-UE}, is equivalent to the identity functor on~$\infSp_{\Tel(h)}$ as well. 
   Therefore, to show that~$\sfree_{\infopLie}$ is an equivalence of~$\infty$\nobreakdash-categories, it suffices to show that~$\infcoSp(\pr_{0})$ is an equivalence, since the functor~$\infcoSp(\UE_{\infty})$ is fully faithful by~\cref{lem:ff-costab}.
   
   \begin{claim}
     The functor~$\infcoSp(\pr_{0})$ is an equivalence of~$\infty$\nobreakdash-categories.
   \end{claim}
   Consider the following commutative diagram
   \[
   \begin{tikzcd}
     \cdots \arrow[r] & \infAlg_{\infopE_{n+1}}^{\nut}\left(\infSp_{\Tel(h)}\right) \arrow[r, "\rB_{n+1}"] & \infAlg_{\infopE_n}^{\nut}\left(\infSp_{\Tel(h)}\right) \arrow[r, "\rB_{n}"] & \cdots \arrow[r, "\rB_1"] & \infAlg_{\infopE_0}^{\nut}\left(\infSp_{\Tel(h)}\right) \\
     \cdots \arrow[r, "\susp"] \arrow[u] & \infSp_{\Tel(h)} \arrow[r, "\susp"] \arrow[u, "\free_{\infopE_{n+1}}^{\nut}"] & \infSp_{\Tel(h)} \arrow[r, "\susp"] \arrow[u, "\free_{\infopE_n}^{\nut}"] & \cdots \arrow[r, "\susp"] \arrow[u] & \infSp_{\Tel(h)} \arrow[u, "\id", "\simeq"']
   \end{tikzcd}
   \]
   in~$\infPrl$ obtained from~\cref{cor:Bn-free-commutes}.
   It induces the following commutative diagram
   \[
    \begin{tikzcd}
    \cdots \arrow[r]           & \infcoSp\left(\infAlg_{\infopE_n}^{\nut}\left(\infSp_{\Tel(h)}\right)\right) \arrow[r]           & \cdots \arrow[r]           &  \infcoSp\left(\infAlg_{\infopE_0}^{\nut}\left(\infSp_{\Tel(h)}\right)\right)          \\
    \cdots \arrow[r, "\sim"] \arrow[u, "\simeq"] & \infSp_{\Tel(h)} \arrow[r, "\sim", "\susp"'] \arrow[u, "\sfree_{\infopE_n}", "\simeq"'] & \cdots \arrow[u, "\simeq"'] \arrow[r, "\sim", "\susp"'] & \infSp_{\Tel(h)} \arrow[u, "\sfree_{\infopE_0}", "\simeq"']
    \end{tikzcd}
   \]
   on costabilisations, where the vertical arrows are equivalences by~\cref{cor:costab-En-general}.
   Then the claim follows by the induced map between the colimits of the rows and from the commutativity of the second~diagram.
\end{proof}

\begin{theorem}\label{thm:costab-Lie-Tn}
  The stable~$\infty$\nobreakdash-category~$\infSp_{\Tel(h)}$ together with the functor~$\free_{\infopLie}$ is equivalent to the costabilisation of the~$\infty$\nobreakdash-category~$\infAlg_{\infopLie}(\infSp_{\Tel(h)})$.
\end{theorem}

\begin{proof}
  The functor~$\free_{\infopLie}$ induces a commutative diagram
  \[
    \begin{tikzcd}
      \infSp_{\Tel(h)} \arrow[r, "\susp_{\infty}", "\sim"'] \arrow[d, "\sfree_{\infopLie}"', "\simeq"] & \infSp_{\Tel(h)} \arrow[d, "\free_{\infopLie}"]\\
      \infcoSp\left(\infAlg_{\infopLie}\left(\infSp_{\Tel(h)}\right)\right) \arrow[r, "\susp_{\infty}"'] & \infAlg_{\infopLie}\left(\infSp_{\Tel(h)}\right) 
    \end{tikzcd}
  \]
  of~$\infty$\nobreakdash-categories; here the left vertical arrow is an equivalence by~\cref{prop:sfree-Lie-equi} and the upper horizontal arrow is an equivalence because the~$\infty$\nobreakdash-category~$\infSp_{\Tel(h)}$ is~stable. 
  This commutative diagram exhibits an equivalence between the functor~$\free_{\infopLie}$ and the canonical functor
  $
  \susp_{\infty} \colon \infcoSp\left(\infAlg_{\infopLie}(\infSp_{\Tel(h)})\right) \to \infAlg_{\infopLie}(\infSp_{\Tel(h)}).
  $
\end{proof}

Recall from~\cref{para:vh-stabilisation} that the free--forgetful adjunction of~$\Tel(h)$\nobreakdash-local spectral Lie algebras corresponds to the Bousfield--Kuhn adjunction
$
\rBK_{h} \colon \infSp_{\Tel(h)} \rightleftarrows \infHType_{v_{h}} \cocolon \BK_{h},
$
under the equivalence~$\infAlg_{\infopLie}(\infSp_{\Tel(h)}) \simeq \infHType_{v_{h}}$.
Therefore, \cref{thm:costab-Lie-Tn} implies that the functor~$\lBK_{h}$ is equivalent to the costabilisation of~$\infHType_{v_h}$.
This gives a universal property of~$\rBK_{h}$ and~$\BK_{h}$ and thus proves~\cref{thm:universal-property-BK}. 
We record the content of~\cref{thm:universal-property-BK} here for convenience of the reader.

\begin{corollary}[{\cref{thm:universal-property-BK}}]\label{cor:universal-property-BK}
    \begin{enumerate}
      \item Let~$\Cca$ be a stable~$\infty$\nobreakdash-category.
      Composing with the functor~$\lBK_{h}$ induces an~equivalence
      \[
      \infFun^{\rex}(\Cca, \infSp_{\Tel(h)}) \xrightarrow{\sim} \infFun^{\rex}\left(\Cca, \infHType_{v_{h}}\right),
      \]
      of~$\infty$\nobreakdash-categories, where~$\infFun^{\rex}$ denotes the~$\infty$\nobreakdash-category of right exact functors, \ie functors that preserve finite colimits.
      \item Let~$\Dca$ be a presentable stable~$\infty$\nobreakdash-category.
      Composing with the Bousfield--Kuhn functor~$\BK_{h}$ induces an equivalence
       \[
       \infFun^{R}(\infSp_{\Tel(h)}, \Dca) \xrightarrow{\sim} \infFun^{R}\left(\infHType_{v_{h}}, \Dca\right)
       \]
       of~$\infty$\nobreakdash-categories, where~$\infFun^{R}$ denotes the~$\infty$\nobreakdash-category of functors that are accessible and  preserves small limits, \ie functors admitting left~adjoints. 
    \end{enumerate}  
\end{corollary}

\begin{proof}
  This is a consequence of~\cref{thm:costab-Lie-Tn},~\cref{prop:costab-uni-prop} and~\cref{cor:univ-right-adj-costab}.
\end{proof}

We would like to conclude this article by the following potential direction of future research:~\emph{Costabilisation of truncated chromatic layers}.
Recall the notations from~\cref{sit:monochrom-construction}.
Let~$(V_{h})_{h \geq 1}$ be a sequence of pointed~$p$\nobreakdash-local finite complexes such that
\begin{enumerate}
  \item $V_h$ admits a desuspension for every~$h \geq 1$,
  \item $V_h$ is of type~$h$, and
  \item $\conn(V_{h+1}) > \conn(V_h)$ for every~$h \geq 1$.
\end{enumerate}
Then by the Unstable Class Invariance Theorem~\cite[Theorem~9.15]{Bou94} we obtain a tower of natural transformations
\begin{equation}
  \cdots \to \Null_{V_{h+1}} \to \Null_{V_{h}} \to \cdots \to \Null_{V_1}
\end{equation}
We consider this tower as the unstable analogue of the finite stable chromatic localisation tower~\eqref{eq:finite-localisation-tower}, because of~\cref{thm:Vh-truncatio-vh-periodic-equivalence}.
Denote by~$\infHType_{(p), V_{h}}^{> 1}$ the full~$\infty$\nobreakdash-subcategory of~$\infHType_{(p)}^{> 1}$ whose objects are~$V_{h}$\nobreakdash-\less \htypes, \ie the essential image of~$\Null_{V_{h}}$.

We consider \cref{thm:costab-Lie-Tn} as our first step towards understanding chromatic assembly of \htypes.
For a natural number~$h \geq 2$, the finite truncation~$\infHType_{(p), V_{h}}^{> 1}$ of this tower also admits a non-trivial costabilisation; for~$h = 1$ the costabilisation of~$\infHType_{(p), V_{1}}^{> 1}$~(for example~${V_{1} = \usphere^2/p}$) of pointed simply-connected rational \htypes is trivial.
We would like to pose the following~questions.

\begin{question}
  \begin{enumerate}
    \item What is the costabilisation of the~$\infty$\nobreakdash-category~$\infHType_{(p), \widetilde{V}_{h}}^{> 1}$?
    \item What is the relationship between the Bousfield--Kuhn functor~$\BK_{h}$ and the right adjoint~$\BK_{\leq h}$ to the natural functor 
      $
      \susp_{\infty} \colon \infcoSp(\infHType_{(p), \widetilde{V}_{h}}^{> 1}) \to \infHType_{(p), \widetilde{V}_{h}}^{> 1}?
      $
  This is related to Heuts's question~\cite[Question~C.4]{Heu20}:
    \item Let~$V$ be a pointed finite~$p$\nobreakdash-local \htype.
        If the~$v_{h}$\nobreakdash-periodic homotopy groups of~$V$ vanish, are the~$v_{n}$\nobreakdash-periodic homotopy groups of~$V$ also trivial for every~$0 \leq n \leq h$?
  \end{enumerate}
\end{question}


\printbibliography

\end{document}